\documentclass[a4paper,10p]{amsart}
\usepackage{amsmath,amsthm,latexsym,amscd,amssymb}
\usepackage[all]{xy}
\numberwithin{equation}{section}

\newtheorem{prop}{Proposition}[section]
\newtheorem{lem}[prop]{Lemma}
\newtheorem{ddd}[prop]{Definition}
\newtheorem{theorem}[prop]{Theorem}
\newtheorem{cor}[prop]{Corollary}

\newcommand{\ch}{\mathop{\mbox{\rm ch}}\nolimits}
\newcommand{\End}{{\rm End}}
\newcommand{\Det}{{\rm det}}
\newcommand{\Hom}{\mathop{\mbox{\rm Hom}}}
\newcommand{\vol}{\mathop{\mbox{\rm vol}}}
\newcommand{\ind}{\mathop{\mbox{\rm ind}}}

\newcommand{\dom}{\mathop{\rm dom}}
\newcommand{\spc}{\mathbf A}
\newcommand{\spcb}{\mathbf B}
\newcommand{\ad}{\mathfrak A}

\newcommand{\Pj}{{\mathcal P}}

\newcommand{\tr}{\mathop{\rm tr}\nolimits}

\newcommand{\Tr}{\mathop{\rm Tr}\nolimits}

\newcommand{\N}{{\mathcal N}}

\newcommand{\Ki}{{\mathfrak K}}
\newcommand{\Kappa}{{\mathcal K}}
\newcommand{\E}{{\mathcal E}}
\newcommand{\ideal}{{\mathcal I}}
\newcommand{\F}{{\mathcal F}}

\newcommand{\D}{{\mathcal D}}
\newcommand{\Li}{{\mathcal L}}
\newcommand{\Ca}{{\mathcal C}}
\newcommand{\A}{{\mathcal A}}
\newcommand{\B}{{\mathcal B}}
\newcommand{\Bi}{{\mathcal B}_{\infty}}
\newcommand{\Ai}{{\mathcal A}_{\infty}}

\newcommand{\C}{C^{\infty}}
\newcommand{\Ol}[1]{\hat \Omega_{\le {#1}}}
\newcommand{\Ok}{\hat \Omega_k}
\newcommand{\Oei}{\Ol{\mu}{\mathcal E}_i}
\newcommand{\Ei}{{\mathcal E}_{\infty}}
\newcommand{\Oi}{\hat \Omega_*}

\newcommand{\ra}{\partial}

\newcommand{\ten}{\otimes}

\newcommand{\pl}[1]{\varprojlim\limits_{#1}}

\newcommand{\ve}{\varepsilon}
\newcommand{\ov}{\overline}

\newcommand{\dira}{\partial\!\!\!/}
\newcommand{\dirac}{\partial\!\!\!/}

\DeclareMathOperator{\re}{Re}
\DeclareMathOperator{\supp}{supp}
\DeclareMathOperator{\Ran}{Ran}
\DeclareMathOperator{\Ker}{Ker}
\DeclareMathOperator{\di}{d}

\def\bbbr{{\rm I\!R}} 
\def\bbbn{{\rm I\!N}} 

\def\bbbc{{\rm I\!C}}

\def\bbbq{{\mathchoice {\setbox0=\hbox{$\displaystyle\rm Q$}\hbox{\raise
0.15\ht0\hbox to0pt{\kern0.4\wd0\vrule height0.8\ht0\hss}\box0}}
{\setbox0=\hbox{$\textstyle\rm Q$}\hbox{\raise
0.15\ht0\hbox to0pt{\kern0.4\wd0\vrule height0.8\ht0\hss}\box0}}
{\setbox0=\hbox{$\scriptstyle\rm Q$}\hbox{\raise
0.15\ht0\hbox to0pt{\kern0.4\wd0\vrule height0.7\ht0\hss}\box0}}
{\setbox0=\hbox{$\scriptscriptstyle\rm Q$}\hbox{\raise
0.15\ht0\hbox to0pt{\kern0.4\wd0\vrule height0.7\ht0\hss}\box0}}}}

\def\bbbz{{\mathchoice {\hbox{$\sf\textstyle Z\kern-0.4em Z$}}
{\hbox{$\sf\textstyle Z\kern-0.4em Z$}}
{\hbox{$\sf\scriptstyle Z\kern-0.3em Z$}}
{\hbox{$\sf\scriptscriptstyle Z\kern-0.2em Z$}}}}

\def\bbbc{{\mathchoice {\setbox0=\hbox{$\displaystyle\rm C$}\hbox{\hbox
to0pt{\kern0.4\wd0\vrule height0.9\ht0\hss}\box0}}
{\setbox0=\hbox{$\textstyle\rm C$}\hbox{\hbox
to0pt{\kern0.4\wd0\vrule height0.9\ht0\hss}\box0}}
{\setbox0=\hbox{$\scriptstyle\rm C$}\hbox{\hbox
to0pt{\kern0.4\wd0\vrule height0.9\ht0\hss}\box0}}
{\setbox0=\hbox{$\scriptscriptstyle\rm C$}\hbox{\hbox
to0pt{\kern0.4\wd0\vrule height0.9\ht0\hss}\box0}}}}

\title{The Atiyah--Patodi--Singer index theorem for Dirac operators over $C^*$-algebras}
\author{Charlotte Wahl}

\begin{document}
\begin{abstract}
We prove a higher Atiyah--Patodi--Singer index theorem for Dirac operators twisted by $C^*$-vector bundles. We use it to derive a general product formula for $\eta$-forms and to define and study new $\rho$-invariants generalizing Lott's higher $\rho$-form. The higher Atiyah--Patodi--Singer index theorem of Leichtnam--Piazza can be recovered by applying the theorem to Dirac operators twisted by the Mishenko--Fomenko bundle associated to the reduced $C^*$-algebra of the fundamental group.  
\end{abstract}
\maketitle
\markright{{\sc THE ATIYAH--PATODI--SINGER INDEX THEOREM OVER $C^*$-ALGEBRAS}}

\section*{Introduction}

In noncommutative geometry a compact space $X$ is generalized by a unital $C^*$-algebra $\A$. By applying a noncommutative concept to the commutative $C^*$-algebra $\A=C(X)$ one recovers its classical counterpart. An $\A$-vector bundle on a Riemannian manifold $M$ is a locally trivial bundle of projective $\A$-modules. Its classical counterpart is a (complex) vector bundle on $M\times X$. Thus the index theory of a Dirac operator on $M$ twisted by an $\A$-vector bundle is a variant of family index theory, where the base space, encoded by $\A$, is noncommutative. If $\A=C(X)$, one obtains a Dirac operator on $M$ twisted by a vector bundle on $M\times X$, which we can consider as a vertical operator on the fiber bundle $M \times X \to X$. In the realm of family index theory this situation is particularly simple since the fiber bundle is trivial and the metric on $M$ does not depend on the parameter.

Modelled on the family case, the superconnection formalism has been applied Dirac operators over $C^*$-algebras in \cite{lo1}\cite{lo2}. In the classical case, for the construction of a Bismut superconnection one needs that $X$ is a smooth manifold. In the noncommutative case this is encoded in the choice of a projective system of Banach algebras $(\A_i)_{i \in \bbbn_0}$ with $\A_0=\A$, and with injective structure maps whose images are closed under holomorphic functional calculus. By using the de Rham homology for algebras and the Chern character as defined in \cite{kar} from the $K$-theory of $\A$ to the de Rham homology of the projective limit $\Ai$ (which is assumed to be dense in $\A$)  one can formulate for closed $M$ an index theorem for Dirac operators twisted by $\A$-vector bundles in analogy to the Atiyah--Singer family index theorem. One gets numerical invariants from the index theorem by pairing the de Rham homology classes with reduced cyclic cocycles. 

By results of Lott \cite{lo2} the Atiyah--Singer index theorem for Dirac operators over $C^*$-algebras can be proven by adapting the superconnection proof of Bismut, as given in \cite{bgv}. While the construction of the heat semigroup is completely analogous to the construction in \cite{bgv}, the main difficulty lies in the study of the large time limit of the heat semigroup. In the classical case, where the Dirac operator is a selfadjoint operator on a Hilbert space of $L^2$-sections, the exponential decay of the heat semigroup on the complement of the kernel of the Dirac operator can be proven using the positivity of the square of the Dirac operator on the complement. Here however we deal with Banach spaces, namely subspaces of the $\A_i$-module $L^2(M,\A_i)^n$. The crucial ingredient for the proof of the large time asymptotics is a result in \cite[\S 6]{lo2}, which relates the spectrum of an  operator with integral kernel in $\C(M \times M,\A_i)$ on  $L^2(M,\A_i)$ to the spectrum of the corresponding operator on the Hilbert $\A$-module $L^2(M,\A)$. Then one uses that the spectrum of a holomorphic semigroup gives information on its large time behaviour. 

Using Lott's methods one can also define noncommutative $\eta$-forms and formulate an Atiyah--Patodi--Singer index theorem for Dirac operators over $C^*$-algebras. The proof of this theorem is the main result of the present paper.  (See Theorem \ref{indtheor} in the even case and Theorem \ref{indtheorodd} for the odd case. For the expressions used in its statement, we refer the reader to \S \ref{derham}, where noncommutative de Rham homology is explained as needed here, and \S \ref{condi}, where the geometric situation is introduced.) The proof is a generalization of the proof given in a  special situation in \cite{wa} and uses some of the concepts and results developed there.

One of the main motivations to study the index theory of Dirac operator twisted by $C^*$-vector bundles comes from higher index theory, pioneered by Connes--Moscovici with their higher index theorem  and relevant in connection with the Novikov conjecture \cite{cm}. 
Here the $\A=C_r^*\Gamma$ is the reduced group $C^*$-algebra of the fundamental group $\Gamma$ of $M$ and the $C_r^*\Gamma$-vector bundle is the Mishenko--Fomenko bundle $\tilde M \times_{\Gamma} C_r^*\Gamma$. An appropriate projective system of Banach algebras can be gained from the work of Connes--Moscovici \cite{cm}.  Eta-forms and the superconnection formalism in higher index theory were introduced  by Lott \cite{lo3}\cite{lo1}. Leichtnam--Piazza proved a higher Atiyah--Patodi--Singer index theorem \cite{lp1}\cite{lp2}\cite[\S 4]{lp3}. Many results from family and higher index theory carry directly over to our setting. So does part of the proof of the higher Atiyah--Patodi--Singer index theorem in \cite[\S 4]{lp3}. However the consideration of the large time limit there is based on the higher $b$-pseudodifferential calculus \cite[\S 12]{lp1}, which exploits the special situation given in higher index theory. 
The higher Atiyah--Patodi--Singer index theorem of Leichtnam--Piazza can be reproven by applying our main result to Dirac operators twisted by the Mishenko--Fomenko bundle. A nice introduction into the matter can be found in \cite{lp1}, for example.  

In the remainder of this introduction we discuss  our results in more detail and outline possible applications. Because of the length of the paper, the applications will not be treated in its body. 

The boundary conditions we consider are a generalization of those introduced by Melrose--Piazza \cite{mp1}\cite{mp2} for families and in higher index theory by Leichtnam--Piazza, see \cite{lp5}. They incorporate the product situations discussed in \cite{waprod}. The main application is an easy proof of a product formula for $\eta$-forms (see \S \ref{prodeta}), which generalizes and unifies several formulas proven before \cite[\S 2]{lpetpos} \cite[Theorem 6.1]{ps}\cite[Lemma 6]{mp2} (the latter when translated from family index theory to the present setting). It is used in the proof of the index theorem in the odd case. The approach should also work in family index theory. In our proof of the index theorem we follow a strategy of Bunke \cite{bk}\cite{bu}, which is inspired by, but more pedestrian than the  proof of Melrose--Piazza based on Melrose's $b$-calculus \cite{mp1}. It is also related to M\"uller's approach \cite{mu2}, but differs in the treatment of the large time limit. In particular we study an index problem for a manifold $M$ with cylindric ends. It is well-known that any  Atiyah--Patodi--Singer boundary value problem can be translated into such a problem (see \cite[Prop. 2.1]{waprod} for the result as needed here and references therein for corresponding earlier results in higher and family index theory). As in the closed case, the main additional difficulty arising in the noncommutative situation lies in the study of the behaviour of the heat semigroup for large times. One important step is an extension of Lott's method mentioned above, see \S \ref{parametrix} and \S \ref{regpar}. 

On the way we obtain a large time estimate (see \S \ref{liminf}) that seems to be new also in higher index theory (see \cite{mp1}\cite{bk} for the family case). It allows to define noncommutative $\eta$-forms for invertible Dirac operators on manifolds with cylindric ends in analogy to M\"uller's construction in \cite{mu1} and in the family case Bunke's definition in \cite{bu}. These are interesting in connection with the gluing problem for $\eta$-forms. 

As a special case of our main result one obtains an Atiyah--Patodi--Singer index theorem for Dirac operators twisted by the Mishenko--Fomenko bundle $\tilde M \times_{\Gamma} C^*\Gamma$ associated to the maximal group $C^*$-algebra $C^*\Gamma$. This is relevant in connection with $\rho$-invariants; see \cite{ps}, where the zero degree part of the theorem was proven. In \S \ref{rhoinvar} we define new $\rho$-invariants, motivated by Lott's higher $\rho$-form \cite{lo3} and the higher $\rho$-forms for the signature operator defined by Leichtnam--Piazza \cite{lpsign}, and study their properties.

In higher index theory our framework may be applied to the study of Lott's delocalized $L^2$-invariants \cite{lodeloc}, which can be defined in full generality by using the new projective systems of Banach algebras $(\A_i)_{i \in \bbbn_0}$ recently introduced by Puschnigg \cite{p}. (In the approach of Leichtnam--Piazza the algebra $\Ai$ seems to be implicitely fixed by compatibility with \cite[Def. 12.1(ii)]{lp1}, so their results do not seem to generalize directly to Puschnigg's systems.)

Possible applications of our results also arise from twisted higher index theory. A higher twisted Atiyah--Singer index theorem has been proven by Marcolli--Mathai \cite{mm}. In this situation the $C^*$-vector bundle playing the role of the Mishenko--Fomenko bundle is in general not flat.

Furthermore, our theorem can be applied to flat foliated bundles: Let $\Gamma$ act by diffeomorphisms on a closed manifold $T$ and consider the reduced crossed product $C(T) \times_r \Gamma$. An appropriate projective system of subalgebras has been defined in \cite[\S 4.1]{wa2}. Any invariant vector bundle on $\tilde M \times T$ descends to a $C(T) \times_r \Gamma$-vector bundle on $M$. Thus we get an index theorem for a Dirac operator on $M$ twisted by such a bundle. In contrast to the Mishenko--Fomenko bundle this $C^*$-vector bundle is in general not flat. The resulting theorem should be compared with the Atiyah--Patodi--Singer index theorem for \'etale groupoids by Leichtnam--Piazza \cite{lpetgroup}, which applies to more general foliated bundles under some additional assumptions, in particular that $\Gamma$ is virtually nilpotent and the boundary operator invertible. It would be interesting to find a strategy for the Atiyah--Patodi--Singer index theory for foliated bundles which at the same time overcomes the restrictions of the present method (the special geometric situation) and those of \cite{lpetgroup} (in particular the condition on the group).

\subsection*{Conventions}

All tensor products are graded and completed. When dealing with Hilbert $C^*$- modules they are tensor products of Hilbert $C^*$-modules, else they are projective tensor products (sometimes for clarity denoted by $\ten_{\pi}$). Commutators are graded. We use the following convention concerning $L^2$-spaces of Banach algebra valued functions: We use the Hilbert $C^*$-module completion if the Banach algebra is a $C^*$-algebra. The completion is given by the norm $\| f \|=| \int f(x)^*f(x) ~dx |^{1/2}$. (Here and in general in this paper $|\cdot |$ also denotes a norm.) If the algebra is not a $C^*$-algebra, the completion is with respect to the norm $\| f \|=(\int |f(x)|^2 ~dx)^{1/2}$. 
An analogous convention holds for $L^2$-spaces of sections of bundles whose fibres are projective modules over a Banach resp. $C^*$-algebra. 

In general, notation and conventions are as in \cite{wa}. 
We refer to \cite[Ch. 5]{wa} for the notions of Hilbert-Schmidt and trace class operators for $L^2$-spaces of Banach algebra valued functions and results guaranteeing that for our purposes they behave as in the classical case. 

\tableofcontents

\section{Differential algebras, connections and Chern character}
\label{derham}

\subsection{Differential algebras and de Rham homology} 
Let $\B$ be a involutive unital locally $m$-convex Fr\'echet algebra \cite{ma} which is in addition a local Banach algebra \cite{bl}. The $\bbbz$-graded universal differential algebra is defined as $\Oi\B= \prod_{k=0}^{\infty} \Ok \B$ with $\Ok \B := \B (\ten \B/\bbbc)^k, ~k \in \bbbn_0$ with differential $\di$ of degree one given by 
$\di(b_0 \ten b_1 \dots \ten b_k)=1 \ten b_0 \ten \dots b_k$ and linear extension. The product is determined by
$b_0 \ten b_1 \dots \ten b_k= b_0 \di b_1 \di \dots \di b_k$, the graded Leibniz rule and $\bbbc$-linearity. Differential and product are continuous.

The involution on $\B=\hat\Omega_0\B$ extends uniquely to an involution on $\Oi\B$ fulfilling $(\alpha \beta)^*= \beta^*\alpha^*$.

Let $\ideal \subset \Oi\B$ be a closed homogenous involutive ideal, closed under $\di$. Then $\Oi^{\ideal}\B:=\Oi\B/\ideal$ inherits the structure of an involutive differential algebra. 

The de Rham homology $H^{\ideal}_*(\B)$ is the topological homology of the complex $(\Oi^{\ideal}\B/\ov{[\Oi^{\ideal}\B,\Oi^{\ideal}\B]},\di)$, ``topological'' meaning that the closure of the range of $\di$ is divided out such that $H^{\ideal}_*(\B)$ is a Hausdorff space.

Let $V$ be a finitely generated projective right $\B$-module.

Let $V^*$ be the left $\B$-module of all right $\B$-module maps $s:V \to  \B$. If $V \subset \B^p$, we define for $s \in V$ with $s=\sum_{i=1}^p s_i e_i$
$$s^*:=\sum_{i=1}^p s_i^* e_i^* \in V^* \ ;$$ 
here $(e_1, \dots, e_p)$ is an orthonormal basis of $\bbbc^p$ and $(e_1^*, \dots, e_p^*)$ is its dual basis.  

We also define the right $\Oi^{\ideal}\B$-module $\Oi^{\ideal}V:=V \ten_{\B} \Oi^{\ideal}\B$ and the left $\Oi^{\ideal}\B$-module $\Oi^{\ideal}V^*=\Oi^{\ideal}\B \ten_{\B} V^*$.
The algebra $\Oi^{\ideal} V \otimes_{\Oi^{\ideal} \B} \Oi^{\ideal} V^*$ equals the algebra of $\Oi^{\ideal}\B$-linear homomorphisms on $\Oi^{\ideal}V$. It is endowed with a supertrace $$\tr:\Oi^{\ideal} V \otimes_{\Oi^{\ideal} \B} \Oi^{\ideal} V^* \to \Oi^{\ideal} \B/\ov{[\Oi^{\ideal} \B, \Oi^{\ideal} \B]} \ ,$$
which can be understood, for example, as induced by the supertrace 
$$\tr:M_p(\Oi^{\ideal} B) \to \Oi^{\ideal}\B/\ov{[\Oi^{\ideal}\B,\Oi^{\ideal}\B]}$$ 
defined, as usual, by adding up the elements on the diagonal.

\subsection{Connections and Chern character}
\label{connect}

In the following we study the product of a manifold with the ``noncommutative space'' $\B$. This can be considered as a special case of the previous situation applied to the algebra $\C(M,\B)$. However, we will use a different notation, since we want to keep the standard geometric notation as far as possible.  

Let $M$ be a manifold and $\Lambda^* T^*M$ the bundle of differential forms. The de Rham differential $d$ and  differential $\di$ make $\C(M,\Lambda^* T^*M \ten \Oi^{\ideal}\B)$ a double complex. As usual, the differential $\di_{tot}$ of the total complex  acts on a form $\alpha \beta$ with $\alpha \in \C(M,\Lambda^m T^*M)$ and $\beta \in \Ok^{\ideal}\B$ as $$
\di_{tot}(\alpha \beta)=(d+\di)\alpha \beta =(d \alpha)\beta+ (-1)^m\alpha(\di \beta) \ .$$ 

The topological homology of the total complex $H^*_{\B,\ideal}(M)$ is naturally isomorphic to $$\sum_{p+q=*} H^p(M,H^{\ideal}_q(\B)) \ .$$ 
Furthermore for $M$ closed
$$H^p(M,H^{\ideal}_q(\B)) \cong H^p(M) \ten H^{\ideal}_q(\B) \ .$$

Let $P \in \C(M,M_p(\B))$ be a (selfadjoint) projection and let $\F:=P(M\times \B^p)$. We call $\F$ a $\B$-vector bundle, since its fibers are projective finitely generated right $\B$-modules. The bundle $\F$ is endowed with a fiberwise $\B$-valued non-degenerated product (see \cite[Def. 5.27]{wa} for the terminology)
$$\langle s,t\rangle= \sum_{i=1}^p s_i^*t_i \ .$$
Therefore we call $\F$ a $\B$-hermitian bundle.

We have bundles $\Ol{\mu}\F=\F \ten_{\B} \Ol{\mu}\B=P(M \times (\Ol{\mu}\B)^p)$ and $\Ol{\mu}\F^*=\Ol{\mu}\B \ten_{\B} \F^*$. The fibers are right resp. left $\Ol{\mu}\B$-modules. We identify the bundle of Banach algebras $\Ol{\mu}\F \otimes_{\Ol{\mu}\B} \Ol{\mu}\F^*$ with $\{s \in M \times M_p(\Ol{\mu}\B)~|~PsP=s\}$. 

A connection on $\F$ in direction of $M$ is a $\B$-linear map $\nabla^{\F}:\C(M,\F)\to \C(M,T^*M \ten \F)$ such that for $s \in \C(M,\F)$ and $f \in \C(M,\B)$
$$\nabla^{\F}(sf)=\nabla^{\F}(s)f+ s df \ .$$

It extends to a map on $\C(M,\Lambda^* T^*M \ten\Oi^{\ideal}\F)$ such that for $s \in \C(M,\Lambda^k T^*M \ten\hat \Omega_m^{\ideal}\F)$ and $\alpha \in \C(M,\Lambda^* T^*M \ten\Oi^{\ideal}\B)$ 
$$\nabla^{\F}(s\alpha)=\nabla^{\F}(s)\alpha + (-1)^{k+m}s d\alpha \ .$$

A connection on $\F$ in direction of $\B$ is a $\C(M)$-linear map $\di_{\F}:\C(M,\F) \to \C(M,\hat \Omega_1^{\ideal}\F)$ such for $s \in \C(M,\F)$ and $f \in \C(M,\B)$
$$\di_{\F}(sf)=\di_{\F}(s)f + s\di(f)\ .$$

It extends to a map on $\C(M,\Lambda^* T^*M \ten\Oi^{\ideal}\F)$ such that for $s \in \C(M,\Lambda T^kM \ten \hat \Omega_m^{\ideal}\F)$ and $\alpha \in \C(M,\Lambda^* T^*M \ten\Oi^{\ideal}\B)$ 
$$\di_{\F}(s\alpha)=\di_{\F}(s)\alpha + (-1)^{k+m}s \di \alpha \ .$$

We call $\nabla^{\F}+\di_{\F}$ a total connection on $\F$.

A useful example is the total Grassmannian connection $PdP+ P\di P$.

As in the classical case one has:

\begin{lem}
\begin{enumerate}
\item Let $\nabla^{\F}, \tilde \nabla^{\F}$ be connections on $\F$ in direction of $M$. Then $\nabla^{\F}-\tilde \nabla^{\F} \in \C(M,\Lambda^1 T^*M \ten \F \ten_{\B} \F^*)$.
\item Let $\di_{\F}, \tilde \di_{\F}$ be connections on $\F$ in direction of $\B$. Then $\di_{\F}-\tilde \di_{\F} \in \C(M,\Oi^{\ideal} \F \otimes_{\Oi^{\ideal} \B} \Oi^{\ideal} \F^*)$. 
\item The curvature $(\nabla^{\F}+\di_{\F})^2$ is an element of $\C(M,\Lambda^* T^*M \ten \Oi\F \otimes_{\Oi\B} \Oi\F^*)$ of total degree 2. 
\end{enumerate}
\end{lem}

\begin{proof}
It is enough to show that the differences and the curvature are $\C(M,\Lambda^* T^*M \ten \Oi^{\ideal} \B)$-linear. For the differences this is straightforward, for the curvature it follows from a little calculation, see \S \ref{supcon} for a similar calculation. 
\end{proof}

The Chern character form is defined as
$$\ch^M_{\B}(\nabla^{\F}+\di_{\F}):=\tr e^{-(\nabla^{\F}+ \di_{\F})^{2}}
=\sum_{n=0}^{\infty} \frac{(-1)^n}{n!}\tr (\nabla^{\F}+ \di_{\F})^{2n} \ .$$
Mostly, we will write $\ch(\F)$, thus suppressing the choice of the connections from the notation.

Let $\nabla^{\F},\di_{\F},\tilde \nabla^{\F},\tilde \di_{\F}$ be as is the statement of the previous lemma and define the linear interpolations $$\nabla^{\F,t}:=(1-t)\nabla^{\F} + t \tilde \nabla^{\F} \ ,$$
$$\di^t_{\F}:=(1-t)\di_{\F}+ t \tilde \di_{\F} \ .$$

The proof of the classical transgression formula, as given for example in \cite[\S 1.5]{bgv}, works here also, yielding
$$\ch^M_{\B}(\tilde \nabla^{\F} + \tilde \di_{\F})-\ch^M_{\B}(\nabla^{\F} +  \di_{\F})=\di_{tot}\int_0^1 \tr\bigl((\tilde \nabla^{\F}-\nabla^{\F} + \tilde \di_{\F}- \di_{\F})e^{-(\nabla^{\F,t}+ \di^t_{\F})^2}\bigl) ~dt\ .$$ 

The Chern character form is closed with respect to $\di_{tot}$: This follows from the equality
$$\di_{tot}\ch^M_{\B}(PdP + P\di P)=\tr[ Pd P + P\di P, e^{-(Pd P + P\di P)^2}]=0$$
in combination with the transgression formula.

If $M$ is closed, the class of $\ch^M_{\B}(\nabla^{\F} + \di_{\F})$ in $H_{\B,\ideal}^*(M)$ only depends on the class of $\F$ in $K_0(C(M,\B))$ and is called the Chern character of $\F$.

The connection $\di_{\F}$ resp. $\nabla^{\F}$ is called compatible with the $\B$-hermitian structure on $\F$ if for $s_1,s_2 \in \C(M,\F)$
$$\di \langle s_1,s_2 \rangle=\langle \di_{\F} s_1,s_2\rangle + \langle s_1,\di_{\F}s_2\rangle$$
resp.
$$d \langle s_1,s_2 \rangle=\langle \nabla^{\F} s_1,s_2\rangle + \langle s_1,\nabla^{\F}s_2\rangle \ .$$

In the following, if $M$ is a Riemannian manifold we introduce for $s \in \C(M,\F)$ the operator 
$$s^*:\C_c(M,\F) \to \B,~ f \mapsto \int_M \langle s(x), f(x)\rangle dx \ .$$ 

\subsection{Projective systems}
\label{prosys}

In general, we will work with projective systems of algebras:

We assume that $(\A_i,\iota_{i+1,i}:\A_{i+1} \to \A_i)_{i \in \bbbn_0}$ is a projective system of
involutive Banach algebras with unit satisfying the following conditions:
\begin{itemize} 
\item The algebra $\A:=\A_0$ is a $C^*$-algebra.
\item For any $i \in \bbbn_0$ the map $\iota_{i+1,i}:\A_{i+1} \to \A_i$ is injective.
\item For any $j \in \bbbn_0$ the map $\iota_j:\Ai:= \pl{i}\A_i \to \A_j$ has dense
range.
\item For any $i \in \bbbn_0$ the algebra $\A_i$ is stable with respect to the holomorphic functional calculus in
$\A$.  
\end{itemize}

The motivating example is $\A_j=C^j(B)$ for a closed manifold $B$. Another important example stems from the Connes--Moscovici algebra in higher index theory.

Furthermore we assume that for each $i \in \bbbn$ there is given a closed homogenous involutive ideal $\ideal_i \subset \Oi\A_i$, closed under $\di$, such that $\iota_{i+1,i}(\ideal_{i+1})$ is a dense subset of $\ideal_i$. We denote its projective limit by $\ideal_{\infty}$. 

Then Lemmata \cite[1.3.6-1.3.8]{wa} imply that $$\pl{i}\Oi^{\ideal_i}\A_i=\Oi^{\ideal_{\infty}}\Ai \ ,$$
$$\pl{i}\Oi^{\ideal_i}\A_i/\ov{[\Oi^{\ideal_i}\A_i,\Oi^{\ideal_i}\A_i]}=\Oi^{\ideal_{\infty}}\Ai/\ov{[\Oi^{\ideal_{\infty}}\Ai,\Oi^{\ideal_{\infty}}\Ai]} \ ,$$
$$\pl{i}H^{\ideal_i}_*(\A_i)=H^{\ideal_{\infty}}_*(\Ai) \ .$$

For simplification, we will in general omit the dependence on the ideal from the notation.

\section{Dirac operators over $C^*$-algebras}
\label{condi}

Let $M$ be an oriented Riemannian  manifold of even dimension with a cylindric end $Z^+$ isometric to $(-1,\infty)\times N$ for some closed Riemannian manifold $N$ such that $M_c:=M \setminus Z^+$ is a compact manifold with boundary $N$. We orient $N$ such that the isometry $Z^+ \cong (-1,\infty)\times N$ is orientation preserving. The coordinate defined by the projection $p_1:Z^+ \to (-1, \infty)$ is denoted by $x_1$. We also define the projection $p:Z^+ \to N$. Furthermore we write $Z=\bbbr \times N$ and identify $N$ with $\{0\} \times N$. We denote the Levi-Civit\`a connection of $M$ by $\nabla^M$. Using the Riemannian metric we identify tangent and cotangent bundle.

Let $E$ be a $\bbbz/2$-graded Clifford module
 on $M$ and assume that there is a vector bundle $E^N$ on $N$ such that $E|_{Z^+}=(\bbbc^+ \oplus \bbbc^-)\ten p^*E^N$. Assume that the hermitian structure of $E$ is of product type on $Z^+$ that the Clifford module structure extends to a translation invariant Clifford module structure on the bundle $(\bbbc^+ \oplus \bbbc^-)\ten p^*E^N$ over $Z$. 

We further twist $E$ with a $C^*$-vector bundle:

Let $P \in \C(M,M_p(\Ai))$ be a selfadjoint projection such that $P|_{Z^+}$ does not depend on $x_1$ and define the $\A_i$-vector bundle 
$\F_i=P(M \times \A_i^p)$, $i \in \bbbn \cup \{0,\infty\}$. We get a projective system of $\bbbz/2$-graded Clifford modules
$${\mathcal E}_i:=E \ten \F_i \ .$$ 
The bundles $\F_i$ inherit a $\A_i$-hermitian structure from the standard $\A_i$-valued scalar product on $\A_i^p$ (see \S \ref{derham}). 

In general we abbreviate $\F=\F_0$.

Note that, by tensoring with the identity, $P\di P$ defines a connection on $\Ei$ in direction of $\Ai$.
Let $P\di P + \gamma$ be a connection on $\E_{\infty}$ in direction of $\Ai$ compatible with the $\Ai$-hermitian structure. Thus $\gamma=-\gamma^* \in \C(M,\Oi \E_{\infty} \otimes_{\Oi \Ai} \Oi \E_{\infty}^*)$. We assume that $\gamma$ supercommutes with Clifford multiplication and that $\gamma|_{Z^+}$ is independent of $x_1$. Clearly, $P\di P+\gamma$ defines also a connection on $\E_i$ in direction of $\A_i$ for each $i$.

Let $\nabla^{\E}$ be a Clifford connection on $\E$ (i.e. a connection on $\E$ in direction of $M$ compatible with the $\A$-valued scalar product and fulfilling $c(\nabla^Mv)=[\nabla^{\E},c(v)]$). We assume that $\nabla^{\E}$ is of product type on $Z^+$. We write $\dirac_{{\mathcal E}}=: c
\circ \nabla^{{\mathcal E}}$ for the associated Dirac operator. 

Let $\nabla^E$ be a Clifford connection on $E$.
The Dirac operator $c \circ \nabla^{E}$ on $\C_c(M,E\ten \A^p)$ is denoted by $\dirac_E$. Then $\dirac_{{\mathcal E}}- P \dirac_EP$ is a bundle endomorphism.

We identify $E^N$ with $E^+|_N$ and write $\F^N:=\F|_N$ and ${\mathcal E}^N:={\mathcal E}^+|_N$. Hence ${\mathcal E}^N=E^N \ten \F^N$. The induced Clifford module structure on $\E^N$ is given by $c_N(v):=c(dx_1)c(v)$ for $v \in TN \subset TM$. We denote the Dirac operator associated to $\E^N$ by $\dira_N$. By identifying $\E^+$ with $\E^-$ on $Z^+$ via $ic(dx_1)$ we have fixed an isomorphism
$$\E|_{Z^+}\cong (\bbbc^+ \oplus \bbbc^-)  \ten (p^*\E^N) \ .$$ 
Let $\Gamma$ denote the grading operator on $\E$. On the cylindric end
$$\dira_{{\mathcal E}}= c(dx_1)(\ra_1 -\Gamma \dira_N) \ .$$

For an operator $B$ acting on the sections of ${\mathcal E}^N$ we write $\tilde B$ for the induced operator $\Gamma B$ acting on the sections of ${\mathcal E}|_{Z^+}$. 
Note that $c(dx_1)\tilde B+ \tilde B c(dx_1)=0$. 

The proofs of the following facts are as in the complex case. We use the notation of \cite[Prop. 3.43]{bgv}. By $\Delta$ we denote the scalar Laplacian.

The curvature of $\nabla^{\mathcal E}$ decomposes as a sum 
$$(\nabla^{\mathcal E})^2=R^{\E} +F^{\E/S}$$ 
where $F^{\E/S}$ is the relative curvature of the bundle $\E$ and $R^{\E}$ is the Riemannian curvature, which acts on $\E$ via the Clifford multiplication.
The operator $\dirac_{{\mathcal E}}^2$ is a generalized Laplacian.  
Define for $s \in \C(M,{\mathcal E})$
$$\Delta^{\mathcal E}s=-\Tr(\nabla^{T^*M \ten {\mathcal E}} \nabla^{\mathcal E}s) \ .$$

We have the following analogue of the Lichnerowitz formula:
$$\dirac_{{\mathcal E}}^2=\Delta^{\mathcal E}+ c(F^{\E/S})  + \frac{r_M}{4} $$
where $r_M$ is the scalar curvature of $M$ and for $X \in \Omega^2(M,\End(\E))$
$$c(X):=\sum_{i<j} X(e_i,e_j) c(e_i)c(e_j)\in \C(M,\End(\E))$$
for an orthonormal frame $(e_i)_{i=1, \dots, \dim M}$ of $TM$. 
A similar formula yields a map $c:\Omega^n(M,\End(\E)) \to \C(M,\End(\E))$.

\section{Fredholm properties}
\label{DirHilb}

Recall that a closed densely defined operator $D$ on a Hilbert $\A$-module $H$ is called regular if $(1+D^*D)$ is surjective. Furthermore we call a regular operator $D$ Fredholm if it is Fredholm as a bounded operator from $H(D)$ to $H$, where $H(D)$ is the Hilbert $\A$-module that coincides with $\dom D$ as a right $\A$-module and is endowed with the $\A$-valued scalar product $$\langle f,g\rangle_{H(D)}:=\langle Df,Dg \rangle + \langle f,g\rangle \ .$$ 
Here $\langle ~,~\rangle$ denotes the $\A$-valued scalar product on $H$.

In the following we denote by $B(H)$ the algebra of bounded adjointable operators and by $K(H)$ the ideal of compact operators on $H$.

\begin{prop}
The closure $\D_E$ of the operator $\dirac_{\E}$ with domain $\C_c(M,\E)$ is regular on $L^2(M,\E)$.
\end{prop}

\begin{proof}
The closure of $\dira_E$ is clearly regular on $L^2(M,E \ten \A^p)$, hence so is the closure $\D_{\E}$ of $\dirac_{{\mathcal E}}$ since the operator $\dirac_{{\mathcal E}} \oplus (1-P) \dirac_E (1-P)$ on $L^2(M,E \ten \A^p)$ is a bounded perturbation of $\dirac_E$. 
\end{proof}

By a similar argument the closure $\D_N$ of $\dirac_N$ with domain $\C(N,\E^N)$ is regular. 

The following definition goes back to \cite{mp1}. 

\begin{ddd}
\label{trivial}
A selfadjoint bounded operator $A$ on $L^2(N,{\mathcal E}^N)$ such that $\D_N+A$ is invertible, is called a {\it trivializing operator} for $\D_N$. 
\end{ddd}

The existence of trivializing operators for Dirac operators associated to $C^*$-vector bundles was claimed in \cite{lp5}, see the Remark at the beginning of \cite[\S 2.4]{lp5}. However, the Technical Lemma in \cite[\S 2.2]{lp5} seems to require that the fibers of the $C^*$-vector bundle are full Hilbert $C^*$-modules. This is no major restriction since one may always modify the problem by a stabilization construction such that the condition is fulfilled.

The following definition anticipates the existence of the heat semigroup $e^{-t\D_N^2}$, which is proven in \S \ref{heatm}.

Note that taking the adjoint needs not be continuous on the space of bounded operators on the Banach space $L^2(N,\Ol{\mu}\E_i^N)$ (see \cite[\S 5.2.3]{wa} for a discussion on adjoints in this context).

\begin{ddd}
\label{adapt}
A selfadjoint bounded operator $\ad$ on $L^2(N,{\mathcal E}^N)$ is called {\it adapted} to $\D_N$ if the following holds on $L^2(N,\Ol{\mu}\E_i^N)$ and on $C^k(N,\Ol{\mu}\E_i^N)$ for any $\mu,i,k$: 
\begin{enumerate}
\item The operator $\ad$ is densely defined (by restriction and $\Ol{\mu}\A_i$-linear extension) and bounded and there is $C>0$ such that for $0<t<1$ with $B=\ad \D_N+ \D_N \ad$ 
$$\|Be^{-t\D_N^2}\| \le Ct^{-1/2}\ , \quad  \|e^{-t\D_N^2}B\| \le Ct^{-1/2} \ .$$
\item The operator $[P\di P,\ad]$ is $\Ol{\mu}\A_i$-linear.
\end{enumerate}
\end{ddd}

For example, a selfadjoint integral operator with integral kernel in $\C(N \times N,\E^N_{\infty} \boxtimes_{\Ai} (\E^N_{\infty})^*)$ is adapted to $\D_N$.

Let $\chi \in \C(\bbbr)$ with $\supp(\chi) \subset [1,\infty)$ and $\chi|_{[2,\infty)}=1$. We consider $\chi$ as a function on $Z^+$, which we extend by zero to $M$. Let $A$ be a trivializing operator for $\D_N$. 

\begin{prop}
The closure $\D_{\E}(A)$ of the operator $\dirac_{{\mathcal E}}- \chi c(dx_1) \tilde A$ is a Fredholm operator from $H({\mathcal D}_{\E}(A))$ to $L^2(M,{\mathcal E})$. 
\end{prop}

\begin{proof}
First note that any integral operator with integral kernel in the Schwartz space ${\mathcal S}(M \times M,\E\boxtimes_{\A} \E^*)$ is in $K(L^2(M,\E))$ and in $K(H(\D_{\E}(A)))$.
By standard cutting-and-pasting methods an operator $Q \in B(L^2(M,\E))$ that is also continuous from $L^2(M,\E)$ to $H(\D_{\E}(A))$ can be constructed such that $\D_{\E}(A)Q-1$ and $Q\D_{\E}(A)-1$ are integral operators with integral kernels in ${\mathcal S}(M \times M,\E\boxtimes_{\A} \E^*)$. (The operator $Q(0)$ defined in \S \ref{res} works.) The operator $Q$ is automatically adjointable from $L^2(M,\E)$ to $H(\D_{\E}(A))$ by
\begin{align*}
Q&=(\D_{\E}(A)+i)^{-1}(\D_{\E}(A)+i)Q\\
&=(\D_{\E}(A)+i)^{-1}(\D_{\E}(A)Q-1)+(\D_{\E}(A)+i)^{-1}(1+iQ) \ .
\end{align*}
\end{proof} 

By a homotopy argument the index of $\D_{\E}(A)$ does not depend on $\chi$. Furthermore it is invariant under  small perturbations of $A$.

As usual, we write $$\D_{\E}(A)=\left(\begin{array}{cc} 0 & \D_{\E}(A)^-\\ \D_{\E}(A)^+ & 0 \end{array}\right)$$ with respect to the decomposition $L^2(M,\E)=L^2(M,\E^+) \oplus L^2(M,\E^-)$.
 
In the following we define a perturbation of $\D_{\E}(A)$ with closed range by adapting a construction of Atiyah and Singer, see \cite[Ch. 9.5]{bgv}. Recall that the index of a Fredholm operator with closed range equals the difference of the $K$-theory classes of its kernel and cokernel.

\begin{prop}
\label{projmod}
There is a finite subset $\{f_1, \dots, f_{\N}\}$ of $\C_c(M,\Ei^-)$ generating a projective $\A$-module $Q$ 
such that $Q + \Ran {\mathcal D}_{\E}(A)^+=L^2(M,{\mathcal E}^-)$.
\end{prop}

\begin{proof}
Since ${\mathcal D}_{\E}(A)^+$ is a Fredholm operator, there is a projective module $\tilde Q \subset
L^2(M,{\mathcal E}^-)$ such that $\tilde Q + \Ran {\mathcal D}_{\E}(A)^+=L^2(M,{\mathcal E}^-)$. Let $\tilde P:L^2(M,{\mathcal E}^-) \to \tilde Q$
be the (orthogonal) projection onto $\tilde Q$. Let $(e_i)_{i \in \bbbn} \subset
\C_c(M,E \ten \bbbc^p)$ be an orthonormal basis of $L^2(M,E \ten \A^p)$ and for $j \in \bbbn$
let $P_j$ be the orthogonal projection onto the $\A$-linear span
of the first $j$ basis vectors. We extend $\tilde P$ by zero to a projection on $L^2(M,E \ten \A^p)$. There is $\N \in \bbbn$ such that $\tilde P|_{P_{\N}(\tilde Q)}:P_{\N}(\tilde Q) \to \tilde Q$ is an
isomorphism and such that $P_{\N}(\tilde Q) + \Ran \D_{\E}(A)^+=L^2(M,{\mathcal E}^-)$ (see
\cite[Prop. 5.1.21]{wa}). Since $\tilde P=\tilde PP$, it follows that $\tilde P:PP_{\N}(\tilde Q) \to \tilde Q$ is an
isomorphism. Thus the set $\{f_i:=Pe_i~|~i=1, \dots, \N\}$ fulfills the conditions.
\end{proof}

Now the perturbation of ${\mathcal D}_{\E}(A)^+$ is defined as follows: Let $M'$ be the disjoint union of $M$ and a point $*$. On $M'$ we define the
$\A$-vector bundles
$({\mathcal E}')^+={\mathcal E}^+ \cup \A^{\N}$ with $\N$ as in the proposition, and $({\mathcal E}')^-={\mathcal E}^- \cup \{0\}$. 
Furthermore we define $E'$ on $M'$ by $E'|_M=E$ and $E'(*)=(\bbbc^+)^{\N}$, and $P'$ by $P'|_M=P$ and such that $P'(*)$ is a projection of rank one in $M_p(\bbbc)$. Then $\E'\cong E' \ten P'(M' \times \A^p)$. Furthermore we extend the noncommutative one-form $\gamma$ by zero to $M'$. 

Let $(v_i)_{i=1, \dots, \N}$ be the standard orthonormal basis of $\A^{\N}=\E'(*)$. We set
$$\Ki^+:= \sum\limits_{i=1}^{\N} f_i v_i^*:L^2(M',{\mathcal E}'^+) \to L^2(M',{\mathcal E}'^-) $$ and  
$$\Ki^-:=(\Ki^+)^* = \sum\limits_{i=1}^{\N} v_i f_i^*:L^2(M',{\mathcal E}'^-) \to L^2(M',{\mathcal E}'^+) \ .$$
Then $\Ki:=\left(\begin{array}{cc} 0 & \Ki^- \\ \Ki^+ & 0 \end{array}\right)$ is a selfadjoint and compact operator on $L^2(M',\E')$. 

For any $\rho
\neq 0$ the operator $(\D_{\E}(A)+ \rho \Ki)^+$ is surjective and
\begin{align*}
\ind {\mathcal D}_{\E}(A)^+ &= \ind (\D_{\E}(A)+ \rho \Ki)^+ - [\A^{\N}] \\
& =[\Ker (\D_{\E}(A)+ \rho \Ki)]-[\A^{\N}] \ .
\end{align*}

For later purposes we note that we may enlarge $\N$ without changing the operator by setting $f_i=0$ for $i$ large.

In the following we drop the indices and write for example ${\mathcal E}$ and $M$ for ${\mathcal E}'$ and $M'$, respectively.

\section{Heat semigroup and superconnection for a closed manifold}
\label{closed}

In preparation for the general case we restrict our attention to the case where $M$ is closed. In this section $M$ may be odd-dimensional, then the bundle $E$ is assumed ungraded. Let $\ad$ be a symmetric bounded operator on $L^2(M,\E)$ that is adapted to $\D_{\E}$ (see Def. \ref{adapt}). Then $r \ad$ is also adapted for $r \in [-1,1]$. We study the properties of $\dirac_{{\mathcal E}}+ r\ad$ on $L^2(M,\Ol{\mu}{\mathcal E}_i)$ and on $C^k(M,\Ol{\mu}\E_i)$.

In general we demand that unbounded operators on $L^2(M,\Ol{\mu}{\mathcal E}_i)$ resp. $C^k(M,\Ol{\mu}{\mathcal E}_i)$ have $\C(M,\Ol{\mu}{\mathcal E}_i)$ as a core of their domain. The operators considered here are closable on $L^2(M,\Ol{\mu}{\mathcal E}_i)$ resp. $C^k(M,\Ol{\mu}{\mathcal E}_i)$ by \cite[Lemma 5.2.10]{wa} since they have a densely defined adjoint. We define $(\D_{\E} + r\ad)^m$ as the closure of the operator $(\dirac_{{\mathcal E}} +  r\ad)^m$ on $L^2(M,\Oei)$ resp. $C^k(M,\Ol{\mu}{\mathcal E}_i)$, thus we first take powers, then closures.

We often consider the family $\D_{\E}+r\ad$ as an operator on $C([-1,1],L^2(M,\Ol{\mu}{\mathcal E}_i))$ resp. $C([-1,1],C^k(M,\Ol{\mu}{\mathcal E}_i))$.

\subsection{Heat semigroup and resolvents}
\label{heatm}

\begin{prop}
\label{semigroupm}
\begin{enumerate}
\item The family $-(\D_{\E} + r\ad)^2$ generates a holomorphic semigroup on $C([-1,1],L^2(M,\Ol{\mu}{\mathcal E}_i))$ and on $C([-1,1],C^k(M,\Ol{\mu}{\mathcal E}_i))$.

\item For $t>0$ the operator $e^{-t(\D_{\E} + r\ad)^2}$ is an integral operator with integral
kernel $k(r)_t \in \C(M \times M, \Ei \boxtimes_{\Ai} \Ei^*)$ depending smoothly on $(t,r)$.

\item Let $m> \frac{\dim M + k}{2}$. There is $C>0$ such that for $|r|\le 1$ and  $0 < t\le 1$ 
$$\|k(r)_t\| \le Ct^{-m} \ , \quad \|\frac{d}{dr}k(r)_t\| \le Ct^{-m}$$ 
in $C^k(M \times M,\E_i \boxtimes_{\A_i} \E_i^*)$. 

\item Let $B$ be a first order differential operator resp. $B=\D_{\E} \ad+ \ad \D_{\E}$. There is $C>0$ such that for $|r|\le 1$ and  $0 < t\le 1$ 
$$\|B e^{-t(\D_{\E} + r\ad)^2}\| \le C t^{-\frac 12} \ , \quad \| \frac{d}{dr} B e^{-t(\D_{\E} + r\ad)^2}\| \le C \ ,$$
$$\|e^{-t(\D_{\E} + r\ad)^2}B\| \le C t^{-\frac 12} \ , \quad \| \frac{d}{dr}  e^{-t(\D_{\E} + r\ad)^2}B\| \le C $$
on $L^2(M,\Ol{\mu}{\mathcal E}_i)$. Furthermore for $B=\D_\E \ad+ \ad \D_\E$ these estimates hold also on $C^k(M,\Ol{\mu}{\mathcal E}_i)$.

\item Let $\phi, \psi \in \C(M)$ with disjoint support. Assume that there is a sequence $(\xi_n)_{n \in \bbbn} \subset \C(M)$ with $\xi_1=\phi$, $\supp (1-\xi_{n+1}) \cap \supp \xi_n = \emptyset$, $\supp \xi_n \cap \supp \psi = \emptyset$ and such that $[\xi_n, \ad]$ is an integral operator with integral kernel in $\C(M \times M,\E_i \boxtimes_{\A_i} \E_i^*)$. Then for $|r|\le 1,~ 0<t \le 1$ $$\|\phi(x)k(r)_t(x,y)\psi(y)\| \le Ct \ , \quad \|\phi(x)\frac{d}{dr} k(r)_t(x,y)\psi(y)\| \le Ct$$ in $C^k(M \times M,\E_i \boxtimes_{\A_i} \E_i^*)$.
\end{enumerate}
\end{prop}

\begin{proof}
For $t>0$ the integral kernel $k_t(x,y)$ associated of $e^{-t\D_{\E}^2}$ can be constructed as in \cite[Ch. 2]{bgv}. Then one shows as in \cite[\S 3]{wa} that the operator $-\D_{\E}^2$ generates a strongly continuous semigroup on $C^k(M,\Ol{\mu}\E_i)$ and a holomorphic semigroup on $L^2(M,\Ol{\mu}\E_i)$. Assume that  $B$ is a first order differential operator. Then on $L^2(M,\Ol{\mu}\E_i)$ one get as in \cite[\S 3.2]{wa} the estimates $\|Be^{-t\D_{\E}^2}\| \le Ct^{-1/2}$ and $\|e^{-t\D_{\E}^2}B\| \le Ct^{-1/2}$ for small $t$.

It follows from the asymptotics of the heat kernel by using that $-(\D_{\E}^2)_x k_t(x,y)=\frac{d}{dt}k_t(x,y)$ that on $C^k(M,\Ol{\mu}\E_i)$
$$\|\D_{\E}^2 e^{-t\D_{\E}^2}\| \le Ct^{-1}$$ 
for small $t>0$. Thus, by Prop. \ref{exthol} the semigroup extends to a holomorphic semigroup on $C^k(M,\Ol{\mu}\E_i)$. For $r=0$ the first estimate of (3) is \cite[Lemma 2.39]{bgv}. 

(1) follows now from Prop. \ref{relpert}.

We use Volterra development in order to prove (2) and (3). We set $R:=(\D_{\E}+r\ad)^2-\D_{\E}^2$. Since $r\ad$ is adapted to $\D_{\E}$ there is $C>1$ such that $\|Re^{- t \D_{\E}^2}\| \le Ct^{-1/2}$ and $\|e^{- t \D_{\E}^2}R\| \le Ct^{-1/2}$ for  $0<t\le 1,~|r| \le 1$.

The series $$e^{-t({\mathcal D}_{\E}^2+R)}=\sum_{n=0}^{\infty} (-1)^n t^n S_n(t)$$ with
$$S_n(t)=\int_{\Delta^n} e^{-u_0t\D_{\E}^2}  R e^{-u_1 t \D_{\E}^2}
 \dots  R e^{-u_n t \D_{\E}^2} ~du_0 \dots du_n$$ converges for $t < (\frac 1C)^2$ uniformly in $r \in [-1,1]$.

By \cite[Lemma 2.39]{bgv} the integral kernel of $e^{-t{\mathcal D}_{\E}^2}$ is bounded in $C^k(M \times M, \E_i \boxtimes_{\A_i} \E_i^*)$ by $C t^{-m}$ for $m> \frac{\dim M + k}{2}$ and $|r|\le 1$,  $0 < t\le 1$ . By a standard argument (see the proof of \cite[Theorem 9.48]{bgv}), which uses that for $(u_0, \dots, u_n) \in \Delta_n$ there is $u_i\ge\frac 1n$, the operator $S_n(t)$ is an integral operator and there is $C>0$, independent of $n$, such that for $|r|\le 1$, $0 < t\le 1$  its integral kernel is bounded by $C^n\left(\frac{t}{n}\right)^{-m}t^{-n/2}$ in $C^k(M \times M, \E_i \boxtimes_{\A_i} \E_i^*)$. Similar estimates hold for the derivatives with respect to $r$ and $t$. It follows that the series of integral kernels converges in $C^k(M \times M, \E_i \boxtimes_{\A_i} \E_i^*)$ for $0<t< (\frac 1C)^2$ and that the limit depends smoothly on $r,t$. Estimate (3) also follows. For $t\ge(\frac 1C)^2$ the integral kernel is constructed using the semigroup law.

(4) The estimates also follow from the Volterra development: One interchanges the operator $B$ with the summation and integration and uses that $\|Be^{-u_0t\D_{\E}^2}\| \le C(u_0t)^{-1/2}$ and $\|e^{-u_0t\D_{\E}^2}B\| \le C(u_0t)^{-1/2}$ for $t$ small.

(5)  The integral kernel of $\xi_m S_0(t)\psi$ is bounded by $Ct$ in $C^k(M \times M,\E_i \boxtimes_{\A_i} \E_i^*)$ by \cite[Theorem 2.30]{bgv}. We show by induction on $n$ that the integral kernel of $\phi S_n(t)\psi$ converges in $C^k(M \times M, \E_i \boxtimes_{\A_i} \E_i^*)$ for $t \to 0$ uniformly in $r \in [-1,1]$. By the estimate of the integral kernel of $S_n(t)$ from above we need to prove this only for finitely many $n$. Then the assertion follows.

It holds that
\begin{align*}
\xi_m S_n(t) \psi &= \int_0^1  (1-u_0)^{n-1} \xi_m S_0(u_0t)R\xi_{m+2} S_{n-1}((1-u_0)t)\psi ~du_0\\
&\quad +   \int_0^1 (1-u_0)^{n-1} \xi_m S_0(u_0t) R(1- \xi_{m+2}) S_{n-1}((1-u_0)t)\psi~du_0 \ .
\end{align*}

The integral kernel of the first term on the right hand side has a limit for $t \to 0$, since the integral kernel of $\xi_{m+2} S_{n-1}((1-u_0)t)\psi$ has one by induction.

We denote by $\sim$ equality up to integral operators with smooth integral kernels that have a limit for $t \to 0$.

We have that
\begin{align*}
[R,\xi_{m+2}]&=r[\ad\D_{\E} + \D_{\E}\ad + r\ad^2,\xi_{m+2}] \\
&=r([\ad,\xi_{m+2}]\D_{\E} + \ad (1-\xi_{m+1})c(d\xi_{m+2})+ (1-\xi_{m+1})c(d\xi_{m+2})\ad +  \D_{\E}[\ad,\xi_{m+2}]) \\
& \quad + r^2(\ad[\ad, \xi_{m+2}] + [\ad,\xi_{m+2}]\ad)\\
&\sim r(1-\xi_{m+1})(\ad c(d\xi_{m+2})+c(d\xi_{m+2})\ad) \ .
\end{align*}  
Thus 
\begin{align*}
\lefteqn{\xi_m S_0(u_0t) R(1- \xi_{m+2})}\\
&\sim r\xi_m S_0(u_0t)(1-\xi_{m+1})((1-\xi_{m+2})R- \ad c(d\xi_{m+2})- c(d\xi_{m+2})\ad)\\
& \sim 0 \ .
\end{align*}  
Thus also the smooth integral kernel of the second term has a limit for $t \to 0$.

The proof for the derivative with respect to $r$ is analogous.
\end{proof}

We collect some further estimates, in particular concerning the large time behaviour of the heat semigroup.

Since $e^{-t(\D_{\E} + r\ad)^2}$ is an integral operator with smooth integral kernel, its spectral radius on $L^2(M,\Ol{\mu}{\mathcal E}_i)$ smaller than or equal to $1$, see Prop. \ref{invspec}. Thus the spectral radius of $e^{-t(\D_{\E} + r\ad)^2}$ on $C([-1,1],L^2(M,\Ol{\mu}{\mathcal E}_i))$ is smaller than or equal to $1$. From Prop. \ref{specsem} we infer that for any $\ve >0$ there is $C>0$ such that on $L^2(M,\Ol{\mu}{\mathcal E}_i)$ for all $t$ and $r \in [-1,1]$
\begin{align}
\label{decM}\|e^{-t(\D_{\E} + r\ad)^2}\| &\le Ce^{\ve t} \ .
\end{align}

From
$$k(r)_t(x,y)= \int_M k(r)_{1/2}(x,z)\bigr(e^{-(t-1)(\D_{\E} + r\ad)^2}k(r)_{1/2}(\cdot,y)\bigl)(z) ~dz$$ 
it follows that there is $C>0$ such that for $t>1$ and $r \in [-1,1]$ in $C^k(M \times M,\E_i \boxtimes_{\A_i} \E_i^*)$
\begin{align}
\label{estimhkcm}
\|k(r)_t\| &\le Ce^{\ve t} \ .
\end{align}
This in turn implies that the inequality \ref{decM} holds also on $C^k(M,\Ol{\mu}\E_i)$.

Furthermore by Duhamel's formula (see Prop. \ref{duhform})
$$\frac{d}{dr}e^{-t(\D_{\E} + r\ad)^2}=-\int_0^t e^{-(t-s)(\D_{\E} + r\ad)^2}(\D_{\E} \ad + \ad \D_{\E}+2 r\ad^2) e^{-s(\D_{\E} + r\ad)^2}~ds \ .$$
It follows that there is $C>0$ such that for $t>1$ and $r \in [-1,1]$ in $C^k(M \times M,\E_i \boxtimes_{\A_i} \E_i^*)$
\begin{align}
\label{estimhkcmder}
\|\frac{d}{dr}k(r)_t\| &\le Ce^{\ve t} \ .
\end{align}

By Duhamel's formula we can improve Prop. \ref{semigroupm} (4): There is $C>0$ such that for $t$ small and $r \in [-1,1]$ on $C^k(M,\Oei)$ as well as on $L^2(M,\Oei)$
$$\|\frac{d}{dr} e^{-t(\D_{\E} + r\ad)^2}\| \le C t^{1/2} \ .$$

Our estimates allow us to obtain information about the resolvents of $\D_\E + r \ad$:

For $\re \lambda^2<0$ and $n \in \bbbn$ the integral
$$(\D_{\E} + r\ad -\lambda)^{-n}=\frac{1}{(n-1)!} \int_0^{\infty} t^{n-1} (\D_{\E} + r\ad+\lambda)^n e^{-t((\D_{\E} + r\ad)^2-\lambda^2)} ~dt \ ,$$ which converges as a bounded operator on $L^2(M,\Ol{\mu}{\mathcal E}_i)$, is of class $C^1$ in $r$.

\begin{prop}
\label{Mregul}

Let $m >\frac{\dim M + k+1}{2}$ and $\re \lambda < 0$.

The operator $((\D_{\E} + r\ad)^2 -\lambda)^{-m}$ is an integral operator with  integral kernel in $C^k(M \times M,\E_i \boxtimes_{\A_i} \E_i^*)$, which is of class $C^1$ in $r$. In particular $((\D_{\E} + r\ad)^2 -\lambda)^{-m}$ is a bounded operator from $L^2(M,\Ol{\mu}\E_i)$ to $C^k(M,\Ol{\mu}\E_i)$, which is of class $C^1$ in $r$.
\end{prop}

\begin{proof}
We have that
$$((\D_{\E} + r\ad)^2 -\lambda)^{-m}= \frac{1}{(m-1)!}\int_0^{\infty} t^{m-1} e^{-t((\D_{\E} + r\ad)^2- \lambda)} ~dt \ .$$ 
On the level of integral kernels the convergence of the integral follows for large $t$ from  eq. \ref{estimhkcm} and \ref{estimhkcmder} and for small $t$ from Prop. \ref{semigroupm} (3). 
\end{proof}

\begin{prop}
\label{offdiag}
Let $\phi, \psi \in \C(M)$ with disjoint support. Assume that there is a sequence $(\xi_n)_{n \in \bbbn} \in \C(M)$ with $\xi_1=\phi,~ \supp (1-\xi_{n+1}) \cap \supp \xi_n = \emptyset$ and $\supp \xi_n \cap \supp \psi = \emptyset$ and such that $[\xi_n, \ad]$ is an integral operator with integral kernel in $\C(M \times M,\E_i \boxtimes_{\A_i} \E_i^*)$.

For $\re \lambda<0$ and $k \in \bbbn$ the operator 
$$\phi((\D_{\E} + r\ad)^2 - \lambda)^{-k}\psi$$ is an integral operator with integral kernel in $\C(M\times M,\E_i\boxtimes_{\A_i} \E_i^*)$, which is of class $C^1$ in $r$.
\end{prop}

\begin{proof}
The assertion follows from Prop. \ref{semigroupm} (5) by the above integral formula.
\end{proof}

The previous proposition implies an analogous statement for $\phi(\D_{\E} + r\ad - \lambda)^{-1}\psi$ if $\re \lambda^2 <0$ by 
\begin{align*}
\lefteqn{\phi(\D_{\E} + r\ad - \lambda)^{-1}\psi}\\
 &=\phi(\D_{\E} + r\ad + \lambda)((\D_{\E} + r\ad)^2 - \lambda^2)^{-1}\psi\\
&=(\D_{\E} + r\ad + \lambda)\phi((\D_{\E} + r\ad)^2 - \lambda^2)^{-1}\psi+[\phi,\D_{\E}+r\ad]((\D_{\E} + r\ad)^2 - \lambda^2)^{-1}\psi
\end{align*}

In the following proposition we fix $r=1$. 

\begin{prop}
\label{Mresolv}
Assume that $M$ is even-dimensional (and closed). 
For $\N$ as in \S \ref{condi} big enough there is $\omega>0$ and a symmetric integral operator $\Kappa$ with integral kernel in $\C(M \times M,\E_{\infty} \boxtimes_{\Ai} \E_{\infty}^*)$ such that
 $\D_{\E} + \ad+ \Kappa -\lambda$ is invertible on $L^2(M,\Ol{\mu}{\mathcal E}_i)$ for $\re \lambda^2<\omega$.
 
For $\re \lambda^2<\omega$ an analogue of Prop. \ref{Mregul} holds for $(\D_{\E} + \ad+ \Kappa)^2 -\lambda^2$ and an analogue of Prop. \ref{offdiag} for $\phi((\D_{\E} + \ad +  \Kappa)^2 - \lambda^2)^{-k}\psi$.
\end{prop}

\begin{proof}
We first show that for $\N$ big enough there is an integral operator $\Kappa$ with integral kernel in $\C(M \times M,\E_{\infty} \boxtimes_{\Ai} \E_{\infty}^*)$ such that $\D_{\E} + \ad+ \Kappa$ is invertible on $L^2(M,\E)$.

The operator $\D_{\E} + \ad$ is a Fredholm operator with vanishing index. By a construction as in \S \ref{DirHilb} for $\N$ big enough there is a symmetric integral operator $\Kappa_1$ with integral kernel in $\C(M \times M,\E_{\infty} \boxtimes_{\Ai} \E_{\infty}^*)$ such that $\D_{\E} + \ad + \Kappa_1$ has closed range as an operator on $L^2(M,\E)$. Hence the projection $\Pj$ onto the kernel of $\D_{\E} + \ad+ \Kappa_1$ is well-defined and $\D_{\E} + \ad + \Kappa_1 + \Pj$ has a bounded inverse on $L^2(M,\E)$. By Prop. \ref{proj} the operator $\Pj$ is an integral operator with integral kernel in $L^2(M \times M,\E \boxtimes_{\A} \E^*)$. Now the claim from the beginning of the proof follows from the fact that $\C(M \times M,\E_{\infty} \boxtimes_{\Ai} \E_{\infty}^*)$ is dense in $L^2(M \times M,\E \boxtimes_{\A} \E^*)$.

There is $\omega>0$ such that the spectrum of the semigroup $e^{-t(\D_{\E} + \ad + \Kappa)^2}$ on $L^2(M,\E)$ is contained in the interval $[0,e^{-\omega t}]$ for all $t$. As above by Prop. \ref{invspec} and Prop. \ref{specsem} for any $\ve >0$ there is $C>0$ such that for all $t>0$ on $L^2(M,\Ol{\mu}{\mathcal E}_i)$ and on $C^k(M,\Ol{\mu}\E_i)$
$$\|e^{-t(\D_{\E} + \ad + \Kappa)^2}\| \le Ce^{-(\omega-\ve) t} \ .$$
Now the assertion follows as above.
\end{proof}

\subsection{Heat kernel for the superconnection}

Recall the definition of the connection $P\di P + \gamma$ in direction of $\A_i$ from \S \ref{condi}.

We define the superconnection associated to $\D_{\E} + r\ad$ by 
$$\spc(r):=P\di P+ \gamma + \D_{\E} + r\ad \ .$$

The curvature of $\spc(r)$ is the $\Ol{\mu}\A_i$-linear map
$$\spc(r)^2=({\mathcal D}_{\mathcal E}+r\ad)^2+ [P\di P+ \gamma,\D_{\E} + r\ad] + P\di P \di P \ .$$ 

The operator $\spc(0)^2$ is a generalized Laplacien (see \S \ref{supcon} for the calculation of the supercommutator), hence the construction of \cite[Ch. 2]{bgv} applies here as well. Then one can construct and study the semigroup generated by $-\spc(r)^2$ as in the proof of Prop. \ref{semigroupm}.

Since $\spc(r)^2$ is a nilpotent bounded perturbation of $({\mathcal D}_{\mathcal E} + r\ad)^2$ on $L^2(M,\Ol{\mu}{\mathcal E}_i)$ resp. on $C^k(M,\Ol{\mu}{\mathcal E}_i)$, analogues of the estimates \ref{decM} and \ref{estimhkcm} hold.

We conclude:

\begin{prop}
\label{intkersupcon}
The operator $-\spc(r)^2$ generates a holomorphic semigroup on $C([-1,1],L^2(M,\Ol{\mu}{\mathcal E}_i))$ resp. on $C([-1,1],C^k(M,\Ol{\mu}{\mathcal E}_i))$, and $e^{-t\spc(r)^2}$ is an integral operators with integral kernel $p(r)_t \in \C(M \times M,\Ol{\mu}\E_i\boxtimes_{\Ol{\mu}\A_i} \Ol{\mu}\E_i^*)$ depending smoothly on $(t,r)$. The following estimates hold:
\begin{enumerate}
\item Let $B$ be a first order differential operator resp. $B=\D_{\E} \ad+\ad \D_{\E}$. There is $C>0$ such that for $|r|\le 1$ and  $0 < t\le 1$
$$\|B e^{-t\spc(r)^2}\| \le C t^{-\frac 12} \ , \quad \| \frac{d}{dr} B e^{-t \spc(r)^2}\| \le C \ ,$$
$$\|e^{-t\spc(r)}B\| \le C t^{-\frac 12} \ , \quad \| \frac{d}{dr}  e^{-t \spc(r)^2}B\| \le C $$
on $L^2(M,\Ol{\mu}{\mathcal E}_i)$. Furthermore for $B=\D_\E \ad+\ad \D_\E$ these estimates hold also on $C^k(M,\Ol{\mu}{\mathcal E}_i)$.
\item For $\ve>0$ there is $C>0$ such that on $L^2(M,\Ol{\mu}\E_i)$ and on $C^k(M,\Ol{\mu}\E_i)$ for all $r \in [-1,1]$ and $t>1$
$$\|e^{-t\spc(r)^2}\| \le C e^{\ve t} \ , \quad \|\frac{d}{dr}e^{-t\spc(r)^2}\| \le C e^{\ve t} \ .$$  
\item For $\ve>0$ there is $C>0$ such that for all $t>1$ and $r \in [-1,1]$
$$\|p(r)_t\| \le C e^{\ve t}, \quad \|\frac{d}{dr}p(r)_t\| \le  C e^{\ve t} \ ,$$
 where the norm is taken in $C^k(M \times M, \Ol{\mu}\E_i \boxtimes_{\Ol{\mu}\A_i} \Ol{\mu}\E_i ^*)$.
\item Let $m>\frac{\dim M+k}{2}$. There is $C>0$ such that for for $|r|\le 1$ and  $0 < t\le 1$ 
$$\|p(r)_t\| \le Ct^{-m},\quad \|\frac{d}{dr}p(r)_t\| \le Ct^{-m}$$ 
in $C^k(M \times M, \Ol{\mu}\E_i \boxtimes_{\Ol{\mu}\A_i} \Ol{\mu}\E_i^*)$. 
\item Let $\phi, \psi \in \C(M)$ with disjoint support. Assume that there is a sequence $(\xi_n)_{n \in \bbbn}\subset \C(M)$ with $\xi_1=\phi,~ \supp (1-\xi_{n+1}) \cap \supp \xi_n = \emptyset$ and $\supp \xi_n \cap \supp \psi = \emptyset$ and such that $[\xi_n, \ad]$ is an integral operator with integral kernel in $\C(M \times M,\E_i \boxtimes_{\A_i} \E_i^*)$. Then there is $C>0$ such that for $|r|\le 1$ and  $0 < t\le 1$ 
$$\|\phi(x)p(r)_t(x,y)\psi(y)\| \le Ct,\quad\|\phi(x)\frac{d}{dr}p(r)_t(x,y)\psi(y)\| \le Ct$$
 in $C^k(M \times M,\Ol{\mu}\E_i \boxtimes_{\Ol{\mu}\A_i} \Ol{\mu}\E_i^*)$.
\end{enumerate}
\end{prop}

\subsection{Rescaled superconnection and limit for $t \to 0$}
\label{supcon}

In the following we write $$\spc:=\spc(0)=P\di P + \gamma + \D_{\E}$$ and define the rescaled superconnection $$\spc_t=P \di P + \gamma + \sqrt t \D_{\E} \ .$$ 

We determine the limit of the pointwise supertrace of the heat kernel of $e^{-\spc_t^2}$ for $t \to 0$, closely following \cite[Ch. 4]{bgv}. 

Assume $n=\dim M$ even.

We adapt the Getzler rescaling to our situation:

Let $U$ be a geodesic coordinate patch centered at $x_0 \in M$. In the following we denote by $x$ the coordinate. Let $V:=T_{x_0}U\cong \bbbr^n,~S_V:=S_{x_0}$ and $W:=\Hom_{C(V)}(S_V,E_{x_0})$. Via parallel transport along geodesics we identify $TU$ with $U \times V$, furthermore the spinor bundle $S|_U$ with $U\times S_V$, and $\Hom_{C(TU)}(S|_U,E|_U)$ with $U \times W$. There is an induced isomorphism $E|_U \cong U \times (S_V \ten W)$. As before ${\mathcal F}|_U= P(U\times \A^p)$. We consider ${\mathcal E}_i|_U$ as a subbundle of the trivial bundle $U\times (S_V \ten W \ten \A_i^p) \cong U\times S_V \ten \A^m_i$ for $m=p\dim W$.  

For $\alpha \in \C(\bbbr^+ \times U, \Lambda^* V^* \ten M_m(\Ol{\mu}\A_i))$ we rescale
$$(\delta_u \alpha)(t,x):= \sum_{k,l} u^{-\frac{k+l}{2}} \alpha(ut, u^{\frac 12} x)_{[k][l]} \ .$$
Here  $\alpha(t, x)_{[k][l]}$ is the homogenous component of $\alpha(t,x)$ that is of degree $k$ with respect to the grading on $\Lambda^*V^*$ and of degree $l$ with respect to $\Ol{\mu}\A_i$.

We identify the Clifford algebra $C(V^*)$ with the exterior algebra $\Lambda V^*$ via the symbol map 
$$\sigma(a):= c(a)1$$ with $1 \in \Lambda^0(V^*)$. The restriction of the integral kernel of $e^{-t\spc^2}$ to $U \times U$ is denoted by $k(t,x,y)$ and the restriction of the kernel of $e^{-\spc_t^2}$ by $\hat k(t,x,y)$. We have that $$\bigl((t,x) \to k(t,x,x_0)\bigr) \in \C(\bbbr^+ \times U, \Lambda V^* \ten M_m(\Ol{\mu}\A_i)) \ .$$
Set
$$r(u,t,x):=u^{n/2} \delta_u k(t,x, x_0) \ .$$

Then 
\begin{align*}
r(u,1,x) &= \sum_{k,l}u^{(n-k-l)/2}  k(u,u^{\frac 12}x,x_0)_{[k][l]} \\
&=\sum_{k,l}u^{(n-k)/2}  \hat k(u,u^{\frac 12}x,x_0)_{[k][l]} \ .
\end{align*}

Hence
$$\lim\limits_{u \to 0} r(u,1,x_0)_{[n][l]}=\lim\limits_{u \to 0}  \hat k(u,x_0,x_0)_{[n][l]} \ .$$
  
We consider $\spc^2$ as an operator on the space $\C(U, \Lambda V^* \ten M_m(\Ol{\mu}\A_i))$. Note that 
$$(\ra_t - u \delta_u \spc^2 \delta_u^{-1})r(u,t,x)=0 \ .$$ 

In the following we determine the limit of $u \delta_u \spc^2 \delta_u^{-1}$ for $u \to 0$. 

For that aim we first calculate $[\D_{\E},P\di P+ \gamma]=[\D_\E, P\di P] + c(d\gamma)$.

Let $\{e_j\}_{j=1,\dots,n}$ be an orthonormal basis of $V^*$. We denote the dual basis vector of $e_j$ by $\ra_j,~j=1, \dots , n$.

We use that there is $\alpha \in \C(U, \Lambda^1 V^* \ten \Lambda^{ev} V^* \ten M_m(\A_i))$ such that $\nabla^{\E}=PdP+\alpha$. Thus 
$${\mathcal D}_{\mathcal E}=\sum\limits_{j=1}^n c(e_j)\nabla^{\E}_{\ra_j}=\sum\limits_{j=1}^n c(e_j)P\ra_j P+ c(\alpha) \ .$$ Furthermore $P(\ra_j P)P=0$ and $P(\di P)P=0$. Then
\begin{align*}
\lefteqn{[\sum\limits_{j=1}^{n} c(e_j)P\ra_jP,P\di P]}\\
&= \sum\limits_{j=1}^{n} c(e_j)\bigl(P\ra_j P \di P -P \di P \ra_j  P\bigr) \\
&= \sum\limits_{j=1}^{n} c(e_j) \bigl( P(\ra_j P) \di P+ P\ra_j \di P    -P(\di P)\ra_j P - P\di  \ra_j P\bigr) \\
&= \sum\limits_{j=1}^{n} c(e_j) \bigl( P(\ra_j P)(\di P)+  P(\ra_j P)P \di -P(\di P)(\ra_j P) - P(\di P)P \ra_j \bigr) \\
&= \sum\limits_{j=1}^{n} c(e_j) \bigl( P(\ra_j P)(\di P) -P(\di P)(\ra_j P) )\bigr) \\
&= Pc(d P)(\di P) +P(\di P)c(d P) \ .
\end{align*}

We decompose $\alpha=\sum_{j=1}^n e_j \beta_j$ with $\beta_j \in \C(U,\Lambda^{ev} V^*\ten M_m(\A_i))$. Then
\begin{align*}
[c(\alpha),P\di P] &= \sum_{j=1}^n c(e_j) (\beta_j \di P- P \di \beta_j)\\
&= \sum_{j=1}^n c(e_j) (\beta_j (\di P)+ \beta_j \di - P(\di \beta_j)-  \beta_j \di)\\
&= \sum_{j=1}^n c(e_j) (\beta_j (\di P) - P(\di \beta_j)) \\
&= c(\alpha) (\di P) + P(\di c(\alpha)) \ .
\end{align*}

Summarizing, $$[\D_{\E},P\di P + \gamma]=Pc(d P)(\di P) +P(\di P)c(d P)+ c(d\gamma) + c(\alpha) (\di P) + P(\di c(\alpha)) \ .$$
It holds that
$$u \delta_u (P\di P + \gamma)^2 \delta_u^{-1}= (P\di P+ \gamma)^2 \ ,$$
and
$$u \delta_u [{\mathcal D}_{\mathcal E},P\di P + \gamma] \delta_u^{-1}=\sum_{j=1}^n (\ve(e_j)+ u \iota(e_j))(P(\ra_j P)(\di P) -P(\di P)(\ra_j P) +  \beta_j (\di P) - P(\di \beta_j) + \ra_j\gamma) \ ,$$
which converges for $u \to 0$ to $$Q:=P(d P)(\di P) +P(\di P)(d P)+ \alpha (\di P) + P(\di\alpha)+ d\gamma \ .$$

Denote by $R(x_0)_{ij} \in \Lambda^2V^*$ the matrix coefficients of the curvature tensor of $M$ at $x_0$ and let $K$ be the differential operator
$$K=-\sum_{i=1}^n \bigl(\ra_i - \frac 14 \sum_{j=1}^n R(x_0)_{ij}x_j\bigr)^2+F^{\E/S}(x_0) \ .$$  
From the previous calculations and \cite[Prop. 4.19]{bgv} it follows that
$$\lim\limits_{u \to 0} u \delta_u \spc^2 \delta_u^{-1} = 
K +  Q  +  (P\di P+ \gamma)^2 \ .$$

Let $\vol_M$ be the volume form of $M$. As in \cite[\S 4.3]{bgv} one concludes that $\lim\limits_{u \to 0}  (\tr_s \hat k(u,x_0,x_0))\vol_M$ (in this expression we do not apply the symbol map to $\hat k(u,x_0,x_0)$) equals 
the $\Lambda^*V^*$-homogenous component of degree $n$ of
\begin{align}
\label{heatlim} \lefteqn{(2 \pi i)^{-n/2} \Det^{1/2}\left(\frac{R/2}{\sinh(R/2)}\right)  \tr_s e^{-(F^{\E/S}+ Q + (P\di P + \gamma)^2)}|_{x_0}}\\
\nonumber &= (2 \pi i)^{-n/2} \hat A(M) \tr_s e^{-(F^{\E/S}+ Q + (P\di P + \gamma)^2)}|_{x_0} \ .
\end{align}
In order to compare $\tr_s e^{-(F^{\E/S}+ Q + (P\di P+ \gamma)^2)}$ with the Chern character of $\nabla^{\E}+ (P\di P+ \gamma)$ we calculate the curvature of $\nabla^{\E}+ (P\di P + \gamma)$:
\begin{eqnarray*}
(\nabla^{\E}+ P\di P+ \gamma)^2 &=& (\nabla^\E)^2 + [\nabla^{\E},P\di P + \gamma]+ (P\di P+ \gamma)^2\\
&=&(\nabla^{\E})^2 + [Pd P + \alpha,P\di P+ \gamma] + (P\di P + \gamma)^2\\
&=&R^{\E} + F^{\E/S} + Q + (P\di P + \gamma)^2 \ ,
\end{eqnarray*}
where the equality $$[\nabla^{\E}, P\di P + \gamma]=[Pd P + \alpha,P\di P+ \gamma]=Q$$ is proved by replacing $c(e_i)$ with $e_i$ in the above calculations. 

In particular, if $\nabla^{\E}$ is of the form $\nabla^E \ten 1 + 1 \ten \nabla^{\F}$, then 
$$\tr_s (e^{-(F^{\E/S}+ Q + (P\di P + \gamma)^2)})=\tr_s (e^{-F^{E/S}})\ch(\nabla^{\F}+ (P\di P + \gamma)) \ .$$
In general, we will be sloppy and write $\ch(\E/S)$ for the left hand side, thus suppressing the dependence on the connection.

\section{The $\eta$-form}
\label{eta}

We apply the results of the previous sections to the odd-dimensional manifold $N$ of \S \ref{condi} and construct the $\eta$-form. The construction works also if $N$ is not a boundary. 

Let $C_1$ be the Clifford algebra of $\bbbr$ with odd generator $\sigma$ with $\sigma^2=1$. Define the superconnection $$\spc^N=P\di P + \sigma \D_N$$ and the rescaled superconnection $$\spc^N_t=P \di P + \sqrt t \sigma \D_N \ .$$ Both act as odd operators on sections of the $\bbbz/2$-graded bundle $C_1 \ten \Ol{\mu}\E^N_i$.

Let $\tr_{\sigma}(a+\sigma b):= \tr(a)$ and define analogously the operator trace $\Tr_{\sigma}$.

\begin{prop}
\label{etasmallt}
The expression $t^{-1/2}\Tr_{\sigma}( {\mathcal D}_N e^{-(\spc^N_t)^2})$ converges for $t \to 0$ in $\Ol{\mu}\A_i/\ov{[\Ol{\mu}\A_i,\Ol{\mu}\A_i]}$.
\end{prop}

\begin{proof} We adapt \cite[pp. 49-50]{bc}.

Let $S^1=\bbbr/\bbbz$ endowed with the euclidian metric and let $p:S^1 \times N \to N$ be the projection.

Define a Clifford module $\hat \E_i=(\bbbc^+ \oplus \bbbc^-)  \ten (p^*\E_i^N)$ on $S^1 \times N$ with Clifford multiplication $$c(dx_1):=\left(\begin{array}{cc} 0 & -i \\ -i & 0 \end{array} \right)$$
and $c(v):=-c(dx_1)c_N(v)$ for $v \in T^*N$, and endow $\hat \E_i$ with the product $\A_i$-hermitian structure and connection. The projection $P|_N \circ p$ is denoted by $P$ again, and $\gamma|_N \circ p$ is also denoted by $\gamma$. Thus $P\di  P + \gamma$ is a connection on $\hat \E_{\infty}$ in direction of $\Ai$.  

Let $\dira_{\hat \E}$ denote the Dirac operator associated to $\hat \E$; thus $$\dira_{\hat \E}=c(dx_1)(\ra_1 - \tilde \dira_N) \ ,$$
with $\tilde \dira_N=\left(\begin{array}{cc} \dira_N & 0 \\ 0 & -\dira_N \end{array}\right)$ as in \S \ref{condi}.

Let $f\in \C(S^1)$ be a positive function such that $f(x)=x+ \frac 12$ for $x \in [\frac 14,\frac 34]$.
 
For $s >0$ we define a new metric $$g_s= s^{-1} dx_1^2 + f(x_1)^2g_N$$ on $S^1 \times N$. We write $X^* \in T^*N$ for the dual of $X \in TN$ with respect to $f(x_1)^2g_N$. As before,  $\nabla^N$ is the Levi-Civit\`a connection on $N$ with respect to $g_N$. Furthermore $P_{TN}: T(S^1 \times N) \to TN \subset T(S^1 \times N)$ denotes the projection.

Using \cite[(1.18)]{bgv} one checks that the Levi-Civit\`a connection with respect to $g_s$ fulfills $$\nabla_{\ra_1}^{g_s}= \ra_1+ \frac{f'}{2f} P_{TN} \ ,$$
and for $X \in TN$
$$\nabla_{X}^{g_s}=\nabla_X^N -s \frac{f'}{f} \ra_1\ten X^* \ .$$
Let $c_s:T^*(S^1 \times N) \to \End (\hat \E)$ denote the Clifford multiplication with respect to $g_s$. A Clifford connection $\nabla^{\hat \E}$ on $\hat \E$ with respect to $g_s$ is given by 
$$\nabla_{\ra_1}^{\hat \E}v= \ra_1v \ ,$$ 
$$\nabla_X^{\hat \E}v=\nabla_X^{\E^N}v- \frac{f'}{2f}s^{\frac 12} c_s(s^{-\frac 12}dx_1)c_s(X^*)v \ .$$ 
We write $D_N:=-c(dx_1)\tilde \dira_N$. The Dirac operator associated to $\nabla^{\hat \E}$ is
\begin{align*}
\dira_f (s)&= c(dx_1)\bigl(s^{1/2}(\ra_1 + \frac{\dim N}{2}\frac{f'}{f}) - \frac 1f \tilde \dira_N \bigr) \\
&=s^{1/2} c(dx_1)(\ra_1 + \frac{\dim N}{2}\frac{f'}{f})+ \frac 1f D_N \ .
\end{align*}
The rescaled superconnection associated to $\dira_f (s)$ is $$\spcb(s)_t:=  P \di   P +  \gamma + \sqrt t \dira_f (s)$$
with curvature
\begin{align*}
\spcb(s)_t^2&=-t s (\ra_{x_1} + \frac{\dim N}{2}\frac{f'}{f})^2\\
&\quad - tc(dx_1)s^{1/2}\frac{f'}{f^2}D_N +t f^{-2}D_N^2\\
&\quad  +\sqrt t f^{-1}[ P\di  P + \gamma,D_N] +  ( P\di  P + \gamma)^2 \ .
\end{align*}
Similar to \S \ref{supcon} we consider the action of $e^{-\spcb(s)_t^2}$ on $\C(S^1 \times N,(\Lambda^*T^*S^1 \boxtimes \End(E^N)) \ten M_p(\Ol{\mu}\A_i))$.

We introduce the coordinate $r=x_1-\frac 12$ near $x_1=\frac 12$ and rescale it: Let $\alpha \in \C(S^1 \times N,(\Lambda^jT^*S^1 \boxtimes \End(E^N)) \ten M_p(\Ol{\mu} \A_i))$. For $|r|$ small and $u>0$ define
$$\delta_u\alpha(r,x_2)=u^{-\frac{j}{2}}\alpha(u^{1/2}r,x_2) \ .$$ 

Then $$\delta_u c(dr) \delta_u^{-1}=u^{-1/2}\ve(dr) + u^{1/2}\iota(dr) \ ,$$ and we have that $$\delta_u\ra_r\delta_u^{-1}=u^{-1/2}\ra_r \ .$$  
We get that
\begin{align*} \lim_{u \to 0} \delta_u \spcb(u)_t^2 \delta_u^{-1}
&= -t\ra_r^2- t dr  D_N + t D_N^2- \sqrt t [  P\di  P +  \gamma,D_N] -  (P\di  P +  \gamma)^2 \ .
\end{align*}

We represent $\delta_u e^{- \spcb(u)_t^2}\delta_u^{-1}$ by an integral kernel on $S^1 \times S^1$ with values in the integral operators on $\Lambda^*\bbbr \ten(\bbbc^+\oplus \bbbc^-) \ten L^2(N,\Ol{\mu}\E_i^N)$. We write $\bigl(\delta_ue^{- \spcb(u)_t^2} \delta_u^{-1}\bigr)(r)$ for its value at $(\frac 12+r,\frac 12)$, and similarly for other integral operators. By $\Tr_s$ we mean in the following the supertrace of integral operators on $\Lambda^*\bbbr \ten(\bbbc^+\oplus \bbbc^-) \ten L^2(N,\Ol{\mu}\E_i^N)$.
Thus 
\begin{align*}
\lim_{u \to 0}\Tr_s\bigl(\delta_ue^{- \spcb(u)_t^2 }\delta_u^{-1}\bigr)(0)
&= \frac{1}{\sqrt{4 \pi t}}\Tr_s\bigl((1+ t  dr D_N)e^{-( P\di P +  \gamma + \sqrt t D_N)^2}\bigr)(0) \\
&=\frac{\sqrt t}{4\sqrt{ \pi}} \Tr_s \bigl(dr D_Ne^{-( P\di P +  \gamma + \sqrt t D_N)^2}\bigr)(0)
\end{align*}
Thus for some $C>0$
$$\lim_{u \to 0} \Tr_s \bigl(\delta_ue^{-\spcb(u)_t^2 }\delta_u^{-1}\bigr)(0)= C  t^{\frac 12}\Tr_{\sigma}( {\mathcal D}_N e^{-(\spc^N_t)^2}) \ .$$ 
 
Near $r=0$ for $u=1,~t \le 1$ there is an asymptotic development
$$\Tr_s \bigl(\delta_u e^{-\spcb(u)_t^2}\delta_u^{-1}\bigr)(r)\sim \sum_{j \in \bbbn_0} t^{j-n/2}k_{j}(r)$$ 
with $n=\dim S^1\times N$.

By eq. \ref{heatlim} the supertrace  $\Tr_s \bigl(\delta_u e^{-\spcb(u)_t^2} \delta_u^{-1}\bigr)(0)$ converges for $t \to 0$ and $u=1$ to the $dr$-component 
$$(2 \pi i)^{-n/2} \int_N \hat A(S^1 \times N) \ch(\E^N/S_N)|_{r=0} \ .$$ 
One checks that this expression does not involve $dr$. It follows that $k_{j}(r)=0$ if $j \le \frac n2$. This implies the assertion.
\end{proof}

Let $A$ be an adapted trivializing operator for $\D_N$ and let $\psi\in \C(\bbbr)$ be a function with $\psi(r)=0$ for $r\le 1$ and $\psi(r)=1$ for $r\ge 2$.

We define the superconnection $\spc^N(r)=P\di P + \gamma+ \sigma(\D_N +\psi(r)A)$. Its curvature is 
$$\spc^N(r)^2=(\D_N+\psi(r)A)^2 + (P\di P + \gamma)^2 - \sigma R$$ with $R=[P \di P + \gamma, \D_N +  \psi(r)A]$ bounded and $\Ol{\mu}\A_i$-linear. 

The rescaled superconnection is given by $\spc^N(r)_t=P\di P+\gamma +   \sqrt t \sigma(\D_N +\psi(r) A)$.

Set 
$$\eta(\dira_N,A):= \frac{1}{\sqrt{\pi}} \int_0^{\infty}\Tr_{\sigma} \sigma \frac{d\spc^N(t)_t}{dt} e^{-\spc^N(t)_t^2}~dt \in \Oi\Ai/\ov{[\Oi\Ai,\Oi\Ai]} \ .$$
For $t \to 0$ the convergence follows from the previous proposition, for $t \to \infty$ it is implied by the results in \S \ref{heatm}. Note that in general $\eta(\dira_N,A)$  depends  on the choice of $\psi$, $P$ and $\gamma$, although this is not reflected in the notation.

\section{Heat semigroup and superconnection on the cylinder}

\subsection{Resolvents and heat kernel} 
\label{cyl}

The discussion in this section is similar to \cite[\S\S 3.4, 4.2]{wa}.

Let $Z=\bbbr \times N$. Denote by $\E^Z$ the bundle $(\bbbc^+ \oplus \bbbc^-) \ten p^*(\E^N)$ with the Clifford module structure as in the proof of Prop. \ref{etasmallt} and let $\dira_Z$ be the Dirac operator associated to $\E^Z$.

Let $\ad$ be a selfadjoint operator adapted to $\D_N$. Recall that $\tilde \ad$ denotes the translation invariant operator $\left(\begin{array}{cc} \ad & 0 \\ 0 & -\ad \end{array}\right)$ on $L^2(Z,\Ol{\mu}\E_i^Z)$.

Fourier transform is continuous on the Fr\'echet space of rapidly decaying functions ${\mathcal S}(Z,\Ol{\mu}\E^Z_i)$ since ${\mathcal S}(\bbbr)$ is nuclear.

We fix the convention that closed operators have ${\mathcal S}(Z,\Ol{\mu}\E_i)$ as a core for their domain. In particular we take first powers, then closures. 

Let $\D_Z(r)$ be the closure of $\dira_Z- c(dx_1)r \tilde \ad$.

Denote by $\D_N+r\ad$ the closure of $\dira_N+ r\ad$, which acts on ${\mathcal S}(Z,\Ol{\mu}\E_i)$ in a translation invariant way, and denote by $\Delta_{\bbbr}$ the closure of the Laplace operator $- \ra_1^2$. Both, $(\D_N+r\ad)^2$ and $\Delta_{\bbbr}$  generate holomorphic semigroups on $L^2(Z,\Ol{\mu}\E_i)$ by the fact that $L^2(Z,\Ol{\mu}\E_i)=L^2(\bbbr,L^2(N,\Ol{\mu}\E^N_i))$. It follows that $$e^{-t(\D_N+r\ad)^2}e^{-t\Delta_{\bbbr}}=e^{-t{\mathcal D}_Z(r)^2}$$ is a holomorphic semigroup on $L^2(Z,\Ol{\mu}\E_i)$. The strong limit for $t \to 0$ is uniform in $r\in [-1,1]$.

Furthermore $e^{-t{\mathcal D}_Z(r)^2}$ acts also as a holomorphic semigroup on the spaces $C_0^k(\bbbr,\bbbc^2 \ten C^l(N,\Ol{\mu}\E^N_i)),~k,l \in \bbbn_0$, and the strong limit for $t \to 0$ is uniform in $r\in [-1,1]$ as well. Hence this also holds for $C_0^n(Z,\Ol{\mu}\E^Z_i)$. 

Eq. \ref{decM} and the subsequent remarks imply that for any $\ve>0$ there is $C>0$ such that for all $t \ge 1$ and $r \in [-1,1]$
\begin{align}
\label{estimcyl}
\|e^{-t{\mathcal D}_Z(r)^2}\|  \le Ce^{\ve t} \ , &\quad \|\frac{d}{dr}e^{-t{\mathcal D}_Z(r)^2}\|  \le Ce^{\ve t}
\end{align}
on $L^2(Z,\Ol{\mu}\E^Z_i)$ resp. on $C_0^k(Z,\Ol{\mu}\E^Z_i)$.

\begin{lem}
\label{diffopcyl}
\begin{enumerate}
\item Let $B$ be a differential operator of order $k$ with coefficients in $\C(Z,\Ol{\mu}\E^Z_i \ten_{\Ol{\mu}\A_i}(\Ol{\mu}\E^Z_i)^*)$. Let $m> \frac{\dim Z + k}{2}$. Then there is $C>0$ such that for $|r|\le 1$ and  $0 < t\le 1$  on $L^2(Z,\Ol{\mu}\E^Z_i)$
$$\|Be^{-t{\mathcal D}_Z(r)^2}\| \le Ct^{-m} \ , \quad \|\frac{d}{dr}Be^{-t{\mathcal D}_Z(r)^2}\| \le Ct^{-m} \ $$ 
$$\|e^{-t{\mathcal D}_Z(r)^2}B\| \le Ct^{-m} \ , \quad \|\frac{d}{dr}e^{-t{\mathcal D}_Z(r)^2}B\| \le Ct^{-m} \ $$ 
\item Let $B$ be a differential operator of order $1$ with coefficients in $\C(Z,\Ol{\mu}\E^Z_i \ten_{\Ol{\mu}\A_i}(\Ol{\mu}\E^Z_i)^*)$ resp. let $B=\D_N\ad+ \ad\D_N$. Then for $|r|\le 1$ and  $0 < t\le 1$ on $L^2(Z,\Ol{\mu}\E_i)$
$$\|Be^{-t{\mathcal D}_Z(r)^2}\| \le C t^{-\frac 12} \ , \quad \| \frac{d}{dr}Be^{-t{\mathcal D}_Z(r)^2}\| \le C \ ,$$
$$\|e^{-t{\mathcal D}_Z(r)^2B}\| \le C t^{-\frac 12} \ , \quad \| \frac{d}{dr}e^{-t{\mathcal D}_Z(r)^2}B\| \le C \ .$$ 
For $B=\D_N\ad+\ad\D_N$ these estimates hold also on $C_0^k(Z,\Oei)$. 
\end{enumerate}
\end{lem}

\begin{proof}
Without loss of generality we may restrict to the case where $B$ is translation invariant.
 
We decompose $B=\sum_j B_j (i \ra_{x_1})^j$ with $B_j$ a differential operator on $\C(N,\Ol{\mu}\E^N_i)$ of order at most $k-j$ and write
$$B_j(i \ra_{x_1})^j e^{-t{\mathcal D}_Z(r)^2}=B_j e^{-t(\D_N+ r\ad)^2}  (i \ra_{x_1})^j e^{-t\Delta_{\bbbr}} \ .$$
Now the assertion follows since by Prop. \ref{semigroupm} (3) for $|r|\le 1$ and  $0 < t\le 1$
$$\|B_j e^{-t(\D_N+ r\ad)^2}\| \le C t^{-m+\frac j2} \ , \quad \|\frac{d}{dr}B_j e^{-t(\D_N+ r\ad)^2}\| \le C t^{-m+ \frac j2}$$ and since furthermore
$$\|(i \ra_1)^j e^{-t\Delta_{\bbbr}}\| \le  C t^{-\frac{j}{2}} \ . $$

The second claim follows analogously by using Prop. \ref{semigroupm} (4).
\end{proof}

\begin{prop}
\label{cylresolv}
Let $\lambda \in \bbbc$ with $\re \lambda^2 < 0$.
Then ${\mathcal D}_Z(r)-\lambda$ has a bounded inverse on ${\mathcal S}(Z,\Ol{\mu}\E^Z_i)$ as well as on $L^2(Z,\Ol{\mu}\E^Z_i)$. On $L^2(Z,\Ol{\mu}\E^Z_i)$ the inverse is of class $C^1$ in $r$ (with respect to the norm topology). Furthermore its derivative is bounded on ${\mathcal S}(Z,\Ol{\mu}\E^Z_i)$. If  $f \in {\mathcal S}(Z,\Ol{\mu}\E^Z_i)$, then $\bigl(r \mapsto ({\mathcal D}_Z(r)-\lambda)^{-1}f\bigr) \in C^1([-1,1],{\mathcal S}(Z,\Ol{\mu}\E^Z_i))$.
\end{prop}

\begin{proof}
We first consider the action on ${\mathcal S}(Z,\Ol{\mu}\E^Z_i)$.

Fourier transform intertwines the operator ${\mathcal D}_Z(r)-\lambda=c(dx_1)(\ra_1 - (\tilde \D_N + r\tilde \ad))-\lambda$ with the operator 
$c(dx_1)(i x_1- (\tilde \D_N + r\tilde \ad))-\lambda$. It is enough to study the inverse of  $(-\lambda^2+x_1^2+(\D_N+ r\ad)^2)$.

The resolvent $(-\lambda^2+x_1^2+(\D_N+ r\ad)^2)^{-1}$ exists on $C^k(N,\Ol{\mu}\E_i^N)$ by the results in \S \ref{closed} and is bounded by $C(|x_1|+1)^{-2}$, with $C$ independent of $r \in [-1,1]$. The bound follows from the standard estimate for resolvents of generators of holomorphic semigroups. Thus $(-\lambda^2+x_1^2+(\D_N+ r\ad)^2)^{-1}$ defines a bounded operator on ${\mathcal S}(Z,\Ol{\mu}\E_i)$. The usual resolvent formula implies that $(-\lambda^2+x_1^2+(\D_N+ r\ad)^2)^{-1}f$ is continuous in $r$ and that  
$$\frac{d}{dr}(-\lambda^2+x_1^2+(\D_N+ r\ad)^2)^{-1}f$$
$$=-(-\lambda^2+x_1^2+(\D_N+ r\ad)^2)^{-1}(\ad \D_N + \D_N\ad+2r\ad^2)(-\lambda^2+x_1^2+(\D_N+ r\ad)^2)^{-1}f \ ,$$
which is also continuous in $r$.

On $L^2(Z,\Ol{\mu}\E_i)$ the integral
$$({\mathcal D}_Z(r)-\lambda)^{-1}=\frac{1}{(k-1)!} \int_0^{\infty}({\mathcal D}_Z(r)+\lambda)e^{-t({\mathcal D}_Z(r)^2-\lambda^2)} ~ dt$$
converges as a bounded operator. It is $C^1$ in $r$ by the estimates of the previous Lemma. 
\end{proof}

It follows that the powers of $\D_Z(r)-\lambda$ with $\re \lambda^2 < 0$ are closed.

\begin{prop}
\label{cylregul}
Let $m \in \bbbn,k \in \bbbn_0$ with $m >\frac{\dim Z + k+1}{2}$.
Let $\lambda \in \bbbc$ with $\re \lambda < 0$.

The operator $({\mathcal D}_Z(r)^2-\lambda)^{-m}$ is a bounded operator from $L^2(Z,\Ol{\mu}\E_i)$ to $C^k(Z,\Ol{\mu}\E_i)$ that is of class $C^1$ in $r$.
\end{prop}

\begin{proof}
For a differential operator $B= \sum_j B_j(i\ra_1)^j$ of order smaller or equal to $k$, where the $B_j$ are differential operators on $\C(N,\Ol{\mu}\E_i)$ of order at most $k-j$, the operator $B({\mathcal D}_Z(r)^2-\lambda)^{-m}$ is well-defined on ${\mathcal S}(Z,\Ol{\mu}\E_i)$ and 
$$B({\mathcal D}_Z(r)^2-\lambda)^{-m}=\frac{1}{(m-1)!} \int_0^{\infty} t^{m-1} \sum_j B_j e^{-t(\D_N+ r\ad)^2}(i\ra_1)^je^{-t\Delta_{\bbbr}}e^{\lambda t}dt \ .$$
The following estimates hold for small $t$ and $r \in [-1,1]$, where the norm is the operator norm for bounded operators from $L^2(Z, \Ol{\mu}\E_i)$ to $C(Z,\Ol{\mu}\E_i)$:
$$\|B_je^{-t(\D_N+ r\ad)^2}\|_{L^2 \to C} < Ct^{-m-1+\frac{j}{2}} \ , \quad \|B_j\frac{d}{dr}e^{-t(\D_N+ r\ad)^2}\|_{L^2 \to C} < Ct^{-m-1 +\frac{j}{2}} \ ,$$   
$$\|(i\ra)^je^{-t\Delta_{\bbbr} }\|_{L^2 \to C} \le Ct^{-\frac{j+1}{2}} \ .$$
Hence the integral converges as a bounded operator from $L^2(Z, \Ol{\mu}\E_i)$ to $C(Z,\Ol{\mu}\E_i)$ and is $C^1$ in $r$.
\end{proof}

\begin{prop}
\label{cyloffdiagsmooth} Let $k \in \bbbn$ and $\lambda \in \bbbc$ with $\re \lambda < 0$.

Let $\phi,\psi \in \C(\bbbr)$ with disjoint support and with $\phi$ (resp. $\psi$) compactly supported. Then $\phi ({\mathcal D}_Z(r)^2-\lambda)^{-k}\psi$ is an integral operator with integral kernel in $C^1([-1,1],{\mathcal S}(Z \times Z,\E^Z_i \boxtimes_{\A_i} (\E^Z_i)^*))$.
\end{prop}

\begin{proof}
We only prove the case where $\psi$ is compactly supported. The other case is analogous. We use induction. First assume that $k=1$.

For $T>0$ the integral kernel of $$\phi \int_0^T e^{-t(\D_N+ r\ad)^2}  e^{-t\Delta_{\bbbr}}e^{t\lambda} ~dt~ \psi$$ 
is in $C^1([-1,1],{\mathcal S}(Z \times Z,\E^Z_i \boxtimes_{\A_i} (\E^Z_i)^*))$. 

Furthermore $$\phi \int_T^{\infty} e^{-t(\D_N+ r\ad)^2}  e^{-t\Delta_{\bbbr}}e^{t\lambda} ~dt~ \psi=\phi ({\mathcal D}_Z(r)^2-\lambda)^{-1}e^{-T(\D_N+ r\ad)^2}  e^{-T\Delta_{\bbbr}}e^{T\lambda} \psi \ ,$$
and the integral kernel of $e^{-T(\D_N+ r\ad)^2}  e^{-T\Delta_{\bbbr}}e^{T\lambda} \psi$ is in $C^1([-1,1],{\mathcal S}(Z \times Z,\E^Z_i \boxtimes_{\A_i} (\E^Z_i)^*))$. We conclude from Prop. \ref{cylresolv} that the integral kernal of $\phi \int_T^{\infty} e^{-t(\D_N+ r\ad)^2}  e^{-t\Delta_{\bbbr}}e^{-t\lambda} ~dt~ \psi$ is in $C^1([-1,1],{\mathcal S}(Z \times Z,\E^Z_i \boxtimes_{\A_i} (\E^Z_i)^*))$.

For $k>1$ let $\xi \in \C_c(\bbbr)$ be such that $\supp(1- \xi) \cap \supp \psi =\emptyset$ and $\supp\xi \cap \supp \phi =\emptyset$. By induction and by Prop. \ref{cylresolv} each of the terms on the right hand side of the following equation is an integral operator with integral kernel in $C^1([-1,1],{\mathcal S}(Z \times Z,\E^Z_i \boxtimes_{\A_i} (\E^Z_i)^*))$:
\begin{align*}
\phi ({\mathcal D}_Z(r)^2-\lambda)^{-k}\psi&= \phi ({\mathcal D}_Z(r)^2-\lambda)^{-k+1}\xi ({\mathcal D}_Z(r)^2-\lambda)^{-1}\psi\\
&+ \phi ({\mathcal D}_Z(r)^2-\lambda)^{-k+1}(1-\xi) ({\mathcal D}_Z(r)^2-\lambda)^{-1}\psi \ .
\end{align*}
\end{proof}

If $\ad$ is an adapted trivializing operator for $\D_N$ then there is $\omega >0$ such that on $L^2(Z,\Ol{\mu}\E^Z_i)$
$$\|e^{-t(\D_N+ \ad)^2}\| \le C e^{-\omega t} \ .$$

Analogues of the previous three propositions hold for the resolvents $({\mathcal D}_Z(1)^2-\lambda)^{-1}$ with $\re \lambda < \omega$ resp. $({\mathcal D}_Z(1)-\lambda)^{-1}$ with $\re \lambda^2 <\omega$.

\subsection{The superconnection}

A superconnection associated to $\D_Z(r)$ is defined by $$\spc^Z(r):= P \di P + \gamma +  {\mathcal D}_Z(r)=P\di P + \gamma + c(dx_1)\ra_1 -c(dx_1)(\tilde \D_N + r \tilde \ad) \ ,$$
and the corresponding rescaled superconnection by
$\spc^Z(r)_t:= P \di P + \gamma + \sqrt t {\mathcal D}_Z(r)$
It holds that
\begin{align*}
\spc^Z(r)^2 &={\mathcal D}_Z(r)^2 + (P\di P + \gamma)^2 - c(dx_1)[P\di P + \gamma, \tilde \D_N+ r\tilde \ad] \\
&= \Delta_{\bbbr} +  (P \di P + \gamma)^2 + (\D_N + r\ad)^2 - c(dx_1) \tilde R \ .
\end{align*}

Note that $c(dx_1)\tilde R=\tilde Rc(dx_1)$.

There is a linear map $$\Li:a+\sigma b \mapsto a+ c(dx_1) \tilde b \ ,$$
with $\sigma$ as in \S\ref{eta}.

Here, for example, $a,b$ are homomorphisms on $\Ol{\mu}\E^N_i$ and the image is a homomorphism on $\Ol{\mu}\E^Z_i$. If $a,b$ are operators on $L^2(N,\Ol{\mu}\E^N_i)$, then $\Li(a+\sigma b)$ is a translation invariant operator on $L^2(Z,\Ol{\mu}\E^Z_i)$. Note that for homomorphisms the supertrace vanishes on the image of $\Li$.

Using Volterra development with respect to $c(dx_1) \tilde R$ one verifies that 
\begin{align}
\label{cyleq1}
e^{-t\spc^Z(r)^2}&=e^{-t \Delta_{\bbbr}} \Li(e^{-t\spc^N(r)^2}) \ .
\end{align}

We conclude from Prop. \ref{intkersupcon}:
\begin{prop}
\label{kersupconcyl} 
The operator $e^{-t\spc^Z(r)^2}$ is an integral operator with smooth integral kernel $p^Z(r)_t$ depending smoothly on $(r,t)$. For any $\ve_1,\ve_2>0$ and $c>4$ there is $C>0$ such that for $|x_1-y_1| >\ve_2$ for all $t>0$ and $r \in [-1,1]$
$$|p^Z(r)_t(x,y)| \le Ce^{\ve_1 t-\frac{(x_1-y_1)^2}{ct}} \ .$$
Analogous estimates hold for the first derivatives with respect to $r,t$ and the partial derivatives in $x,y$.
\end{prop}

The pointwise supertrace of the integral kernel of $e^{-t\spc^Z(r)^2}$ vanishes. The same holds true for the composition of $e^{-t\spc^Z(r)^2}$ with even differential operators commuting with $c(dx_1)$. Moreover the pointwise supertrace of the integral kernel of $\D_Z(r)e^{-t\spc^Z(r)^2}$ vanishes on the diagonal. Analogous statements hold for the integral kernel of the rescaled operator $e^{-\spc^Z(r)_t^2}$. 

For later use note also that for homomorphisms $a,b,c$ on $\Ol{\mu}\E^N_i$
\begin{align}
\label{cyleq2}
\tr_s \tilde c \Li(a+ \sigma b)&=\tr_s(\tilde ca + c(dx_1)cb)=2 \tr(ca)=2\tr_{\sigma}c(a+\sigma b) \ .
\end{align}

\section{Resolvents and heat semigroup on $M$}
\label{cutpaste}

Now we return to the situation introduced in \S \ref{condi} and \S\ref{DirHilb}. In particular, $M$ is the union of a manifold with cylindric end and an isolated point. We assume that $A$ is an adapted trivializing operator for $\D_N$. Recall $\chi, \Ki$ from \S \ref{DirHilb} and the function $\psi$ defined in the end of \S \ref{eta}.

We set
$$\D(r,\rho):={\mathcal D}_{\mathcal E} - \psi(r)c(dx_1)\chi\tilde A+\rho \Ki \ .$$

For the definition of a parametrix and for the study of the heat semigroup patching arguments will be used, which we prepare now.

We assume $\N \in \bbbn$ defined in \S \ref{DirHilb} large enough for the constructions in the following sections to work.

Let $d\ge 4$ be such that the support of the integral kernel of $\Ki$ is contained in $(M_c \cup \{x \in Z^+~|~x_1 < d-1\})^2$.

Define an open covering ${\mathcal U}:= (U_0,U_1,U_2)$ of $M$ by
\begin{itemize}
\item $U_0:=M_c \cup \{x \in Z^+~|~x_1 < \frac 12\} \cup *$ \ ,
\item $U_1:=\{x \in Z^+~|~ x_1 \in (0,d)\}$ \ , 
\item $U_2:=\{x \in Z^+~|~ x_1> d-1\}$ \ .
\end{itemize}
  
We have that $\supp \chi \cap \ov{U_0}=\emptyset$, $\supp \chi' \subset U_1$ and $\supp \chi' \cap \ov{U_2}= \emptyset$.
Furthermore the intersection of the support of the integral kernel of $\Ki$ with $M \times U_2 \cup U_2 \times M$ is empty.

Let $(\phi_0,\phi_1,\phi_2)$ be a smooth partition of unity subordinate to ${\mathcal U}$ and let $(\zeta_0,\zeta_1,\zeta_2)$ with $\zeta_j \in \C(M)$ be such that $\supp \zeta_j \subset U_j$ and $\supp(1-\zeta_j) \cap \supp \phi_j=\emptyset$. Note that the derivatives of $\phi_j,\zeta_j$ are supported on the cylindric end. We assume that $\phi_j,\zeta_j$ only depend on the variable $x_1$ on the cylindric end.  

Let $W$ be the union of a closed manifold and an isolated point $*_W$ and let $E_W$ be a $\bbbz/2$-graded Clifford module on $W$. Assume  that there is an isomorphism $E|_{U_0} \to
E_W$ of Clifford modules whose base map is an isometric embedding mapping $*$ to $*_W$. We identify $E|_{U_0}$ with its image in $E_W$. Let $P_0 \in \C(W,M_p(\Ai))$ be a projection such that $P_0|_{U_0}=P|_{U_0}$. We set $\E^W=E^W \ten P_0(W \times \A^p)$ and choose a Clifford connection on $\E^W$ which agrees on $U_0$ with  $\nabla^{\E}$.  Let $\D_0(r):=\D_W$ be the
Dirac operator associated to $\E^W$. \\
Furthermore choose a connection $P_0\di P_0 + \gamma_0$ on $\Ei^W$ in direction of $\Ai$ with $\gamma_0|_{U_0}=\gamma|_{U_0}$. Furthermore we assume that $\gamma_0$ supercommutes with Clifford multiplication and that $\gamma_0=-\gamma_0^*$.   
 
Let $\xi \in \C_c(U_1)$ be such that $\xi$ equals $\chi$ on a neighbourhood of $\supp \zeta_1$. Let $S_{d+1}=\bbbr/(d+1)\bbbz$ be the circle with circumference $d+1$. We define $Y$ as the union of $S_{d+1}\times N$ and an isolated point $*_Y$. Let $p:Y\setminus *_Y \to N$ be the projection and let $\E^Y_{Y \setminus *_Y}=(\bbbc^+ \oplus \bbbc^-)\ten p^*\E^N$ and $\E^Y(*_Y) = \E(*)$. We endow $(\bbbc^+ \oplus \bbbc^-) \ten p^*\E^N$ with the product Clifford module structure and the Clifford connection of product type induced by the connection $\nabla^{\E^N}$. Furthermore we identify $\E|_{U_1}$ with the restriction of $\E^Y$ to $(0,d) \times N \subset Y$.\\
Let $\D_Y$ be the Dirac operator associated to $\E^Y$ (which is assumed to vanish on the isolated point). Then the operator 
$$\D_1(r):={\mathcal D}_Y - \psi(r) c(dx_1)\xi \tilde A$$ 
is a closed operator on $L^2(Y,\Ol{\mu}\E^Y_i)$. A straightforward calculation exploiting the product structure shows that $c(dx_1)\xi \tilde A$ is adapted to $\D_Y$. Hence the results of \S \ref{closed} apply.\\
We define $P_1,\gamma_1$ by $P_1(x)=P(p(x)), \gamma_1(x)=\gamma(p(x))$ for $x\in Y \setminus *_Y$ and $P_1(*_Y)=P(*), \gamma_1(*_Y)=\gamma(*)$ and thus get a connection $P_1 \di P_1 + \gamma_1$ on $\Ei^Y$ in direction of $\Ai$.

Furthermore we set
$$\D_2(r):={\D}_Z - \psi(r)c(dx_1)\tilde A$$
and $P_2(x)=P(p(x))$, $\gamma_2(x)=\gamma(p(x))$. 
Here $p:Z \to N$ is the projection.

\subsection{The resolvents}
\label{res}

In the following we define a parametrix of $\D(r,\rho)-\lambda$, which will be used to apply the results of \S \ref{regpar}.

By Prop. \ref{Mresolv} for $\N$ large enough there is a symmetric integral operator $\Kappa_0$ with integral kernel in $\C(W \times W,\E^W_i \boxtimes_{\A_i} (\E_i^W)^*)$ and $\omega_0>0$ such that $$Q^0_{\lambda}:=(\D_0(2)+ \Kappa_0 -\lambda)^{-1}$$ 
exists for $ \re\lambda^2 < \omega_0$. 

Similarly, assuming $\N$ large enough there is $\omega_1>0$ and an symmetric integral operator $\Kappa_1$ with integral kernel in $\C(Y \times Y,\E^Y_i \boxtimes (\E_i^Y)^*)$ such that $Q^1_{\lambda}:=(\D_1(2) + \Kappa_1 -\lambda)^{-1}$ exists for $\re\lambda ^2 < \omega_1$.  
 
Furthermore we can choose $\omega_2>0$ such that $Q^2_{\lambda}=(\D_2(2)-\lambda)^{-1}$ is well-defined for $\re \lambda^2 < \omega_2$.

For $\re \lambda^2 < \min(\omega_0,\omega_1,\omega_2)$ we set
$Q_{\lambda}:=\sum_{j=0}^2\phi_j Q^j_{\lambda} \zeta_j$.

In order to study the $r$-dependence we define 
$Q^j_{\lambda}(r)=(\D_j(r) -\lambda)^{-1}, j=0,1,2$ for $\re \lambda^2 < 0$ 
and set
$$Q_{\lambda}(r):=\sum_{j=0}^2\phi_j Q^j_{\lambda}(r) \zeta_j \ .$$ 

We verify that $Q_{\lambda}(r)$ is a regular left parametrix of $\D(r,\rho) - \lambda$ in the sense of \S \ref{regpar}.

We denote equality up to integral operators with integral kernel in $C^1(\bbbr \times [0,1],{\mathcal S}(M \times M,\E_i \boxtimes_{\A_i} \E_i^*))$ by $\sim$.

We have that
\begin{align*} 
\lefteqn{Q_{\lambda}(r)(\D(r,\rho) - \lambda)-1}\\
&\sim   Q_{\lambda}(r)(\D(r,0) - \lambda)-1\\
&=\sum_{j=0}^2 \phi_j Q^j_{\lambda}(r)(\zeta_j\D(r,0)-\D_j(r)\zeta_j) +\sum_{j=0}^2\phi_j \bigl(Q^j_{\lambda}(r)(\D_j(r)-\lambda)-1 \bigr) \zeta_j\\
&\sim -\sum_{j=0}^2 \phi_j Q^j_{\lambda}(r)c(d\zeta_j) \\
&\sim 0 \ .
\end{align*}
The last equivalence holds by the off-diagonal properties of the resolvents, which follow from Prop. \ref{offdiag} and Prop. \ref{cyloffdiagsmooth} (see the remark after Prop. \ref{offdiag}).

It follows that $Q_{\lambda}(r)$ is a regular left parametrix of $\D(r,\rho)-\lambda$ as defined in \S \ref{regpar}.

Analogously for $\re \lambda^2 < \min(\omega_0,\omega_1,\omega_2)$ and $r\ge2$ the operator $Q_{\lambda}$ is a regular left parametrix of $\D(r,\rho)-\lambda$.

We also need parametrices of higher order. We only carry out the construction in the case $\re \lambda^2<0$, the case $\re \lambda^2 < \min(\omega_0,\omega_1,\omega_2)$ and $r\ge 2$ is similar.

Set
$$R^m_{\lambda}(r):=\sum_{j=0}^2\phi_j (\D_j(r)^2 -\lambda^2)^{-m} \zeta_j \ .$$ 

\begin{lem}
The operator $R^m_{\lambda}(r)$ is a regular left parametrix of $(\D(r,\rho)^2-\lambda^2)^m$.
\end{lem}

\begin{proof}
We have
\begin{align*}
\lefteqn{\sum_{j=0}^2\phi_j (\D_j(r)^2 -\lambda^2)^{-m} \zeta_j (\D(r,\rho)^2-\lambda^2)^m}\\
&\sim  \sum_{j=0}^2\phi_j(\D_j(r)^2 -\lambda^2)^{-m+1}(\D_j(r) +\lambda)^{-1}\zeta_j(\D(r,0) +\lambda)(\D(r,0)^2-\lambda^2)^{m-1}\\
& \quad + \sum_{j=0}^2\phi_j(\D_j(r)^2 -\lambda^2)^{-m}(\zeta_j \D(r,0)-\D_j(r)\zeta_j)  (\D(r,0) +\lambda)(\D(r,0)^2-\lambda^2)^{m-1} \ .
\end{align*}

By induction $\sum_{j=0}^2\phi_j(\D_j(r)^2 -\lambda^2)^{-m+1}(\D_j(r) +\lambda)^{-1}\zeta_j$ is a regular left parametrix of $(\D(r,0) +\lambda)(\D(r,0)^2-\lambda^2)^{m-1}$. This shows that the first term is equivalent to 1. In the second term $\phi_j(\D_j(r)^2 -\lambda^2)^{-m}(\zeta_j \D(r,0)-\D_j(r)\zeta_j)$ is equivalent to zero by the off-diagonal behavior of $(\D_j(r)^2 -\lambda^2)^{-m}$.
\end{proof}

\begin{prop}
\label{smooth}
Assume that $\D(r,\rho)-\lambda$ has a bounded inverse on $L^2(M,\E)$. 

\begin{enumerate}
\item If $\re \lambda^2 < 0$, then $\D(r,\rho)-\lambda$ has a bounded inverse on  $L^2(M,\Ol{\mu}\E_i)$ that is of class $C^1$ in $(r,\rho)$. The inverse also acts continuously on ${\mathcal S}(M,\Ol{\mu}\E_i)$ and for $f \in {\mathcal S}(M,\Ol{\mu}\E_i)$ it holds that
$((r,\rho) \mapsto (\D(r,\rho)-\lambda)^{-1}f) \in C^1(\bbbr \times [-1,1],{\mathcal S}(M,\Ol{\mu}\E_i))$.
\item Let $\re \lambda^2 < 0$ and let $m,k \in \bbbn$ with $m >\frac{\dim M + k+1}{2}$.
Then $(\D(r,\rho)^2-\lambda^2)^{-m}$ maps
 $L^2(M,\Ol{\mu}\E_i)$ continuously to $C^k(M, \Ol{\mu}\E_i)$ and is of class $C^1$ in $(r,\rho)$. 
 
\item Analogues of (1) and (2) hold for $r\ge 2$ and $\re \lambda^2 < \min(\omega_0,\omega_1,\omega_2)$, 
\end{enumerate}
\end{prop}

\begin{cor}
If $\rho \neq 0$ and $r \ge 2$, then there is $\omega>0$ such that the operator $\D(r,\rho)-\lambda$ has a bounded inverse on  $L^2(M,\Ol{\mu}\E_i)$ for $\lambda \neq 0$ and $\re \lambda^2 < \omega$.
\end{cor}

\begin{proof}
Since the selfadjoint operator $\D(r,\rho)$ has closed range for $\rho \neq 0$ and $r \ge 2$, there is $\omega>0$ such that the operator $\D(r,\rho)-\lambda$ has a bounded inverse on  $L^2(M,\E)$ for $\lambda \neq 0$ and $\re \lambda^2 < \omega$.
\end{proof}

\begin{cor}
Let $\rho \neq 0$ and $r \ge 2$.
The kernel of $\D(r,\rho)$ on $L^2(M,\Ol{\mu}\E_i)$ consists of elements of ${\mathcal S}(M,\Ol{\mu}\E_i)$. 
\end{cor}

\begin{proof}
Let $f$ be an element of the kernel of $\D(r,\rho)$ on $L^2(M,\Ol{\mu}\E_i)$. Then by the previous proposition for $m >\frac{\dim M + k+1}{2}$
$$f=(\D(r,\rho)^2+1)^{-m}f \in C^{k}(M,\Ol{\mu}\E_i) \ .$$ Thus $f \in \C(M,\Ol{\mu}\E_i)$. This implies that 
$$\D_2(r)\phi_2 f=(\D_2(r)\phi_2-\phi_2\D(r,\rho))f= c(d\phi_2)f \in {\mathcal S}(Z,\Ol{\mu}\E^Z_i) \ .$$ 
Hence by Prop. \ref{cylresolv} 
$$\phi_2f =\D_2(r)^{-1}c(d\phi_2)f \in {\mathcal S}(Z,\Ol{\mu}\E^Z_i) \ .$$ Thus $f \in {\mathcal S}(M,\Ol{\mu}\E_i)$.
\end{proof}

For $\rho \neq 0$ and $r\ge 2$ let $\Pj$ be the projection onto the kernel of $\D(r,\rho)$ on $L^2(M,\E)$. It exists since the range of $\D(r,\rho)$ is closed. Furthermore $\D(r,\rho)+\Pj$ is invertible on $L^2(M,\E)$.

\begin{prop}
\label{kerhs}
Let $\rho \neq 0$ and $r \ge 2$.
\begin{enumerate}
\item It holds that $\Pj=\sum_{j=1}^k g_j h_j^*$ for appropriate elements $g_j,h_j \in {\mathcal S}(M,\E_{\infty})$ in the kernel of ${\mathcal D}(r,\rho)$. 
\item An analogue of Prop. \ref{smooth} holds for $\D(r,\rho)+\Pj$. In particular the operator $\D(r,\rho)+\Pj$ has a bounded inverse on  $L^2(M,\Ol{\mu}\E_i)$.
\end{enumerate}
\end{prop}  

\begin{proof}
By Prop. \ref{smooth} zero is an isolated point in the spectrum of $\D(r,\rho)$ on $L^2(M,\Oei)$ for $\rho \neq 0$ and $r \ge 2$. Hence for $\ve>0$ small enough
$${\mathcal P}=\frac{1}{2 \pi i}\int_{|\lambda|=\ve} (\D(r,\rho)-\lambda)^{-1} d \lambda$$ is well-defined and bounded on $L^2(M,\Oei)$. 

For $\mu=0$ we have that $\Ran_{\infty} {\mathcal P} \subset {\mathcal S}(M,\E_{\infty})$ by the previous lemma. (Here we use the notation of Prop. \ref{projker}). Since ${\mathcal S}(M,\Ai) \subset L^2(M,\bbbc) \ten_{\pi} \Ai$, and since the bundle $\Ei$ can be isometrically embedded in a trivial bundle, the first assertion follows from Prop. \ref{projker}. The proof of the second assertion is as the proof of Prop. \ref{smooth}. Note that $Q_{\lambda}$ is also a left regular parametrix of  $\D(r,\rho)+\Pj-\lambda$.
\end{proof}

\subsection{The heat semigroup}

We prove the existence of the heat semigroup on $C_0(\bbbr  \times [-1,1],L^2(M,\Oei))$ generated by $-\D(r,\rho)^2$ (considered as a family parametrized by $(r,\rho) \in \bbbr \times [-1,1]$) by defining an approximation and then using Volterra development as in \cite[\S 2.4]{bgv}.

We set $$E(r)_t:= \sum_{j=0}^2\zeta_j e^{-t\D_j(r)^2} \phi_j \ .$$

\begin{prop} 
\label{heatsemest}
Let $m,k \in \bbbn_0$ with $m >\frac{\dim M + k+1}{2}$. 
\begin{enumerate}
\item
The operator $-\D(r,\rho)^2$  generates a holomorphic semigroup on 
$C_0(\bbbr\times [-1,1],L^2(M,\Ol{\mu}\E_i))$. For $t>0$ the operator $e^{-t\D(r,\rho)^2}$ is of class $C^1$ in $(r,\rho,t)$ on $L^2(M,\Oei)$.
\item
For $t>0$ the operator $e^{-t\D(r,\rho)^2}$ is a continuous operator from $C_c^{2m}(M,\Ol{\mu}\E_i)$ to $C^k(M, \Ol{\mu}\E_i)$, which is of class $C^1$ in $(r,\rho,t)$.
\item Let $B$ be a first order differential operator. On $C_0(\bbbr \times [-1,1],L^2(M,\Ol{\mu}\E_i))$ for $0<t<1$ 
$$\|Be^{-t\D(r,\rho)^2}\| \le Ct^{-\frac 12}, \quad \|e^{-t\D(r,\rho)^2}B\| \le Ct^{- \frac 12}\ ,$$
$$\|\frac{d}{dr} Be^{-t\D(r,\rho)^2}\| \le C, \quad \|\frac{d}{dr}e^{-t\D(r,\rho)^2}B\| \le C\ .$$
\item
For $\ve>0$ there is $C>0$ such that for all $t \ge 0, r \in \bbbr, \rho \in [-1,1]$
$$\|e^{-t\D(r,\rho)^2}\| \le Ce^{\ve t} \ ,\quad \|\frac{d}{dr}e^{-t\D(r,\rho)^2}\| \le Ct^{1/2}e^{\ve t}  $$ 
as an operator on $L^2(M,\Ol{\mu}\E_i)$ resp. as an operator from $C_c^{2m}(M,\Ol{\mu}\E_i)$ to $C^k(M, \Ol{\mu}\E_i)$.  
\item
Let $\rho \neq 0$ and $r \ge 2$. Then the estimate in (4) holds with $\ve =0$. Let $\Pj$ be the projection onto the kernel of $\D(r,\rho)$. There is $\omega>0$ and $C>0$ such that for all $t\ge 0$
$$\|e^{-t(\D(r,\rho)+\Pj)^2}\| \le Ce^{-\omega t}$$
as an operator on $L^2(M,\Ol{\mu}\E_i)$ resp. as an  operator from $C_c^{2m}(M,\Ol{\mu}\E_i)$ to $C^k(M, \Ol{\mu}\E_i)$. 
\end{enumerate}
\end{prop}

\begin{proof}
For $t>0$ let 
\begin{align*}
R(r,\rho)_t &:=(\frac{d}{dt}+ \D(r,\rho)^2)E(r)_t\\
&=\sum_{j=0}^2[\D_{\E},c(d\zeta_j)]e^{-t\D_j(r)^2}\phi_j +\bigl(\rho [\Ki, \D(r,0)]+ \rho^2 \Ki^2\bigr)E(r)_t 
\end{align*}
and set $R(r,\rho)_0=\rho [\Ki, \D(r,0)] + \rho^2 \Ki^2$.

Note that $R(r,\rho)_t$ is a bounded operator from $L^2(M,\Oei)$ to $C^k_c(M,\Oei)$ for any $k \in \bbbn$ and is of class $C^1$ in $(r,\rho,t)$. 

Thus (see \cite[\S 2.4]{bgv}) the series $\sum_{n=0}^{\infty}(-1)^nQ^n(r,\rho)_t$ with
$$Q^n(r,\rho)_t:=\int_{t\Delta_n}E(r)_{t-u_1}R(r,\rho)_{u_1-u_2} \dots R(r,\rho)_{u_{n-1}-u_n}R(r,\rho)_{u_n}du_1 \dots du_n$$
converges and defines the strongly continuous semigroup $e^{-t\D(r,\rho)^2}$ on $L^2(M,\Oei)$. (The semigroup is unique since on $L^2(M,\E_i)$ it has to agree with the one defined on the Hilbert $\A$-module $L^2(M,\E)$ by the functional calculus for selfadjoint operators.) One checks that for $t>0$ the limit is of class $C^1$ in $(r,\rho,t)$. 

By a similar calculation as above for $0<t<1$ on $L^2(M,\Oei)$
$$\|\frac{d}{dt} E(r)_t+E(r)_t\D(r,\rho)^2\| \le Ct^{-\frac 12} \ .$$ Thus, since $R(r,\rho)_t$ is uniformly bounded for small $t$ 
$$\|\D(r,\rho)^2E(r)_t+E(r)_t\D(r,\rho)^2\|\le Ct^{-\frac 12} \ .$$ 
Using that $\D(r,\rho)^2R(r,\rho)_t$ is uniformly bounded as well for $t$ small we get that $\|\D(r,\rho)^2Q^n(r,\rho)_t\| \le C^nt^{n-\frac 12}$ for small $t$ and $n\ge 1$. Since $\|\D(r,\rho)^2E(r)_t\| \le Ct^{-1}$, this implies that $\|\D(r,\rho)^2 e^{-\D(r,\rho)^2}\| \le Ct^{-1}$. Hence the semigroup extends to a holomorphic one, see Prop. \ref{exthol}.

This implies (1). The first two estimates in (3) follow since $\|BE(r)_t\| \le Ct^{-\frac 12}$ and $\|E(r)_tB\| \le Ct^{-\frac 12}$ by Prop. \ref{semigroupm} and Prop. \ref{diffopcyl}.

(2) follows from Prop. \ref{smooth} and 
$$e^{-t\D(r,\rho)^2}=(\D(r,\rho)^2+ \lambda)^{-m}e^{-t\D(r,\rho)^2}(\D(r,\rho)^2+ \lambda)^m \ .$$

From the above series expansion one infers that $e^{-t\D(r,\rho)^2}$ preserves the space ${\mathcal S}(M,\Oei)$. Thus we can apply Duhamel's formula (Prop. \ref{duhform}) and get
$$\frac{d}{dr}e^{-t\D(r,\rho)^2}=-\int_0^t e^{-s\D(r,\rho)^2}(\frac{d}{dr}\D(r,\rho)^2)e^{-(t-s)\D(r,\rho)^2}~ds \ .$$

The first estimate of (4) follows from Prop. \ref{smooth} and Prop. \ref{specsem}, the second from the first in combination with the previous formula. We also get the remaining estimates of (3).

The proof of (5) is analogous and uses Prop. \ref{kerhs} (2) and, for the exponential decay, the fact that $e^{-t(\D(r,\rho)+\Pj)^2}=e^{-t\D(r,\rho)^2}(1-\Pj)$.
\end{proof}

\subsection{The heat kernel}
\label{heatkernel}
  
In the following we construct and study the heat kernel of $e^{-t\D(r,\rho)^2}$ by  comparison with $E(r)_t$. 

By Duhamel's principle
\begin{align}
\nonumber \lefteqn{E(r)_t-e^{-t\D(r,\rho)^2}}\\
\label{duhprin} &= \sum_{j=0}^2 \int_0^t e^{-s\D(r,\rho)^2}[\D_{\E},c(d\zeta_j)]e^{-(t-s)\D_j(r)^2}\phi_j ~ds \\
\nonumber & \quad +\int_0^t e^{-s\D(r,\rho)^2}\bigl(\rho [\Ki, \D(r,0)] + \rho^2\Ki^2 \bigr) E(r)_{t-s} ~ds \ .
\end{align}

Using  the heat kernel estimates of \S \ref{heatm} and Prop. \ref{kersupconcyl} we infer that the operator $\sum_{j=0}^2[\D_{\E},c(d\zeta_j)]e^{-t\D_j(r)^2}\phi_j$ has a smooth integral kernel $g_t(x,y)$, which is supported in $U_1\times M$. For any $\ve>0,~c>4$ there is $C>0$ such that for $t > 0, r\in \bbbr$ 
$$| g_t(x,y)| \le Ce^{\ve t}1_{U_1}(x)te^{-\frac{d(U_0 \cup U_1,y)^2}{ct}} \ .$$
Analogous estimates hold for the partial derivatives and the first derivatives with respect to $r$ and $t$.  

The term $\bigl(\rho [\Ki, \D(r,0)] + \rho^2 \Ki^2\bigr)E(r)_t$ is an integral operator with smooth integral kernel supported on $(U_0 \cup U_1) \times (U_0 \cup U_1)$. The integral kernel converges in $\C_c(M \times M,\E_i \boxtimes_{\A_i} \E_i^*)$ for $t \to 0$ uniformly in $r \in \bbbr,~ \rho \in [-1,1]$. Furthermore for any $\ve>0$ there is $C>0$ such that the integral kernel is bounded by $Ce^{\ve t}1_{U_0 \cup U_1}(x)1_{U_0 \cup U_1}(y)$ for $t\ge 0,r \in \bbbr,\rho \in [-1,1]$. Analogous estimates hold for the partial derivatives and the first derivatives with respect to $r,\rho$ and $t$.  

Let $m,k \in \bbbn_0$ with $m >\frac{\dim M + k+1}{2}$. 
Since $e^{-s\D(r,\rho)^2}:C^{2m}_c(M,\Ol{\mu}\E_i) \to C^k(M,\Ol{\mu}\E_i)$ is bounded by Prop. \ref{heatsemest}, these arguments imply that  $E(r)_t-e^{-t\D(r,\rho)^2}$ is an integral operator with smooth integral kernel, which furthermore is of class $C^1$ in $(r,\rho,t)$. Hence $e^{-t\D(r,\rho)^2}$ is an integral operator with integral kernel $k(r,\rho)_t \in \C(M \times M,\Ol{\mu}\E_i \boxtimes_{\Ol{\mu}\A_i} \Ol{\mu}\E_i^*)$, which is of class $C^1$ in $(r,\rho,t)$. 

We denote by $e(r)_t$ the integral kernel of $E(r)_t$.

From the previous considerations and Prop. \ref{heatsemest} we obtain the following estimate on the integral kernel:

For $\ve>0,~c>4$ there is $C>0$ such that for all $t>0,~\rho \in [-1,1],~r \in \bbbr$
\begin{align}
\label{compheatker}
|k(r,\rho)_t(x,y)- e(r)_t(x,y)| &\le Ce^{\ve t}te^{-\frac{d(U_0 \cup U_1,y)^2}{ct}} \ .
\end{align}
Analogous estimates hold for the first derivatives with respect to $r,\rho,t$ and the partial derivatives with respect to $x,y$.

\begin{prop} 
\label{heatsemC}
The operator $-\D(r,\rho)^2$ generates a holomorphic semigroup on $C_0(\bbbr \times [-1,1],C_0^k(M,\Ol{\mu}\E_i))$. For $t>0$ the operator $e^{-t\D(r,\rho)^2}$  is of class $C^1$ in $(r,\rho,t)$. There is $C>0$ such that on $C_0(\bbbr \times [-1,1],C_0^k(M,\Ol{\mu}\E_i))$ for $t$ small
$$\|\frac{d}{dr}e^{-t\D(r,\rho)^2}\| \le Ct^{1/2}$$

For any $\ve>0$ there is $C>0$ such that for all $t>0$  on $C_0(\bbbr\times [-1,1],C_0^k(M,\Ol{\mu}\E_i))$ 
$$\| e^{-t\D(r,\rho)^2} \| \le Ce^{\ve t} \ .$$
\end{prop}

\begin{proof}
The operator $E(r)_t$ is strongly continuous in $t$. It fulfills the bounds in the assertion by the estimates in \S \ref{semigroupm} and by estimate \ref{estimcyl}. Furthermore, since $E(r)_t$ is defined from the holomorphic semigroups on a closed manifold and the cylinder, on $C_0(\bbbr \times [-1,1],C_0^k(M,\Ol{\mu}\E_i))$ for $t$ small
$$\|\frac{d}{dt} E(r)_t\| \le Ct^{-1} \ .$$
The estimate \ref{compheatker} shows that $E(r)_t-e^{-t\D(r,\rho)^2}$ is well-defined as a bounded operator on $C_0(\bbbr \times [-1,1],C_0^k(M,\Ol{\mu}\E_i))$. Furthermore it is of class $C^1$ with respect $(r,\rho,t)$ and it and its first derivatives are bounded by $C te^{\ve t}$ for $t>0$. It follows that $e^{-t\D(r,\rho)^2}$ is a strongly continuous semigroup fulfilling the estimates. This and the previous estimate imply that for $t$ small 
$$\|\frac{d}{dt}e^{-t\D(r,\rho)^2} \| \le Ct^{-1} \ .$$
Thus, by Prop. \ref{exthol}, the semigroup is holomorphic.
\end{proof}

\section{The semigroup associated to the superconnection}
\label{Msupcon}

We define the superconnection associated to $\D(r,\rho)$ by $$\spc(r,\rho):=P\di P+ \gamma + \D(r,\rho)$$ and the rescaled superconnection by $$\spc(r,\rho)_t:=P\di P + \gamma  + \sqrt t \D(r,\rho) \ .$$ Later we will set $r=t$, however the following calculations are more transparent with both variables kept separate.

The curvature $$\spc(r,\rho)^2=\D(r,\rho)^2+ [P\di P + \gamma,\D(r,\rho)] + (P\di P + \gamma)^2$$  is a nilpotent perturbation of $\D(r,\rho)^2$, and the difference $\spc(r,\rho)^2-\D(r,\rho)^2$ is bounded on $C_0(\bbbr \times [-1,1],L^2(M,\Ol{\mu}\E_i))$ resp. on $C_0(\bbbr \times [-1,1],C^k_0(M,\Ol{\mu}\E_i))$. It follows that $e^{-t\spc(r,\rho)^2}$ is a holomorphic semigroup on $C_0(\bbbr \times [-1,1],L^2(M,\Ol{\mu}\E_i))$ and on $C_0(\bbbr \times [-1,1],C^k_0(M,\Ol{\mu}\E_i))$. If $N$ is the operator that multiplies a homogenous element (with respect to the $\bbbz$-grading on $\Ol{\mu}\A_i$) of order $m$ by $m$, then $$e^{-\spc(r,\rho)_t^2}= t^{-N/2} e^{-t\spc(r,\rho)^2}t^{N/2} \ .$$

The strategy for the study of $e^{-\spc(r,\rho)_t^2}$ is as in the previous section the comparison with an appropriate approximation. 

For $j=0,1,2$ set $\spc_j(r)=P_j \di P_j + \gamma_j + \D_j(r)$ and $\spc_j(r)_t:=P_j \di P_j +  \gamma_j + \sqrt t \D_j(r)$. 
We define $$H(r)_t:=\sum_{j=0}^2 \zeta_j e^{-\spc_j(r)_t^2} \phi_j$$
and write $h(r)_t$ for its integral kernel.

Since $e^{-\spc(r,\rho)_t^2}$ is not a semigroup, we cannot directly use Duhamel's principle as in the previous section. This makes the proof of the following proposition more technical.

\begin{prop}
\label{compare}
The operator $e^{-\spc(r,\rho)_t^2}$ is an integral operator with integral kernel $p(r,\rho)_t \in \C(M \times M,\Oei \boxtimes_{\Ol{\mu}\A_i} (\Oei)^*)$, which is of class $C^1$ in $(r,\rho,t)$. The following estimates hold:
\begin{enumerate}
\item For any $\ve>0,~c>4$ there is $C>0$ such that for $t>\ve,~\rho \in [-1,1],~r\in \bbbr$ 
$$|p(r,\rho)_t(x,y)-h(r)_t(x,y)|  \le C e^{\ve t}e^{ -\frac{d(y,U_0 \cup U_1)^2}{ct}} \ .$$
\item For $T>0$ and $c>4$ there is $C>0$ such that for $0<t<T$ 
$$|p(0,0)_t(x,y)-h(0)_t(x,y)| \le C\sum_{j=0}^2 e^{-\frac{d(y,\supp d\zeta_j)^2}{ct}}1_{\supp \phi_j}(y) \ .$$
\end{enumerate}
Analogous estimates hold for the partial derivatives with respect to $x,y$. The first estimate also holds for the first derivatives with respect to $r, \rho$.
\end{prop}

\begin{proof}
In order to keep track of the rescaling process we use Volterra development. 

With $B=[P\di P+ \gamma,\D(r,\rho)] + (P\di P + \gamma)^2$ and 
$$I_n(r,\rho,t)=\int_{\Delta^n} e^{-u_0t\D(r,\rho)^2}Be^{-u_1t\D(r,\rho)^2} \dots Be^{-u_nt\D(r,\rho)^2}~du_0du_1 \dots du_n$$
we have that
$$e^{-t\spc(r,\rho)^2}= \sum\limits_{n=0}^{\infty} (-1)^n t^n I_n(r,\rho,t) \ .$$
Only the first $\mu +1$ terms of the series are non-trivial.

The operators $I_n(r,\rho,t)$ are bounded operators on $L^2(M,\Ol{\mu}\E_i)$ resp. on $C^k_0(M,\Ol{\mu}\E_i)$, which are of class $C^1$ in $(r,\rho,t)$ for $t>0$. The derivatives with respect to $r$ and $\rho$ converge for $t \to 0$. 
 
Prop. \ref{heatsemest} and \ref{heatsemC} imply that for any $\ve>0$ there is $C>0$ such that for all $t\ge 0$ on $C_0([-1,1]\times \bbbr,L^2(M,\Ol{\mu}\E_i))$ as well as on $C_0([-1,1]\times \bbbr,C_0^k(M,\Ol{\mu}\E_i))$
$$\| I_n(r,\rho,t) \| \le Ce^{\ve t} \ .$$  
In the following homogenous components are taken with respect to the grading on $\Ol{\mu}\A_i$.
We define $I_n(r,\rho,t)_m$ to be the homogenous component of degree $m$. It holds that $I_n(r,\rho,t)_m=0$ for $m<n$ or $m>2n$. 

Furthermore we implicitely define operators $J_m(r,\rho,t)$, homogenous of degree $m \in \bbbn_0$, by
$$e^{-\spc(r,\rho)_t^2}= \sum_{m=0}^{\infty} J_m(r,\rho,t) \ .$$ 
From $$J_m(r,\rho,t)=\sum_{n=[(m+1)/2]}^m (-1)^nt^{n-m/2} I_n(r,\rho,t)_m$$ it follows
that for any $\ve>0$ there is $C>0$ such that $\|J_m(r,\rho,t)\| < Ce^{\ve t}$ for $t\ge 0$ on $C_0([-1,1] \times \bbbr, L^2(M,\Ol{\mu}\E_i))$ resp. on $C_0([-1,1]\times \bbbr,C^k_0(M,\Ol{\mu}\E_i))$. 

Let $W(r,t):=\sum_{j=0}^2 \zeta_j e^{-t\spc_j(r)^2} \phi_j$
and for $m \in \bbbn_0$ let $V_m^j(r,t)$ be the unique homogenous operator of degree $m$ such that
$$e^{-t\spc_j(r)^2}=\sum_{m=0}^{\infty} V^j_m(r,t) \ .$$

For $t>0$ the operators $W(r,t)$, $V^j_m(r,t)$ are integral operators with smooth integral kernels, which are of class $C^1$ in $(r,t)$.

Duhamel's principle yields that
\begin{align*}
\lefteqn{W(r,t)- e^{-t \spc(r,\rho)^2}}\\
&= \int_0^t e^{-s \spc(r,\rho)^2}(\frac{d}{dt}+ \spc(r,\rho)^2)W(r,t-s)~ds\\
&=\sum_{j=0}^2 \int_0^t e^{-s \spc(r,\rho)^2}\bigl(-\zeta_j \spc_j(r)^2 + \\
&\qquad (\spc(r,0)^2 + \rho [\D(r,0)+P\di P+ \gamma, \Ki] + \rho^2 \Ki^2) \zeta_j\bigr)e^{-(t-s)\spc_j(r)^2}\phi_j~ds\\
&=\sum_{j=0}^2 \int_0^t e^{-s\spc(r,\rho)^2}\bigl(\D_{\E}c(d \zeta_j)+ c(d\zeta_j) \D_j- \\
& \qquad  \zeta_j[P_j \di P_j + \gamma_j,\D_j(r)]+
 [P\di P + \gamma,\D(r,0)]\zeta_j \bigr)e^{-(t-s)\spc_j(r)^2}\phi_j~ds \\
& \quad + \rho \sum_{j=0}^2 \int_0^t e^{-s\spc(r,\rho)^2}([\D(r,0)+P\di P+ \gamma, \Ki]  +\rho \Ki^2) \zeta_je^{-(t-s)\spc_j(r)^2}\phi_j~ds \\
&=\sum_{j=0}^2 \int_0^t e^{-s\spc(r,\rho)^2}(\D_{\E} c(d \zeta_j)+ c(d\zeta_j)\D_j)e^{-(t-s)\spc_j(r)^2}\phi_j ~ds \\
& \quad + \rho \sum_{j=0}^2 \int_0^t e^{-s\spc(r,\rho)^2}([\D(r,0)+P\di P+ \gamma, \Ki] + \rho \Ki^2) \zeta_je^{-(t-s)\spc_j(r)^2}\phi_j~ds \ . 
\end{align*}

The rescaled operator 
$$t^{-N/2}\sum_{j=0}^2 \int_0^t e^{-s\spc(r,\rho)^2}([\D(r,0)+P\di P+ \gamma, \Ki] + \rho \Ki^2) \zeta_je^{-(t-s)\spc_j(r)^2}\phi_j ~ds ~t^{N/2}$$ 
is an integral operator whose integral kernel is supported in $M\times (U_0\cup U_1)$ and bounded in by $Ce^{\ve t}$ uniformly in $(\rho,r) \in [-1,1] \times \bbbr$ for $t> \ve$.

By using that $$e^{-t \spc(r,\rho)^2}=\sum_{n=0}^{\infty} t^{n/2}J_n(r,\rho,t)$$
one gets that the degree $n$ component of the second before last line equals, after rescaling,
$$\sum_{m=0}^n\sum_{j=0}^2 t^{-n/2}\int_0^t s^{(n-m)/2}J_{n-m}(r,\rho,s)(\D_{\E} c(d \zeta_j)+ c(d\zeta_j)\D_j)V_m^j(r,t-s)\phi_j~ds \ .$$

Now the assertions follow from off-diagonal estimates applied to $(\D_{\E} c(d \zeta_j)+ c(d\zeta_j)\D_j)V_m^j(r,t-s)\phi_j$:
The first assumption follows from the off-diagonal estimates on the cylinder, see Prop. \ref{kersupconcyl}.  For the second assumption we use in addition the off-diagonal estimates for the heat kernel on a closed manifold, which follow from the asymptotic development \cite[Ch. 2]{bgv}. 
\end{proof}

\section{The index theorem}
\label{secindtheor}

We prove the Atiyah--Patodi--Singer index theorem using the approach of \cite{bk}. The notions of Hilbert--Schmidt operators and trace class operators used in the following have been defined in \cite{wa}. 

At this point we set $r=t$.

Let $\alpha:\bbbr \to [0,1]$ be a smooth function with $\alpha(x)=1$ for $x \le 0$ and $\alpha(x)=0$ for $x \ge 1$. For $R\ge 0$ set $\alpha_R(x)=\alpha(x-R)$. Then $\alpha_R$ defines a function on $M$, again denoted by $\alpha_R$, that equals $1$ on $M_c$ and $\alpha_R(x_1)$ on the cylindric end. 

Assume that $K$ is a bounded operator on $L^2(M,\Ol{\mu}\E_i)$ such that $\alpha_R K \alpha_R$ is trace class for any $R\ge 0$. Define 
$$\Tr_s' K:= \lim\limits_{R \to \infty} \Tr_s(\alpha_R K \alpha_R)$$ provided the limit exists. The definition does not depend of the choice of $\alpha$.

The following properties are easy to verify.
\begin{enumerate}
\item If $K$ is a trace class operator, then $\Tr_s'K=\Tr_s K$.
\item Let $K$ be an integral operator. If $\Tr_s' K$ exists, then
$\Tr_s' [P\di P+ \gamma,K]$ exists as well and equals $\di \Tr_s' K$.
\item Let $K_1,K_2$ be integral operators with continuous integral kernels $k_1,k_2$. Assume that one of the following conditions holds:
\begin{enumerate}
\item  The function $(x,y)\mapsto \tr_s (k_1(x,y)k_2(y,x))$ is compactly supported on $M \times M$.
\item  There are $f,g \in L^1(M), h \in L^1(\bbbr)$ such that $|k_1(x,y)| \le f(y)+ h(d(x,y))$ and $|k_2(x,y)| \le g(y)$ for all $x,y \in M$.
\end{enumerate} 
Then $\Tr_s'[K_1,K_2]=0$. 
\end{enumerate}

\begin{prop}
\label{uniflim}
The expressions
$$\Tr_s' e^{-\spc(t,\rho)_t^2}$$ and $$\Tr_s' \frac{d\spc(t,\rho)_t}{dt}e^{-\spc(t,\rho)_t^2}$$ are well-defined for $t>0$.

Furthermore there are $C,\delta>0$ such that for all $t>1$
$$| \Tr_s' e^{-\spc(t,\rho)_t^2}- \Tr_s \alpha_t e^{-\spc(t,\rho)_t^2}\alpha_t | \le Ce^{-\delta t} \ ,$$
$$| \Tr_s' \frac{d\spc(t,\rho)_t}{dt} e^{-\spc(t,\rho)_t^2}- \Tr_s \alpha_t    \frac{d\spc(t,\rho)_t}{dt} e^{-\spc(t,\rho)_t^2}\alpha_t | \le Ce^{-\delta t} \ .$$
\end{prop}

\begin{proof}
By Prop. \ref{compare} and since $\tr_s h(t)_t(x,x)=0$ for $x \in Z^+$, for any $\ve>0, ~c>4$ there is $C>0$ such that  
$$|\tr_s p(t,\rho)_t(x,x)| \le Ce^{\ve t}e^{-\frac{d(x,U_0\cup U_1)^2}{ct}}$$
for all $x \in M,~t>\ve$.

In particular $\Tr_s' e^{-\spc(t,\rho)_t^2}$ is well defined.

Furthermore there are $C,\delta>0$ such that for $t>1$ 
$$\int_t^{\infty} e^{-\frac{s^2}{ct}}~ds < Ce^{-\delta t} \ .$$
Choose $\ve<\delta$.
Then the previous estimates imply that there is $C>0$ such that for all $t>1$
$$|\Tr_s' (1-\alpha_t^2)e^{-\spc(t,\rho)_t^2}| < Ce^{(\ve-\delta)t} \ .$$

The proof for $\Tr_s' \frac{d\spc(t,\rho)_t}{dt}e^{-\spc(t,\rho)_t^2}$ is analogous.
\end{proof} 

\subsection{The limit $t \to \infty$}
\label{liminf}

In the following we fix $\rho \neq 0$. Recall that $\Pj$ is the projection onto the kernel of $\D(r,\rho)$ for $r \ge 2$. (Clearly, $\Pj$ does not depend on $r \ge 2$). 
The second estimate of the following theorem is not used in the proof of the index theorem, however it is useful, for example, for the definition of $\eta$-forms for Dirac operators on manifold with cylindric ends and therefore is included here.

\begin{theorem}
\label{liminftheo}
There is $C>0$ such that for all $t>2$
$$|\Tr_s' e^{-\spc(t,\rho)^2_t} - \Tr_s e^{-(\Pj_0 \di\Pj_0 +\Pj_0\gamma \Pj_0)^2}| \le C  t^{-1/2} \ $$
and 
$$|\Tr_s' \frac{d\spc(t,\rho)_t}{dt}e^{-\spc(t,\rho)^2_t}| \le C t^{-\frac 32} \ .$$
\end{theorem}

\begin{proof}
We set $\Pj_0=\Pj$ and $\Pj_1=1-\Pj_0$.

For the first estimate the proof of \cite[Theorem 9.19]{bgv} works in slightly modified form. We sketch the argument. Let ${\mathcal N}_m,~ m\in \bbbn$ denote the nilpotent algebra of trace class operators on $L^2(M,\Ol{\mu}\E_i)$ that increase the degree with respect to the $\bbbz$-grading on $\Ol{\mu}\A_i$ by at least $m$ and that are of the form $\sum_j f_j g_j^*$ with $f_j,g_j \in {\mathcal S}(M,\Ol{\mu}\E_i)$. Then, by an adaption of \cite[Lemma 9.21]{bgv}, there is $g \in 1+{\mathcal N}_1$, such that 
$$g \spc(t,\rho)^2 g^{-1}  =(\Pj_0 \di\Pj_0 +\Pj_0\gamma \Pj_0)^2+\Pj_1\D(t,\rho)^2\Pj_1+B_0+B_1$$
with $B_0 \in \Pj_0 {\mathcal N}_3\Pj_0$ and $B_1 \in \Pj_1[P\di P + \gamma, \D(t,\rho)]\Pj_1 + \Pj_1{\mathcal N_1} \Pj_1$.

Recall that $\spc(t,\rho)_t^2=t^{-N/2+1}\spc(t,\rho)^2t^{N/2}$. Define $g_t=t^{-N/2}g t^{N/2}$. Then $$g_t\spc(t,\rho)_t^2 g_t^{-1}=t^{-N/2+1}g\spc(t,\rho)^2g^{-1} t^{N/2} \ .$$

This implies that 
\begin{align*}
\Tr_s' e^{-\spc(t,\rho)_t^2}&=\Tr_s'\exp(t^{-N/2+1}g\spc(t,\rho)^2g^{-1} t^{N/2})\\
&=\Tr_s \exp(- (\Pj_0 \di \Pj_0+\Pj_0\gamma\Pj_0)^2- t^{-N/2+1}B_0t^{N/2}) \\
&\quad + \Tr_s'\exp(- t \Pj_1\D(t,\rho)^2- t^{-N/2+1}B_1t^{N/2}) \ .
\end{align*}

By Volterra development there is $C>0$ such that for $t>2$
$$|\Tr_s \exp(- (\Pj_0 \di \Pj_0+\Pj_0\gamma\Pj_0)^2- t^{-N/2+1}B_0t^{N/2})-\Tr_s e^{- (\Pj_0 \di \Pj_0+\Pj_0\gamma\Pj_0)^2}|\le Ct^{-1/2} \ .$$ 
Furthermore
\begin{align*}
\lefteqn{\exp(- t \Pj_1\D(t,\rho)^2-t^{-N/2+1}B_1t^{N/2})}\\
&=t^{-N/2}\sum_{n=0}^{\infty}(-1)^nt^n \int_{\Delta^n} \Pj_1e^{-u_0t\D(t,\rho)^2}B_1\Pj_1e^{-u_1t\D(t,\rho)^2} \\
&\qquad \dots B_1\Pj_1e^{-u_nt\D(t,\rho)^2}~du_0du_1 \dots du_n~ t^{N/2}\ .
\end{align*}
We can decompose the right hand side further by decomposing $B_1=B_{11}+B_{12}$ with $B_{11}=\Pj_1[P\di P+ \gamma, \D(t,\rho)]\Pj_1$ and $B_{12} \in \Pj_1{\mathcal \N_1} \Pj_1$. Any term in the development having a factor $B_{12}$ is a trace class operator whose supertrace vanishes exponentially for $t \to \infty$ by Prop. \ref{heatsemest} (5).

As in the proof of \cite[Lemma 4.4.5]{wa} one shows that there are $C,\delta>0$ such that for $t>2$
$$| \Tr_s \alpha_t^2 \int_{\Delta^n} \Pj_1e^{-u_0t\D(t,\rho)^2} B_{11} \Pj_1e^{-u_1t\D(t,\rho)^2} \dots $$
$$ \dots B_{11} \Pj_1e^{-u_nt\D(t,\rho)^2}~du_0du_1 \dots du_n | \le C e^{-\delta t} \ .$$

By Prop. \ref{uniflim} this implies the assertion.

The second estimate is the analogue of \cite[Theorem 9.23]{bgv}. For $t>2$ we have that
$$\frac{d\spc(t,\rho)_t}{dt}=\frac 12 t^{-1/2} \D(t,\rho) \ .$$ The assertion follows as above by using that 
\begin{align*}
g_t\D(t,\rho)g_t^{-1} &= t^{-N/2}g\D(t,\rho)g^{-1}t^{N/2}\\
&\in \D(t,\rho)+t^{-N/2}(\N_1\D(t,\rho)+ \D(t,\rho)\N_1)t^{N/2}+ t^{-N/2}{\mathcal N}_2t^{N/2}
\end{align*} 
and that any operator in $t^{-N/2}\N_m t^{N/2}$ is bounded by $Ct^{-m/2}$ in the trace class norm for $t>2$. We also use that $\Pj_0\D(t,\rho)=0$.
\end{proof}

\subsection{Variation formulas}

Recall the definition of $\spc^N(r)_t$ from the end of \S \ref{eta}. 

\begin{prop}
\label{varfor}
We have that
$$\frac{d}{dt} \Tr_s' e^{-\spc(t,\rho)_t^2}= -\frac{1}{\sqrt{\pi }} \Tr_{\sigma}\sigma\frac{d\spc^N(t)_t}{dt} e^{-\spc^N(t)_t^2}- \di \Tr_s' \frac{d\spc(t,\rho)_t}{dt}e^{-\spc(t,\rho)_t^2}$$
and
$$\frac{d}{d \rho} \Tr_s' e^{-\spc(t,\rho)_t^2} = - \di \Tr_s \frac{d\spc(t,\rho)_t}{d\rho} e^{-\spc(t,\rho)^2_t} \ .$$
\end{prop}

\begin{proof}
Since by \S \ref{cyl} the supertrace of $\frac{d}{dt} h(t)_t(x,x)$ vanishes for $x \in M\setminus (U_0 \cup U_1)$, it follows from Prop. \ref{compare} that there are $c,C>0$ such that
$$|\frac{d}{dt}\tr_s p(t,\rho)_t(x,x)| \le Ce^{-cd(x,U_0 \cup U_1)^2} $$
uniformly for $t$ in a compact subset of $(0,\infty)$. Hence 
$$\frac{d}{dt} \Tr'_s e^{-\spc(t,\rho)_t^2}= \Tr'_s \frac{d}{dt} e^{-\spc(t,\rho)_t^2} \ .$$

As in the proof of \cite[Lemma 4.4.9]{wa}, chain rule and Duhamel's formula (see Prop. \ref{duhform}) imply that
$$\frac{d}{dt} e^{-\spc(t,\rho)_t^2} = - \int_0^1 e^{- (1-s)\spc(t,\rho)_t^2}\frac{d \spc(t,\rho)^2_t}{dt} e^{- s\spc(t,\rho)_t^2}~ds$$
and
$$\frac{d}{d\rho} e^{-\spc(t,\rho)_t^2}
= - \int_0^1 e^{- (1-s)\spc(t,\rho)_t^2}\frac{d \spc(t,\rho)^2_t}{d\rho} e^{- s\spc(t,\rho)_t^2}~ds \ .$$

In the following $t>0$ is fixed.

We set $$A(s)=e^{-s\spc(t,\rho)_t^2}- \sum_{j=0}^2 \zeta_j e^{- s \spc_j(t)_t^2}\phi_j$$
and $$B(s)=\sum_{j=0}^2 \zeta_j e^{- s\spc_j(t)_t^2}\phi_j \ .$$
Then 
\begin{align*}
\lefteqn{ \int_0^1 e^{- (1-s)\spc(t,\rho)_t^2}\frac{d \spc(t,\rho)^2_t}{dt} e^{- s\spc(t,\rho)_t^2}~ds}\\
  &=  \int_0^1 e^{- (1-s)\spc(t,\rho)_t^2}\frac{d \spc(t,\rho)^2_t}{dt}A(s)~ds \\
& \quad +\int_0^1 A(1-s)\frac{d \spc(t,\rho)^2_t}{dt}B(s)~ds \\
&\quad + \int_0^1 B(1-s)\frac{d \spc(t,\rho)^2_t}{dt}B(s)~ds \ .
\end{align*}
We treat the first term on the right hand side of the equation. We denote the integral kernel of $A(s)$ by $A(s)(x,y)$ and similarly for other integral operators.

The proof of Prop. \ref{compare} can be modified easily to show that for $s\in [0,1]$ $$\|A(s)(\cdot,y)\|_{C_0^k}\le Ce^{-cd(y,U_0 \cup U_1)^2} \ .$$
This and the fact that $e^{-(1-s)\spc(t,\rho)_t^2}$ is uniformly bounded on $C_0^k(M,\Ol{\mu}\E_i)$ imply that there are $C,c>0$ such that for all $s \in [0,1]$ and $x \in M$
$$|\bigl( e^{- (1-s)\spc(t,\rho)_t^2}\frac{d \spc(t,\rho)^2_t}{dt}A(s)\bigr)(x,x) | \le Ce^{-cd(x,U_0 \cup U_1)^2} \ .$$
Thus we can interchange $\Tr_s'$ and integration over $s$.

Furthermore for fixed $s,t$ there are $C,c>0$ such that
$$|e^{- (1-s)\spc(t,\rho)_t^2}(x,y)| \le C(e^{-cd(x,y)^2}+e^{-cd(y,U_0 \cup U_1)^2}) \ .$$
Thus property (3) of $\Tr'_s$ (see the beginning of \S \ref{secindtheor}) implies that
$$\Tr_s'[e^{- (1-s)\spc(t,\rho)_t^2},\frac{d \spc(t,\rho)^2_t}{dt}A(s)]=0 \ .$$
Simllar arguments hold for the second term.

For the third term one checks that for all $x,y \in M\setminus (U_0 \cup U_1)$
$$\tr_s B(1-s)(x,y)(\frac{d \spc(t,\rho)^2_t}{dt}B(s))(y,x)=0 \ .$$
It follows that $\Tr_s'$ and integration on $s$ can be interchanged.
Again, property (3) of $\Tr_s'$ implies that 
$$\Tr_s'[ B(1-s),\frac{d \spc(t,\rho)^2_t}{dt}B(s)]=0 \ .$$
 
Summarizing and evaluating, we get that
\begin{align*}
\frac{d}{dt} \Tr_s' e^{-\spc(t,\rho)_t^2}&=-\int_0^1 \Tr_s' e^{- (1-s)\spc(t,\rho)_t^2}\frac{d \spc(t,\rho)^2_t}{dt} e^{- s\spc(t,\rho)_t^2}~ds \\
&=- \Tr_s' \frac{d \spc(t,\rho)^2_t}{dt} e^{- \spc(t,\rho)_t^2} \\
&= -\Tr'_s[\spc(t,\rho)_t, \frac{d\spc(t,\rho)_t}{dt}e^{-\spc(t,\rho)_t^2}]  \\
&= - \Tr'_s [P \di P+ \gamma,\frac{d\spc(t,\rho)_t}{dt}e^{-\spc(t,\rho)_t^2}] - \sqrt t \Tr'_s [\D(t,\rho),\frac{d\spc(t,\rho)_t}{dt}e^{-\spc(t,\rho)_t^2}] \\
&=- \di \Tr_s' \frac{d\spc(t,\rho)_t}{dt}e^{-\spc(t,\rho)_t^2}+\sqrt t \lim\limits_{R \to \infty} \Tr_s c(d\alpha_R)\frac{d\spc_2(t)_t}{dt}e^{-\spc_2(t)_t^2} \ .
\end{align*}
See the proof of \cite[Lemma 4.4.10]{wa} for omitted details.

Now
\begin{align*}
\Tr_s c(d\alpha_R)\frac{d\spc_2(t)_t}{dt}e^{-\spc_2(t)_t^2}&=\Tr_s\alpha'_R\frac{d\bigl(\sqrt t (\tilde \D_N + \psi(t) \tilde A)\bigr)}{dt}e^{-\spc_2(t)_t^2} \\
& \qquad - \frac 12 t^{-1/2} \Tr_s (\alpha'_R \ra_1 e^{-\spc_2(t)_t^2}) \ .
\end{align*}
The second term on the right hand side vanishes since $\ra_1 e^{-(x_1-y_1)^2/4t}$ vanishes on the diagonal.
By using eq. \ref{cyleq1} and \ref{cyleq2} and  
$$\int_0^{\infty} \alpha_R'(x_1)~dx_1 =-1$$ 
one sees that the first term on the right hand side equals
$$-\frac{2}{\sqrt{4\pi t}}\Tr_{\sigma}\frac{d\bigl(\sqrt t (\D_N + \psi(t) A)\bigr)}{dt}e^{-\spc^N(t)_t^2} \ .$$

This implies the first assertion of the proposition.

The equation 
$$\frac{d}{d \rho} \Tr_s' e^{-\spc(t,\rho)_t^2} = - \Tr_s' \frac{d\spc(t,\rho)^2_t}{d \rho} e^{-\spc(t,\rho)_t^2} $$
is proven by similar but simpler arguments as above. 

Hence
\begin{align*}
\frac{d}{d \rho} \Tr'_s e^{-\spc(t,\rho)_t^2}&= -\Tr'_s [\spc(t,\rho)_t, \frac{d\spc(t,\rho)_t}{d\rho}e^{-\spc(t,\rho)_t^2}] \\
&=-\sqrt t \Tr_s [P \di P+ \gamma, \Ki e^{-\spc(t,\rho)_t^2}] - t \Tr_s [\D(t,\rho),  \Ki e^{-\spc(t,\rho)_t^2}] \\
&=-\sqrt t \di \Tr_s  \Ki e^{-\spc(t,\rho)_t^2} \ .
\end{align*}
We used that
$$\Tr_s[\D(t,\rho),  \Ki e^{-\spc(t,\rho)_t^2}]= \Tr_s[\D(t,\rho)\Ki,  e^{-\spc(t,\rho)_t^2}] = 0 \ .$$
\end{proof}

\subsection{The theorem}

The proof of the index theorem is now as usual. Recall the definition of $\D_{\E}(A)$ from \S \ref{DirHilb}. In the following $M$ denotes the manifold without the isolated point, as in the beginning of \S \ref{DirHilb}.

\begin{theorem}
\label{indtheor}
In $H^{\ideal_{\infty}}_*(\Ai)$ it holds that
$$\ch(\ind \D_{\E}(A)^+)=(2 \pi i)^{-n/2} \int_M \hat A(M)\ch(\E/S) - \eta(\dira_N,A) \ .$$
\end{theorem}

\begin{proof}
By the remarks at the end of \S \ref{DirHilb} for $r\ge 2,~\rho=1$
$$\ch (\ind \D_{\E}(A)^+)=\ch(\Ker \D(r,\rho))-\N$$ 
By Theorem \ref{liminftheo} and Prop. \ref{projker}
$$\ch (\Ker\D(3,1))=\lim_{t \to \infty} \Tr_s' e^{-\spc(t,1)_t^2} \ .$$
For $0 < s<1<t$ it holds that
\begin{align*}
\Tr_s' e^{-\spc(t,1)_t^2}-\Tr_s' e^{-\spc(s,0)_s^2}&=\int_s^1\frac{d}{du} \Tr_s' e^{-\spc(u,0)_u^2}~du \\
&\quad + \int_0^1 \frac{d}{d\rho}\Tr_s' e^{-\spc(1,\rho)_1^2}~d\rho \\
&\quad + \int_{u=1}^t\frac{d}{du} \Tr_s' e^{-\spc(u,1)_u^2}~du \ .
\end{align*}
Now it follows from Prop. \ref{varfor} that modulo exact forms
$$\Tr_s' e^{-\spc(t,1)_t^2}-\Tr_s' e^{-\spc(s,0)_s^2}=-\frac{1}{\sqrt{\pi }} \int_s^t \Tr_{\sigma}\sigma\frac{d\spc^N(t)_t}{dt} e^{-\spc^N(t)_t^2} \ .$$

Furthermore, by Prop. \ref{compare} and by the small time limit from \S \ref{supcon}
\begin{align*}
\lim_{s\to 0}\Tr_s' e^{-\spc(s,0)_s^2}&=\lim_{s \to 0}\Tr_s'H(s)_s \\
&=(2 \pi i)^{-n/2} \int_M \hat A(M)\ch(\E/S) + \N \ .
\end{align*}
Here we also used that the supertrace of the integral kernel of $H(s)$ vanishes on the cylindric end. The summand $\N$ comes from the isolated point.
\end{proof}

We can now define the $\eta$-form for trivializing operators not necessarily adapted to $\D_N$. It is well-defined up to (the closure of the space of) exact forms.

\begin{ddd}
\label{apprtriv}
Let $A$ be a trivializing operator of $\D_N$. A symmetric operator $\ad_A$ adapted to $\D_N$ is called an {\rm adapted approximation} of $A$ if on $L^2(N,\E^N)$
$$\|(A-\ad_A)(\D_N+A)^{-1}\| < \tfrac 12 \ .$$
\end{ddd}

There is always an adapted approximation $\ad_A$. We can even demand that $\ad_A$ is an integral operator with integral kernel in $\C(N \times N,\E_{\infty} \boxtimes_{\Ai} \E_{\infty}^*)$. Note that by the Neumann series $\D_N+s\ad_A+(1-s)A$ is invertible for $s \in [0,1]$. 

If $A$ is a trivializing operator of $\D_N$ which is not adapted, we may set
$$\eta(\D_N,A):=\eta(\D_N,\ad_A) \ .$$
The Atiyah--Patodi--Singer index theorem on the cylinder implies that modulo the closure of the space of exact forms the definition does not depend on the choice of $\ad_A$, and also not on the choice of $\psi$ (see the end of \S \ref{eta} for that choice).

With this definition the index theorem holds also if $A$ is not adapted since 
$\ind \D_{\E}(A)^+=\ind \D_{\E}(\ad_A)^+$.

\section{Product formula for $\eta$-forms}
\label{prodeta}

In the following we prove a product formula for $\eta$-forms. Such a formula was derived in \cite[\S 2]{lpetpos} under the assumption that the Dirac operator is invertible. Also the equality established in \cite[Lemma 6]{mp2} (when translated from family index theory to the present setting) concerning suspended $\eta$-forms and the product formulas proven in \cite[\S 4]{ps} are special cases of our formula.

We assume that $\A=\B \ten \Ca$ for unital $C^*$-algebras $\B,\Ca$.  Here we use the spatial tensor product. Let $\A_i$ be as before and endow $\B_i :=\A_i \cap \B$ and $\Ca_i:=\A_i \cap \Ca$ with the subspace topology of $\A_i$. Note that $\B_i$ resp. $\Ca_i$ are closed under holomorphic functional calculus in $\B$ resp. $\Ca$. We assume that the projective limit $\B_{\infty}$ resp. $\Ca_{\infty}$ is dense in $\B_i$ resp. $\Ca_i$.  We define $\ideal_i$ to be the closed ideal in $\Oi\A_i$ generated by the supercommutators $[\alpha,\beta]$ with $\alpha \in  \Oi\B_i$ and $\beta \in \Oi\Ca_i$. In the following we deal with $\Oi^{\ideal_i}\A_i$.
We write $\di_1$ for the differential on $\Oi\B_i$ and $\di_2$ for the differential on $\Oi\Ca_i$.

Let $N_1$ be an odd-dimensional and $N_2$ an even-dimensional Riemannian manifold and let $N=N_1 \times N_2$. Furthermore let $E_1$ resp. $E_2$ be Clifford modules over $N_1$ resp. $N_2$. We denote the grading operator on $E_2$ by $\Gamma_2$. In the following  $E_2$ will be considered as an ungraded bundle. Set $E_N=E_1 \boxtimes E_2$.

Furthermore let $P_1 \in \C(N_1,M_m(\B_{\infty}))$ and $P_2 \in \C(N_2,M_n(\Ca_{\infty}))$ be projections and set $P=P_1 \ten P_2 \in \C(N,M_{m+n}(\A_{\infty}))$. We define $\E_i^1=E_1 \ten P_1(N_1 \times \B_i^m),~ \E_i^2=E_2 \ten P_2(N_2 \times \Ca_i^n)$ and $\E_i=E_N \ten P(N \times \A_i^{m+n})$. 
Choose connections $P_1\di_1P_1 + \gamma_1$, $P_2 \di_2 P_2 + \gamma_2$ on $\E_i^1$, $\E_i^2$ in direction of $\B_i$, $\Ca_i$, respectively and Clifford connections $\nabla^1,~ \nabla^2$ on $\E^1:=\E_0^1,~\E^2=\E_0^2$. We get induced connections on $\E_i$.
 
Let $\D_{N_1},\D_{N_2}$ be the Dirac operators associated to $\E^1,\E^2$. Then $$\D_N=\Gamma_2\D_{N_1}+\D_{N_2}$$ 
is the Dirac operator associated to $\E=\E_0$.

If $A$ is a trivializing operator for $\D_{N_1}$, then $\hat A:=\Gamma_2 (A \ten 1)$ is a trivializing operator for $\D_N$. We assume that $A$ resp. $\hat A$ is adapted to $\D_{N_1}$ resp. $\D_N$. (This is the case, for example, if $A$ is an integral operator with integral kernel in $\C(N_1 \times N_1,\Ei^1 \boxtimes_{\Bi} (\Ei^1)^*)$. Here we use that we have divided out the ideal $\ideal_i$.)  
  
\begin{prop}
In $\Oi^{\ideal_{\infty}}\Ai$ modulo the closure of $[\Oi^{\ideal_{\infty}}\Ai,\Oi^{\ideal_{\infty}}\Ai] + (\Oi\Bi) \di (\hat \Omega_{od}\Ca_{\infty})$ it holds that
$$\eta(\D_N,\hat A)= \eta(\D_{N_1},A)\ch(\ind \D_{N_2}^+) \ .$$ 
\end{prop}

The proof is very similar to the proof in \cite[\S 2]{lpetpos}, and is given here for completeness.

There are analogous formulas for other parities of $\dim N_1$ and $\dim N_2$, if $\hat A$ is as defined in \cite{waprod}. See \S \ref{oddind} for the definition of the $\eta$-form in the even-dimensional case. 

Here $\eta(\D_N,\hat A)$ and $\eta(\D_{N_1},A)$ are defined using the same function $\psi$ (see the end of \S \ref{eta}). Otherwise we have to divide out in addition the closure of $(\di\Oi\Bi)\Oi\Ca_{\infty}$.

\begin{proof}
We set $$\spcb^{N_1}(t)_t:=P_1\di_1 P_1 + \gamma_1 +\sqrt t \sigma \Gamma_2(\D_{N_1}+ \psi(t)A) \ ,$$  
$$\spcb^{N_2}_t:=P_2\di_2 P_2 + \gamma_2 + \sqrt t \sigma \D_{N_2} \ .$$ 

Then $\spc^N(t)_t=\spcb^{N_1}(t)_t+ \spcb^{N_2}_t$ and $\spc^N(t)^2_t=(\spcb^{N_1}(t)_t)^2 + (\spcb^{N_2}_t)^2$. 

Furthermore  $\spcb^{N_1}(t)_t$ and $\spcb^{N_2}_t$ anticommute. Hence 
\begin{align*}
\sigma \frac{d\spc^N(t)_t}{dt} e^{-\spc^N(t)_t^2}&=\sigma \frac{d \spcb^{N_1}(t)_t}{dt}e^{-\spcb^{N_1}(t)_t^2} e^{-(\spcb^{N_2}_t)^2} \\
&\quad + \sigma e^{-\spcb^{N_1}(t)_t^2} \frac{d\spcb^{N_2}_t}{dt}e^{-(\spcb^{N_2}_t)^2} \ .
\end{align*}
The second term on the right hand side anticommutes with $\Gamma_2$. Thus
$$\Tr_{\sigma}\sigma e^{-\spcb^{N_1}(t)_t^2} \frac{d\spcb^{N_2}_t}{dt}e^{-(\spcb^{N_2}_t)^2}=0 \ .$$

Now consider the first term. There are integral operators $a_1,b_1$ on $L^2(N_1,\Ol{\mu}\E^1_i)$ and $a_2,b_2$ on $L^2(N_2,\Ol{\mu}\E^2_i)$ such that $$\sigma \Gamma_2 \frac{d\spcb^{N_1}(t)_t}{dt}e^{-\spcb^{N_1}(t)_t^2}=(a_1\ten 1)+\sigma \Gamma_2 (b_1 \ten 1)$$ and
$$\Gamma_2 e^{-(\spcb^{N_2}_t)^2}=1 \ten a_2+\sigma (1 \ten b_2) \ .$$ Note that $a_1,a_2$ are even, whereas $b_1,b_2$ are odd (with respect to grading on the algebra of noncommutative differential forms). 
It follows from $\Gamma_2[P_2 \di_2 P_2+ \gamma_2,\D_{N_2}]=-[P_2 \di_2 P_2+\gamma_2, \D_{N_2}] \Gamma_2$ by Volterra development that $\Gamma_2 b_2=-b_2 \Gamma_2$. Hence $\Tr(\Gamma_2b_1b_2)=0$.
Thus 
\begin{align*}
\Tr_{\sigma} \sigma \frac{d \spcb^{N_1}(t)_t}{dt}e^{-\spcb^{N_1}(t)_t^2} e^{-(\spcb^{N_2}_t)^2}&=\Tr_{\sigma}\bigl((a_1+\sigma \Gamma_2 b_1)(a_2 +\sigma b_2)\bigr)\\
&=\Tr(a_1a_2) +\Tr(\Gamma_2b_1 b_2)\\
&= \Tr(a_1)\Tr(a_2)\\
&=\Tr_{\sigma}(a_1+\sigma b_1) \Tr_{\sigma}(a_2+\sigma b_2) \ .
\end{align*}
It holds that 
$$a_1+ \sigma b_1=\sigma \frac{d\spcb^{N_1}(t)_t}{dt}e^{-\spcb^{N_1}(t)_t^2} \ .$$

Set $\spc^{N_2}_t:= P_2\di_2 P_2 + \gamma_2+ \sqrt t \D_{N_2}$ and define $e^{-(\spc^{N_2}_t)^2}$ by considering $\E_i^2$ a graded vector bundle with grading operator $\Gamma_2$.
One easily checks that $$\Tr_{\sigma} \Gamma_2 e^{-(\spcb^{N_2}_t)^2}=\Tr_s  e^{-(\spc^{N_2}_t)^2} \ .$$
 By the index theorem and Prop. \ref{varfor} applied to $\D_{N_2}$ we get that  $\Tr_s  e^{-(\spc^{N_2}_t)^2}=\ch(\ind \D_{N_2}^+)$ after dividing out the even forms in $\di (\Oi(\Ca_{\infty})/\ov{[\Oi(\Ca_{\infty}),\Oi(\Ca_{\infty})]})$. 

 This implies the assertion.
\end{proof}

\section{The odd index theorem}
\label{oddind}

In the following we prove an analogue of the odd family index for manifolds with boundary \cite{mp2}. The setting is again as in \S \ref{condi}, but now we assume that $M$ is odd-dimensional and $E$ ungraded. On the cylindric end the Dirac operator is then of the form
$$\dira_{\E}=c(dx_1)(\ra_1-\dira_N) \ .$$
The bundle $\E^N$ is $\bbbz/2$-graded with grading operator $\Gamma_N=ic(dx_1)$.

Let $A$ be an adapted trivialising operator for $\dira_N$. We assume that $A$ is odd with respect to $\Gamma_N$. As in \S \ref{DirHilb} one gets that the closure $\D_{\E}(A)$ of $\dira_{\E}-c(dx_1)\chi A$ is Fredholm. Its index is an element in $K_1(\A)$.
 
Define $\spc^N(t)_t=P\di P + \gamma + \sqrt t (\D_N + \psi(t)A)$. The $\eta$-form is defined as 
$$\eta(\dira_N,A):=\frac{1}{\sqrt{\pi}} \int_0^{\infty}\Tr_s \frac{d\spc^N(t)_t}{dt} e^{-\spc^N(t)_t^2}~dt \in \Oi\Ai/\ov{[\Oi\Ai,\Oi\Ai]} \ .$$

The odd Chern character $\ch_{\Ai}:K_1(\A) \to H^{\ideal_{\infty}}_*(\Ai)$ is determined up to sign by the following commuting diagramm:
$$\xymatrix{
K_1(\A) \ten K_1(C(S^1)) \ar[r]^{\ten}\ar[d]^{\ch_{\Ai} \ten \ch^{S^1}}& K_0(C(S^1,\A))\ar[d]^{\ch_{\Ai}^{S^1}} \\
H_*^{\ideal_{\infty}}(\Ai) \ten H^*(S^1) \ar[r]^{=} & H_*^{\ideal_{\infty}}(\Ai) \ten H^*(S^1) } \ .$$
The diagram determines also $\ch^{S^1}$ up to sign: 

Let $u:S^1 \to U(1)$ be a loop with winding number $1$. We identify $S^1$ with $[0,1]/_{0\sim 1}$ and define on $S^1$ the $C(S^1)$-line bundle $$\Li=[0,1] \times C(S^1)/_{(0,uv) \sim (1,v)} \ .$$ 
We denote by $\D_{\Li}$ the closure of the Dirac operator $-i\ra_x$ on $S^1$ twisted by $\Li$.

The Chern character $\ch^{S^1 \times S^1}:K_0(C(S^1 \times S^1)) \to H^*(S^1 \times S^1)$ applied to $\Li$, considered as a complex vector bundle on $S^1 \times S^1$, fulfills 
$$\int_{S^1\times S^1}\ch^{S^1\times S^1}(\Li)=2 \pi i \ .$$ 
The class $\ind \D_{\Li}$ is a generator of $K_1(C(S^1))$. The above commuting diagram implies that
$$\int_{S^1}\ch^{S^1}(\ind \D_{\Li})=\sqrt{2\pi i} \ .$$ 
The sign of the odd Chern character is fixed by demanding that this equation holds for $\sqrt i:=e^{\pi i/4}$.

\begin{theorem}
\label{indtheorodd}
In $H^{\ideal_{\infty}}_*(\Ai)$ it holds that
$$\ch(\ind \D_{\E}(A))=(2 \pi i)^{-n/2} \int_M \hat A(M)\ch(\E/S) - \eta(\dira_N,A) \ .$$
\end{theorem}

In \cite{mp2} a different normalization of the odd Chern character has been used, leading to a factor $(2 \pi i)^{-(n+1)/2}$ in front of the integral.

\begin{proof}
The idea of the proof is as in \cite{mp2}, namely to reduce the odd index problem to the even index problem via suspension. A different proof, which should also work here, has been given in \cite[Theorem 2.2.18]{bu}.

There are two main steps: a suspension formula for index classes \cite[Prop. 8]{mp2} and a suspension formula for $\eta$-forms \cite[Lemma 6]{mp2}. Probably the proof of the suspension formulas given in \cite{mp2} works also in the present context. Here we derive them from the more general product formulas for Atiyah--Patodi--Singer index classes (see \cite{waprod}) and $\eta$-forms (see \S \ref{prodeta}), respectively. 

On $M\times S^1$ we define the $C(S^1,\A)$-vector bundle
$$\E^{M\times S^1}:=(\E\oplus \E) \boxtimes \Li \ .$$ 

Let $\Gamma_1=\left(\begin{array}{cc} 1 &  0\\ 0 & -1\end{array}\right)$, $\Gamma_2=\left(\begin{array}{cc} 0 & i\\ -i & 0 \end{array}\right)$ act on it. 

The Dirac operator $$\dira_{M \times S^1}:= \Gamma_1\dira_{\E} + \Gamma_2\dira_{\Li}$$
is odd with respect to the grading  $\Gamma_{M \times S^1}=-i\Gamma_1\Gamma_2$. 

Set $\E^{N\times S^1}=(\E^{M\times S^1})^+$. The space $L^2(N \times S^1,\E^{N\times S^1})$ is spanned by sections of the form $(x,x) \ten y$ with $x \in L^2(N,\E|_N)$ and $y \in L^2(S^1,\Li)$. 

Define the Dirac operator $\dira_{N \times S^1}:=\dira_N + \Gamma_N\dira_{\Li}$, which acts on the sections of $\E^{N\times S^1}$. Then
$$\dira_{M \times S^1}=c(dx_1)(\ra_1-\Gamma_{M \times S^1}\dira_{N \times S^1}) \ .$$

As in \cite[\S 4.3]{waprod} we define a trivialising operator $\hat A$ of $\dira_{N \times S^1}$ by 
$$\hat A((x,x) \ten y) = (Ax,Ax) \ten y \ .$$ 
We get a Fredholm operator $\D_{M\times S^1}(\hat A)$, see \S \ref{DirHilb}. 

By the product formula for Atiyah--Patodi--Singer index classes \cite{waprod} 
$$\ind \D_{M\times S^1}(\hat A)^+=\ind(\D_{\E}(A)) \ten \ind(\D_{\Li}) \in K_0(C(S^1,\A)) \ .$$ 

Thus 
$$\sqrt{2 \pi i} ~\ch_{\Ai}(\ind\D_{\E}(A))=\int_{S^1}\ch^{S^1}_{\Ai}(\ind\D_{M\times S^1}(\hat A)^+) \ .$$ 
Furthermore by the product formula for $\eta$-forms, see \S \ref{prodeta},
$$\sqrt{2 \pi i}~\eta(\dira_N,A)=\eta(\dira_{N\times S^1},\hat A) \ .$$

Now the assertion follows from Theorem \ref{indtheor} since 
\begin{align*}
\lefteqn{\int_{S^1}\int_{M\times S^1}\hat A(M\times S^1)\ch_{\Ai}^{M \times S^1}(\E^{M \times S^1}/S_{M\times S^1})\ch^{S^1 \times S^1}(\Li)}\\
 &=\int_M\hat A(M) \ch_{\Ai}^{M}(\E/S)\int_{S^1 \times S^1}\ch^{S^1 \times S^1}(\Li) \\
&=2 \pi i \int_M\hat A(M)\ch_{\Ai}^M(\E/S) \ .
\end{align*}
\end{proof}

\section{Application: $\rho$-invariants}
\label{rhoinvar}

For the sake of brevity we assume in the following that the manifold $N$ is odd-dimensional. The even-dimensional case can be treated analogously.

\subsection{$\rho$-invariants for the Dirac operator on manifolds with positive scalar curvature}

Let $N$ be a closed spin manifold with positive scalar curvature. Let $\pi:\tilde N \to N$ be a Galois covering of $N$ and let $\Gamma$ be the group of deck transformations.

The higher $\rho$-form of $N$ was introduced by Lott \cite{lo3}. In the higher context the $C^*$-algebra involved is the reduced group $C^*$-algebra $C^*_r\Gamma$. Closely related are the classical $\rho$-invariants, which however in general come from a representation of the maximal group $C^*$-algebra $C^*\Gamma$, see \cite{ps} for a detailed account. The following discussion unifies these concepts and allows for the definition and study of new $\rho$-invariants.

In the following let $\A=C^*\Gamma$. We assume that the projective system $(\A_i)_{i \in \bbbn_0}$ is such that $\bbbc\Gamma\subset \Ai$. 
For the reduced group $C^*$-algebra $\B=C_r^*\Gamma$ a suitable projective system of algebras $(\B_i)_{i\in \bbbn_0}$ can be derived from a construction of Connes--Moscovici \cite{cm}, see for example \cite[\S 4]{wa2} for a detailed account. There is a canonical surjection $p:C^*\Gamma \to C^*_r\Gamma$. It is straightforward to check that  $\A_i:=p^{-1}\B_i$ with norm $\|a\|_{\A_i}:=\|a\|_{\A} + \|p(a)\|_{\B_i}$ is a Banach algebra closed under holomorphic functional calculus in $C^*\Gamma$ and hence $(\A_i)_{i \in \bbbn_0}$ is an appropriate projective system.

Denote by $\Pj=\tilde N \times_{\Gamma} C^*\Gamma$ the Mishenko--Fomenko $C^*\Gamma$-vector bundle endowed with the $C^*\Gamma$-valued scalar product induced by the standard $C^*\Gamma$-valued scalar product on $C^*\Gamma$. Let $\nabla^{\Pj}$ be the flat connection on $\Pj$ induced by the de Rham differential on $\tilde N$ and let $\dirac_N$ be the spin Dirac operator on $N$ twisted by $\Pj$. 

We recall the construction of an embedding of $\Pj$ into a trivial $C^*\Gamma$-vector bundle:
Choose a finite connected open cover $(U_k)_{k=1, \dots, n}$ of $N$ such that for any $k$ there is a cross-section $\psi_k:U_k \to \pi^{-1}U_k$. We write $\tilde U_k=\psi_k(U_k)$. Let $g_{kl}: U_k \cap U_l \to \Gamma$ be the locally constant function whose image at a point $x$ is the deck transformation sending $\tilde U_l\cap \pi^{-1}(x)$ to $\tilde U_k \cap \pi^{-1}(x)$. Note that $g_{kl}g_{lm}=g_{km}$ (whenever the left hand side is defined) and $g_{kk}=1$. Let $(\chi_k^2)_{k=1, \dots n}$ be a partition of unity subordinate to $(U_k)$. Define the projection $P \in \C(M,M_n(\bbbc \Gamma))$ by setting $P_{kl}=\chi_k\chi_l g_{kl}$. Then $\Pj$ is isometrically isomorphic to $P(N \times (C^*\Gamma)^n)$ via 
$$\C(\tilde N,C^*\Gamma)^{\Gamma} \to P(\C(N,(C^*\Gamma)^n)),~s\mapsto (\chi_1 (s\circ \psi_1), \chi_2 (s\circ \psi_2), \dots, \chi_n (s\circ \psi_n)) \ .$$ 
One checks that $\nabla^{\Pj}=P d_NP$. 

Note that $\dirac_N$ is invertible. Thus $\eta(\dirac_N,0)$, which we define using the Grassmannian connection $P\di P$ in direction of $\A_i$, exists.

In the following we deal with the universal differential algebra (thus $\ideal_i=0$).

Let $c:(\Ai/\bbbc)^{\ten_{\pi}^{m+1}} \to V$ be a continuous reduced cyclic cocycle on $\Ai$ with values in a Fr\'echet space $V$ such that $c(g_0,\dots, g_m)=0$ for all $(g_0, \dots, g_m) \in \Gamma^{m+1}$ with $g_0g_1 \dots g_m=1$. There is an induced linear map $c:\Oi\Ai/\ov{[\Oi\Ai,\Oi\Ai]} \to V$ defined by $c(a_0da_1 \dots da_m):=c(a_0,a_1, \dots,a_m)$ and by zero on forms of degree not equal to $m$. The map vanishes on exact forms.
We define the maximal $\rho$-invariant associated to $c$ as 
$$\rho^{max}_c(\dirac_N):=c(\eta(\dirac_N,0)) \in V .$$

\begin{lem}
The $\rho$-invariant $\rho^{max}_c(\dirac_N)$ is independent of the choice of the cover and the partition of unity.
\end{lem}

\begin{proof}
Let $(U_k')_{k=1, \dots, p}$ be a second open cover of $N$ and let $P'=(\chi'_{k}g_{kl}'\chi'_l)_{k,l=1,\dots, p}$ be a projection constructed as above. Define $\tilde U_k'$ as above. For $k=1, \dots, p$ and $l=1, \dots, n$ let $h_{kl}: U_k' \cap U_l \to \Gamma$ be the locally constant function such that $h_{kl}(x)$ is the decktransformation mapping $\tilde U_l\cap \pi^{-1}(x)$ to $\tilde U_k' \cap \pi^{-1}(x)$. Then $U=(\chi'_k h_{kl} \chi_l)_{k=1, \dots, p; l=1 \dots n}$ is a partial isometry with $UU^*=P'$ and $U^*U=P$. Thus $U^* P'\di P'U=P\di P + U^*\di(U)$. Let $\phi: \bbbr \to \bbbr$ be a smooth function that vanishes for  $x < 1/6$ and equals $1$ for $x>5/6$. Set $Z=\bbbr \times N$. Define $\gamma(x_1,x_2)=\phi(x_1)U^*\di (U)(x_2)$ for $(x_1,x_2) \in Z$. By the Atiyah--Patodi--Singer index theorem applied to the Dirac operator on $Z$ twisted by $\Pj^Z=P(Z\times \A^n)$ the difference of the $\rho$-invariants defined using $P$ and $P'$ equals 
$C\int_Z \hat A(Z)c(\ch(Pd_ZP + P\di P+\gamma))$. Using that $g_{k_1k_2}(h_{k_3k_2})^{-1}g_{k_3k_4}'h_{k_4k_5}g_{k_5k_6}$ equals $g_{k_1k_6}$ on its domain, one checks that $c(\ch(Pd_ZP + P\di P+\gamma))$ vanishes. 
\end{proof}

From the Atiyah--Patodi--Singer index theorem we get the following generalization of \cite[Theorem 15.1]{lp1}:

\begin{prop}
Let $M$ be a compact spin manifold with boundary $N$ and product structure near the boundary. Let $\tilde M$ be a Galois covering of $M$ with deck transformation group $\Gamma$. Assume that the metric on $M$ and its restriction to $N$ are of positive scalar curvature. Then $\rho^{max}_c(\dirac_N)=0$.
\end{prop}

By using the Atiyah--Patodi--Singer index theorem and the product formula for $\eta$-forms of the previous sections instead of the corresponding higher versions of Leichtnam--Piazza, the proof of \cite[Prop. 13.9]{ps} works also in the present setting, yielding:

\begin{prop}
Let $\Gamma$ be torsion-free. Assume that the assembly map $\mu:K_*(B\Gamma) \to K_*(C^*\Gamma)$ is an isomorphism. Then $\rho^{max}_c(\dirac_N)=0$.
\end{prop}

{\bf Example:} Let $\pi_1,\pi_2:\Gamma \to U(l)$ be two unitary representations and let $F_1,F_2$ be the associated flat vector bundles on $N$. For $j=1,2$ let $\tilde \eta_{F_j}(N) \in \Oi\Ai/\ov{[\Oi\Ai,\Oi\Ai]}$ be the noncommutative $\eta$-form of the spin Dirac operator on $N$ twisted by the flat $C^*\Gamma$-bundle $F_j\ten \Pj$. We define the maximal Atiyah--Patodi--Singer $\rho$-form
$$\rho_{\pi_1,\pi_2}(N):=\tilde \eta_{F_1}(N)-\tilde \eta_{F_2}(N) \in \Oi\Ai/\ov{[\Oi\Ai,\Oi\Ai]} \ .$$
Let $b:(\Ai/\bbbc)^{\ten_{\pi}^{m+1}} \to V$ be an arbitrary continuous reduced cyclic cocycle and define the reduced cyclic cocycle
$$c=b \circ \Tr\circ (\pi_1\ten id)_*- b \circ \Tr\circ (\pi_2\ten id)_*:(\Ai/\bbbc)^{\ten_{\pi}^{m+1}} \to V \ .$$ 
Here $(\pi_j\ten id)_*:\Oi\Ai/\ov{[\Oi\Ai,\Oi\Ai]} \to M_l(\Oi\Ai/\ov{[\Oi\Ai,\Oi\Ai]})$ is induced by the homomorphism $\pi_j\ten id:C^*\Gamma \to C^*(U(l)\ten \Gamma)=M_l(\bbbc) \ten C^*\Gamma$.  
Then $$b(\rho_{\pi_1,\pi_2}(N))=\rho^{max}_c(\dirac_N) \ .$$ 
Thus the previous results apply to $b(\rho_{\pi_1,\pi_2}(N))$.

In the case where $\Gamma$ is a group with torsion, $\rho$-invariants can often be used to distinguish $\Gamma$-bordism classes of metrics of positive scalar curvature (see p.\,e. \cite{lpetpos}, also for the terminology, and \cite{pstor}). Higher $\rho$-invariants are useful for the study of Cartesian product. We give an elementary example, based on results and ideas from \cite{lpetpos}\cite{pstor}. 

\begin{prop}
Let $N$ be a closed spin manifold of dimension $4k+3,~ k \in \bbbn,$ admitting a metric of positive scalar curvature. Assume that the fundamental group $\Gamma$ of $N$ has torsion. Then, for any $m \in \bbbn,$ the manifold $N\times T^m$ has infinitely many $\Gamma\times \bbbz^m$-bordism classes of metrics of positive scalar curvature.
\end{prop}
 
\begin{proof}
Let $\tau_e: C^*\Gamma \to \bbbc$ be the trace defined by $\tau_e(g)=\delta_{g,e}$ for $g \in \Gamma$, and let $\tau_1$ be the trace induced by the trivial representation of $\Gamma$.
Set $\tau=\tau_e- \tau_1$. In the situation of \S \ref{prodeta} we have that $\Ca=C^*(\bbbz)$ and $\Ca_i=C^i(T^m)$. It is well-known that there is a cyclic cocycle $c$ on $\C(T^m)$ such that $\int_{T^m} \hat A(T^m) c(\ch \Pj^{T^m}) \neq 0$.

Since $\rho_{\tau}^{max}(\dira_{N})$ is the $L^2$-$\rho$-invariant, by \cite{pstor} there are infinitely many $\Gamma$-bordism classes of metrics of positive scalar curvature on $N$ distinguished by $\rho_{\tau}^{max}(\dira_{N})$. 

For each bordism class choose a representative (i.\,e. a metric of positive scalar curvature on $N$) such that the product metric on $N\times T^m$ is of positive scalar curvature. The $\rho$-invariants $\rho_{\tau \sharp c}^{max}(\dira_{N \times T^m})$ fulfill
$$\rho_{\tau \sharp c}^{max}(\dira_{N \times T^m})= (2 \pi i)^{-m/2} \rho_{\tau}^{max}(\dira_{N})\int_{T^m} \hat A(T^m) c(\ch \Pj^{T^m}) \ ,$$ and thus distinguish the resulting metrics of positive scalar curvature on $N \times T^m$ up to $\Gamma\times \bbbz^m$-bordism.
\end{proof}

\subsection{$\rho$-invariants for the signature operator}

In a similar spirit we study $\rho$-invariants for the signature operator. Our definition is motivated by the definition of higher $\rho$-invariants suggested in \cite{lo3}. It is a straightforward generalization of the definition presented in \cite{lpsign} for groups of polynomial growth. Note that \cite[Prop. 5.1]{lpsign} need not hold in the present situation.

Thus now $N$ is a closed oriented Riemannian manifold of dimension $2m-1$.

Denote by $\Omega^*(N,\Pj)$ the space of smooth forms with values in $\Pj$ and by $\Omega_{(2)}^*(N,\Pj)$ the completion as a Hilbert $C^*\Gamma$-module with respect to the $L^2$-norm. 

Set $d^{sig}:=d\tau + \tau d$. Here $\tau$ is the chirality operator, which agrees with the Hodge star operator in the normalisation of \cite[Def. 3.57]{bgv}. 
 
We assume that the closure of $d:\Omega^{m-1}(N,\Pj) \to \Omega_{(2)}^{m}(N,\Pj)$ has closed range. (This assumption is independent of the choice of the Riemannian metric on $N$.)

It follows that there is a canonical trivializing operator $A_N$, see \cite{waprod}. (In fact, one can use any of the symmetric boundary conditions introduced in \cite{lp5}.) For a reduced continuous cyclic cocycle $c$ as above define $$\rho^{max}_c(N):=c(\eta(d^{sig},A_N)) \ .$$  We leave it to the reader to show that the Atiyah--Patodi--Singer $\rho$-invariant and the $L^2$-$\rho$-invariant of Cheeger--Gromov are special cases (under our assumption).

The Atiyah--Patodi--Singer index theorem for the cylinder implies that  $\rho^{max}_c(N)$ does not depend on the Riemannian metric since the twisted signature class $\sigma(Z,\Pj^Z) \in K_0(C^*\Gamma)$ (see \cite{waprod} for the terminology) vanishes. 

Theorem \ref{indtheor} implies:

\begin{prop}
Assume that $N$ is the boundary of a manifold $M$ and let $\sigma(M,\Pj^M) \in K_0(C^*\Gamma)$ be the twisted signature class. If $c(\ch(\sigma(M,\Pj^M)))=0$, then $\rho_c^{max}(N)=0$.
\end{prop} 

From the product formula for $\eta$-forms and the Atiyah--Patodi--Singer index theorem one gets in analogy to \cite[Theorem 7.1]{ps}:

\begin{prop}
Let $\Gamma$ be torsion-free. Assume that the assembly map $\mu:K_*(B\Gamma) \to K_*(C^*\Gamma)$ is an isomorphism. Then $\rho^{max}_c(N)=0$.
\end{prop}

It would be interesting to find a connection to the $L$-theoretic higher $\rho$-invariants introduced in \cite{wb}.

\section{Appendix: Operators on Banach spaces}

\subsection{Projective systems and resolvent sets}
\label{parametrix}

In the following let $M$ be a $\sigma$-finite measure space such that $L^2(M)$ is separable. Choose an orthonormal basis of $L^2(M)$ and let $P_m$ the projection onto the span of the first $m$ vectors. We use the definitions from \S \ref{prosys}.
 
\begin{prop}
\label{invspec}
Let $k \in L^2(M \times M) \ten_{\pi} M_n(\A_i)$ and let $K$ be the corresponding Hilbert Schmidt operator on $L^2(M,\Ol{\mu}\A_i^n)$. 
If $1-K$ is invertible in $B(L^2(M, \A^n))$, then it is invertible in $B(L^2(M, \Ol{\mu}\A_i^n))$ as well.
\end{prop}

\begin{proof} 
We adapt a method from \cite[\S 6]{lo2}.

It is enough to find a decomposition $L^2(M, \Ol{\mu}\A_i^n)=X \oplus Y$ with the following property: If we write
$$1-K= \left(\begin{array}{cc} a & b\\ c& d \end{array}\right)$$  
with respect to the decomposition, then $d$ is invertible and $a-bd^{-1}c$ is invertible. Then the assertion follows from
$$\left(\begin{array}{cc} a & b\\
c& d
\end{array} \right)
=\left(\begin{array}{cc} 1 & bd^{-1}\\
0 & 1
\end{array} \right)
\left(\begin{array}{cc} a-bd^{-1}c & 0\\
0& d
\end{array} \right)
\left(\begin{array}{cc} 1 & 0\\
d^{-1}c & 1
\end{array} \right) \ .$$

Acting with respect to one variable on $L^2(M \times M) \ten_{\pi} M_n(\Ai)$, the operator $P_m$ converges strongly to the identity for $m \to \infty$. Hence for $m$ big enough $\|(1-P_m)K(1-P_m)\|< \frac 12$ in $B(L^2(M, \A^n))$ and in $B(L^2(M, \Ol{\mu}\A_i^n))$. Then $d$, defined as above with respect to the decomposition $$L^2(M, \Ol{\mu}\A_i^n)=P_mL^2(M, \Ol{\mu}\A_i^n)\oplus (1-P_m)L^2(M, \Ol{\mu}\A_i^n)  \ ,$$
is invertible.

The assumption implies that $a-bd^{-1}c$ is invertible on $P_m(L^2(M, \A^n))$. Since this is a finitely generated free $\A$-module of rank $mn$ and $M_{mn}(\A_i)$ is closed under holomorphic functional calculus in $M_{mn}(\A)$, the map $a-bd^{-1}c$ is invertible as well on $P_mL^2(M, \Ol{\mu}\A_i^n)$.
\end{proof}

Note that if $M$ is a closed manifold, then $\C(M \times M,M_n(\A_i)) \subset L^2(M \times M) \ten_{\pi} M_n(\A_i)$. If $M$ is a manifold with cylindric ends, then ${\mathcal S}(M \times M,M_n(\A_i)) \subset L^2(M \times M) \ten_{\pi} M_n(\A_i)$.

\begin{prop}
\label{proj}
Let $P \in B(L^2(M, \A^n))$ be a projection onto a projective submodule of
 $L^2(M, \A^n)$. Then $P$ is a Hilbert-Schmidt operator of the form $P=\sum_{j=1}^k f_j h_j^*$ with $f_j,h_j \in \Ran P$.
\end{prop}
 
\begin{proof}
Since $P$ is compact, there is $m \in \bbbn$ such that $\|P(P_m-1)\| \le \tfrac 12$.

Then (see \cite[Prop. 5.1.21]{wa})  the map 
$PP_mP:\Ran P \to \Ran P$ is an isomorphism.

It follows that
$\Ker PP_mP=\Ker P$ and
therefore 
$$P =1- P_{\Ker P P_mP} \ ,$$
where  $P_{\Ker P P_mP}$ denotes the projection onto
 $\Ker P P_mP$.

We can find $r>0$ such
that 
$B_r(0)\setminus \{0\}$ is in the resolvent set of $PP_mP$.

The assertion follows from
\begin{eqnarray*}
P &=& 1- P_{\Ker P P_mP} \\
&=& 1- \frac{1}{ 2 \pi i }\int_{|\lambda|=r} (\lambda-PP_mP)^{-1} d \lambda \\
&=& \frac{1}{2 \pi i} \int_{|\lambda|=r} \left(\lambda^{-1}-  (\lambda-PP_mP)^{-1} \right) d \lambda\\
&=& -\frac{1}{2 \pi i}  PP_mP \int_{|\lambda|=r}\lambda^{-1}(\lambda -PP_mP)^{-1}
d \lambda \ .
\end{eqnarray*}
\end{proof}
 
Recall from \cite[\S 5.2]{wa} that there is a well-behaved notion of trace class operators   on $L^2(M,\A_i^n)$ if (for example) $M$ is a complete Riemannian manifold.

\begin{prop}
\label{projker}
Let $P \in B(L^2(M, \A^n))$ be a projection onto a projective submodule of
 $L^2(M, \A^n)$. Assume further that for any $i \in \bbbn$ it restricts to a bounded
 projection on $L^2(M, \A_i^n)$ and that $P(L^2(M,\bbbc^n)) \subset
 L^2(M) \ten_{\pi} \A_i^n$. Let $$\Ran_{\infty} P:=\bigcap\limits_{i \in \bbbn}
 P(L^2(M,\A_i^n))   \ .$$ 

Then:
\begin{enumerate}
\item The projection $P$ is a Hilbert-Schmidt
operator of the form $P=\sum_{j=1}^k f_j h_j^*$ with $f_j,h_j \in \Ran_{\infty} P \cap (L^2(M) \ten_{\pi} \A_i^n)$.
\item The intersection $\Ran_{\infty} P$ is a projective
$\Ai$-module. The classes $[\Ran P] \in K_0(\A)$ and $[\Ran_{\infty} P] \in K_0(\Ai)$
correspond to each other under the canonical isomorphism $K_0(\A) \cong
K_0(\Ai)$. 
\end{enumerate}
Now assume that $M$ is a complete Riemannian manifold and let $\di + \gamma$ with $\gamma \in \C(M,M_n(\hat\Omega_1 \Ai))$ be a connection of $M\times \Ai^n$ in direction of $\Ai$ as in \S \ref{connect}.

Then 
$$\ch[\Ran_{\infty} P]= \sum\limits_{j=0}^{\infty} \frac{(-1)^j}{j!} \Tr \left(P (\di+\gamma) P (\di+\gamma) P\right)^j \in H_*^{\ideal_{\infty}}(\Ai) \ .$$
\end{prop}

\begin{proof}
For $\gamma=0$, the proof is analogous to the proof of \cite[Prop. 5.3.6]{wa}. Then the assertion follows also for $\gamma \neq 0$ by an analogue of the transgression formula in \S \ref{connect} applied to the connection $P\di P + P\gamma P$.
\end{proof}

\subsection{Parameter dependent parametrices}
\label{regpar}

The following results generalize results from \cite[\S 6]{lo2}.

Let $M,\E$ be as in \S \ref{condi}.

First we make a general observation: If $B$ is a bounded operator on the Hilbert $\A$-module $L^2(M,\E)$, then by restriction and $\Ol{\mu}\A_i$-linear extension one gets an unbounded (not necessarily densely defined) operator on $L^2(M,\Ol{\mu}\E_i)$.

We will also use that there is a continuous map from ${\mathcal S}(M \times M,\E_i\boxtimes_{\A_i} (\E_i)^*)$ to the Banach algebra of bounded operators on $L^2(M,\Ol{\mu}\E_i)$ mapping integral kernels to the corresponding integral operators. The map is also continuous if considered as a a map into the space of bounded operators from $L^2(M,\Ol{\mu}\E_i)$ to ${\mathcal S}(M,\Ol{\mu}\E_i)$ endowed with the strong operator topology. 

Let $I$ be an interval. Let $(D(r))_{r \in I}$ be a family of regular operators on the Hilbert $\A$-module $L^2(M,\E)$. We assume that the following conditions hold:

\begin{enumerate}
\item For each $r \in I$ the space ${\mathcal S}(M, \E_i)$ is a core for $D(r)$. The family $(D(r))_{r \in I}$ defines a bounded operator on the space $C^1(I,{\mathcal S}(M,\E_i))$ and extends to a bounded operator on $C^1(I,{\mathcal S}(M,\Ol{\mu}\E_i))$. 
\item Analogous properties hold for $D(r)^*$ as well.
\end{enumerate}

The closure of $D(r)$ on $L^2(M,\Ol{\mu}\E_i)$ is denoted by $D(r)$ as well.

\begin{ddd}
Let $(Q(r))_{r\in I}$ be a family of bounded operators on $L^2(M,\E)$, which is of class $C^1$ (with respect to the operator norm). Assume that it also defines a family of bounded operator of class $C^1$ on $L^2(M,\Ol{\mu}\E_i)$ and induces a bounded operator on $C^1(I,{\mathcal S}(M,\Ol{\mu}\E_i))$.  

We call $(Q(r))_{r \in I}$ a {\rm parameter dependent regular left parametrix} for $(D(r))_{r \in I}$ if $1-Q(r)D(r)$ is an integral operator whose  integral kernel is in $C^1(I,{\mathcal S}(M \times M, \E_i\boxtimes_{\A_i} \E_i^*))$. 
\end{ddd}

If $M$ is closed, the main examples come from parameter dependent elliptic pseudodifferential operators, and for general $M$ from appropriate generalizations.

\begin{prop}
\label{boundinv}
Assume that $(D(r))_{r\in I}$ has a parameter dependent regular left parametrix $(Q(r))_{r \in I}$ and that $D(r)$ has an inverse in  $B(L^2(M,\E))$ for each $r$, which depends continuously on $r$. Furthermore we assume that the adjoint $(D(r)^*)_{r \in I}$ has a parameter dependent regular left parametrix as well.

Then it holds:
\begin{enumerate}
\item The operator $D(r)$ has a bounded inverse on $L^2(M,\Ol{\mu}\E_i)$, which is of class $C^1$ in $r$.

\item The operator $D(r)^{-1}-Q(r)$ is an integral operator with integral kernel in $C^1(I,{\mathcal S}(M \times M,\E_i\boxtimes_{\A_i} \E_i^*))$.

\item In particular $D(r)^{-1}$ induces a bounded operator on $C^1(I,{\mathcal S}(M,\Ol{\mu}\E_i))$.  
\end{enumerate}
\end{prop}

\begin{proof}
Note that $(D(r))_{r \in I}$ defines a regular invertible operator on the Hilbert $C(I,\A)$-module $C(I,L^2(M,\E))$, which for brevity we also denote by $D(r)$ in the following. 

First we modify $Q(r)$ such that $Q(r)D(r)$ is invertible for each $r$:
Since $1-Q(r)D(r)$ is an integral operator with integral kernel in $C(I,{\mathcal S}(M \times M,\E_i\boxtimes_{\A_i} \E_i^*))$ and $D(r)^*$ is invertible on the Hilbert $C(I,\A)$-module $C(I,L^2(M,\E))$, there is an integral operator $S(r)$ with integral kernel in $C^1(I,{\mathcal S}(M \times M,\E_i\boxtimes_{\A_i}\E_i^*))$ such that in $B(C(I,L^2(M,\E))$
$$\|(1-Q(r)D(r))- S(r)D(r)\| \le \frac 12 \ .$$
Then $Q'(r):=Q(r)+S(r)$ is a regular left parametrix of $D(r)$ as well. Set $K(r):=1-Q'(r)D(r)$. By the estimate $1-K(r)$ has an inverse in $B(L^2(M,\E))$ for each $r\in I$. By Prop. \ref{invspec} this implies that the operator $1-K(r)$ is invertible on $L^2(M,\Ol{\mu}\E_i)$ for each $r$. Since $1-K(r)$ is of class $C^1$ as a bounded operator on $L^2(M,\Ol{\mu}\E_i)$, its inverse is also of class $C^1$ in $r$. Hence $D(r)^{-1}=(1-K(r))^{-1}Q'(r)$ is a bounded operator on $L^2(M,\Ol{\mu}\E_i)$ that is of class $C^1$ in $r$. From 
$$D(r)^{-1}-Q(r)=S(r) + K(r)D(r)^{-1}$$ 
and the fact that $K(r)$ maps  $C^1(I,L^2(M, \Ol{\mu}\E_i))$ continuously to $C^1(I,{\mathcal S}(M, \Ol{\mu}\E_i))$ one gets (3). By applying (3) to the adjoint $D(r)^*$ we conclude that $K(r)D(r)^{-1}$ is an integral operator with integral kernel in $C^1(I,{\mathcal S}(M \times M,\E_i\boxtimes_{\A_i} \E_i^*))$. Then the previous equation implies (2).
\end{proof}

\subsection{Holomorphic semigroups}

In the following we collect some facts about holomorphic semigroups for convenience of the reader. Definitions and proofs can be found in \cite{da} (note however that we do not assume a holomorphic semigroup to be bounded), see also \cite[\S 5.4]{wa} for more details as needed here. 

Let $V$ be a Banach space.

\begin{prop}
\label{exthol}
Let $Z$ be the generator of a strongly continuous semigroup such that $\Ran e^{tZ} \subset \dom Z$ for all $t>0$. The semigroup extends to a holomorphic one if and only if there is $C>0$ such that for $t \neq 0$ small
$$\|Ze^{tZ}\| \le Ct^{-1} \ .$$
\end{prop}

In the following let $Z$ be a generator of a holomorphic semigroup $e^{tZ}$ on $V$. 

The following two propositions describe the connection between spectrum of the generator and the behaviour of the semigroup for  $t \to \infty$:

\begin{prop}
\label{semest}
Assume  that there is $\delta >0$ such that the resolvent $(Z-\lambda)^{-1}$ exists for all $\lambda \in \bbbc$ with $\re \lambda > -\delta$.

Then the semigroup $e^{tZ}$ is bounded.
\end{prop}

\begin{prop}
\label{semires}
Let $\omega \in \bbbr$ be such that for all $\ve>0$ there is $C>0$ with $\|e^{tZ}\| \le Ce^{-(\omega- \ve) t}$ for all $t\ge 0$.
If $\re \lambda >- \omega$, then for $k \in \bbbn$
$$(\lambda- Z)^{-k}=\frac{1}{(k-1)!}\int_0^{\infty} t^{k-1}e^{t(Z-\lambda)} ~dt \ .$$
\end{prop}

The spectrum of the semigroup itself and its behaviour for  $t \to \infty$ are also related to each other:

\begin{prop}
\label{specsem}
Let $r\in \bbbr$ be such that the spectral radius of $e^{tZ}$ is smaller than or equal to $e^{rt}$ for all $t >0$. Then for all $\ve>0$ there is $C\ge 1$ such that
$$\|e^{tZ}\| \le Ce^{(r+\ve)t} \ .$$
\end{prop}

The following proposition addresses the question when a perturbation of a generator is again a generator.

\begin{prop}
\label{relpert}
Let $B$ be a closed operator on $V$ such that $Be^{tZ}$ is well-defined and bounded for all $t>0$ and such that there is $C>0$ with
$$\|Be^{tZ}\| \le Ct^{-1/2}$$
for $0<t<1$.

Then $Z+B$ generates a holomorphic semigroup.
\end{prop}

We recall Duhamel's principle:

\begin{prop}
Let $Z$ be the generator of a strongly continuous semigroup on $V$ and let $u \in C^1([0,\infty),V)$ be such that $u(t) \in \dom Z$ and $(\frac{d}{dt}-Z)u(t) \in \dom Z$ for all $t\ge 0$. Then
$$e^{tZ}u(0)-u(t)=-\int_0^t e^{sZ}(\frac{d}{dt} - Z)u(t-s)~ds \ .$$
\end{prop} 

From Duhamel's principle we derive Duhamel's formula as needed here:

\begin{prop}
\label{duhform}
Let $Z$ be the generator of a holomorphic semigroup on $V$. Let $\ve>0$ and let $(B(r))_{r \in [-\ve,\ve]}$ be a family of closed operators with $B(0)=0$ such that $B(r)e^{tZ}$ is well-defined and bounded for all $t>0$ and $r \in [-\ve,\ve]$ and depends continuously on $r$. We assume that $\frac{d}{dr} (B(r) e^{tZ})|_{r=0}$ exists for all $t>0$ and that there is $C>0$ such that for $r \in [-\ve,\ve] \setminus \{0\}$ and $0<t<1$  
$$\|\frac 1rB(r)e^{tZ}\| \le Ct^{-1/2} \ .$$

Furthermore we assume that the set of all $f \in \dom Z$ such that $B(r)e^{tZ}f \in \dom Z$ for all $t \ge 0$ and $r \in [-\ve,\ve]$ is dense in $V$.  

Then $$\frac{d}{dr}e^{t(Z+B(r))}|_{r=0}=\int_0^te^{sZ}\frac{d}{dr}B(r)e^{(t-s)Z}|_{r=0}~ds \ .$$
\end{prop}

Note that the right hand side is well-defined as a bounded operator, whereas the existence of the left hand side is part of the assertion.

\begin{proof}
We write $B$ for the operator induced by the family $(B(r))_{r \in [-\ve,\ve]}$ on $C([-\ve, \ve],V)$ and consider $Z$ also as an operator on $C([-\ve,\ve],V)$ in the obvious way. Thus $\dom Z \subset C([-\ve,\ve],V)$ in the following. Note that $\dom (Z+B)=\dom Z$. By the estimate and the previous proposition $Z+B$ generates a holomorphic semigroup on $C([-\ve,\ve],V)$.

If $f \in C([-\ve,\ve],V)$ such that $f \in \dom Z$ and $Be^{tZ}f \in \dom Z$ for all $t \ge 0$, then $e^{tZ}f \in \dom Z$ and $(\frac{d}{dt}-Z-B)e^{tZ}f=-Be^{tZ}f \in \dom Z$ for all $t\ge 0$. Duhamel's principle implies that
$$e^{t(Z+B)}f-e^{tZ}f=\int_0^te^{s(Z+B)}Be^{(t-s)Z}f~ds \ .$$
Using Lebesgue Lemma we get that
\begin{align*}
\lim_{r \to 0}\frac 1r (e^{t(Z+B(r))}-e^{tZ}))&=\lim_{r\to 0} \frac 1r \int_0^te^{s(Z+B(r))}B(r)e^{(t-s)Z}~ds\\
&=\int_0^te^{sZ}\frac{d}{dr} B(r)e^{(t-s)Z}|_{r=0}~ds \ .
\end{align*}
This proves the assertion.

\end{proof}

\textsc{Leibniz-Archiv\\Waterloostr. 8 \\30169 Hannover \\ Germany} 

\textsc{wahlcharlotte@googlemail.com}

\end{document}